\documentclass[10pt, a4paper]{article}

\usepackage[utf8]{inputenc}
\usepackage[T1]{fontenc}
\usepackage{lmodern}
\usepackage{amsmath, amssymb, amsthm}
\usepackage[inline, shortlabels]{enumitem}
\usepackage{vmargin}

\usepackage[UKenglish]{babel}

\newcommand{\rk}{\operatorname{rk}}

\newcommand{\qoplus}{(\mkern-4mu+\mkern-4mu)}
\newcommand{\qtimes}{(\mkern-4mu\times\mkern-4mu)}
\newcommand{\bigast}{\ast} 

\newcommand{\cF}{\mathcal{F}}

\newcommand{\up}{\underline{\smash{p}}}

\newcommand{\C}{\mathbb{C}}
\newcommand{\F}{\mathbb{F}}
\newcommand{\K}{\mathbb{K}}
\newcommand{\T}{\mathbb{T}}
\newcommand{\Z}{\mathbb{Z}}

\newcommand{\SL}{\operatorname{SL}}
\newcommand{\PSL}{\operatorname{PSL}}

\newcommand{\ad}{\operatorname{ad}}
\newcommand{\im}{\operatorname{im}}
\newcommand{\Tr}{\operatorname{Tr}}
\newcommand{\End}{\operatorname{End}}

\newtheorem{fact}{Fact}
\newtheorem*{fact*}{Fact}
\newtheorem{lemma}{Lemma}
\newtheorem{proposition}{Proposition}

\newtheorem*{definition*}{Definition}
\newtheorem*{theorem*}{Theorem}
\newtheorem*{question*}{Question}

\theoremstyle{definition}
\newtheorem*{remark*}{Remark}
\newtheorem*{remarks*}{Remarks}
\newtheorem*{notation*}{Notation}

\swapnumbers
\newtheorem{step}{-- Step}[proposition]
\newtheorem{notationinproof}[step]{-- Notation}

\newenvironment{proofclaim}{\begin{proof}[Proof of Step \arabic{step}]\renewcommand{\qedsymbol}{$\Diamond$}}{\end{proof}}

\makeatletter
\def\@fnsymbol#1{\ensuremath{\ifcase#1\or \or *\or \dagger\or \ddagger\or
   \mathsection\or \mathparagraph\or \|\or **\or \dagger\dagger
   \or \ddagger\ddagger \else\@ctrerr\fi}}
\makeatother

\title{Involutive automorphisms of $N_\circ^\circ$-groups of finite Morley rank}
\author{Adrien Deloro\footnote{Institut de Mathématiques de Jussieu --- Paris Rive Gauche; {adrien.deloro@imj-prg.fr}}\ \ and \'Eric Jaligot\footnote{1972-2013}}
\date{}
\begin{document}

\maketitle

\begin{quote}
\itshape
\setlength\leftskip{7cm}
Ma Virgilio n'avea lasciati scemi\\
di sé, Virgilio dolcissimo patre,\\
Virgilio a cui per mia salute die' mi.
\end{quote}
\medskip

\begin{quote}
\textbf{Abstract.} We classify a large class of small groups of finite Morley rank: $N_\circ^\circ$-groups which are the infinite analogues of Thompson's $N$-groups. More precisely, we constrain the $2$-structure of groups of finite Morley rank containing a definable, normal, non-soluble, $N_\circ^\circ$-subgroup.
\end{quote}

{\footnotesize
\tableofcontents
}

\section{Introduction}

This is the final item in the series \cite{DJGroups, DJSmall, DJLie}, a collaboration interrupted by the demise of Jaligot. The present article has a sad story but at least it has the merit to exist: it was started in 2007 with hope and then never completed, started again in 2013 as a brave last sally
and then lost, and then started over again by the first endorser alone.

So for the last time let us deal with $N_\circ^\circ$-groups of finite Morley rank.
And although we have just used some phrases that our prospective reader may not know we hope our work to be of interest to the experts in finite group theory as many ideas and methods will seem familiar to them. Efforts were made in their direction and that of self-containedness.

\subsection{The Context}

\paragraph{Groups of finite Morley rank.}
Let us first say a few words of groups of finite Morley rank. We shall remain deliberately vague as we only hope to catch the reader's attention (possibly through provocation).
Should we succeed we can suggest three books. The first monograph dealing with groups of finite Morley rank, among other groups, was \cite{PGroupes}, translated as \cite{PStable}. An excellent and thorough reference textbook is \cite{BNGroups} which has no pictures but many exercises instead. The more recent \cite{ABCSimple} quickly focuses on the specific topic of the classification of the infinite simple groups of finite Morley rank of so-called even or mixed type, a technical assumption. For the moment let us be quite unspecific.

Morley rank is a notion invented by model theorists for the purposes of pure mathematical logic, and turned out to be an abstract form of the Zariski dimension in algebraic geometry.
It was then natural to investigate the relations between groups of finite Morley rank and algebraic groups.

More precisely (we shall keep this facultative paragraph short and direct the brave to \cite{PGroupes}), the rank introduced by Morley for his categoricity theorem was quickly understood to be a central notion in mathematical logic, enabling a more algebraic treatment of model-theoretic phenomena, and hopefully allowing closer interactions with classical mathematics. This was confirmed when Zilber's ``ladder'' analysis of uncountably categorical theories revealed towers of atomic pieces bound to each other by some definable groups, similar to differential Galois groups in (Kolchin-)Picard-Vessiot theory, and therefore of utmost relevance even as abstract groups. It is expected that understanding the structure of such groups would shed further light on the nature of uncountably categorical theories, which would please model theorists, and other mathematicians as well.

But because of their very nature, groups of finite Morley rank cannot be studied with the techniques of algebraic geometry, and only elementary (in both the naive and model-theoretic senses of the term) methods apply, which results in massive technological smugglering from finite group theory to model theory.

To make a long story short: some abstract groups arose in one part of mathematics; it would be good to classify the simple ones; logicians need finite group theorists.

\paragraph{Groups with a dimension.}
And now for the sake of the introduction we shall suggest a completely different, anachronistic, and self-contained motivation.

The classification of the simple Lie groups, the classification of the simple algebraic groups, and the classification of the finite simple groups are facets of a single truth: in certain categories, simple groups are matrix groups in the classical sense. The case of the finite simple groups reminds us that we are at the level of an erroneous truth, but still there must be something common to Lie groups, algebraic groups, and finite groups beyond the mere group structure that forces them to fall into the same class.

In a sense, groups of finite Morley rank describe this phenomenon; Morley rank is a form of common structural layer, or methodological least common denominator to the Lie-theoretic, algebraic geometric, and finite group-theoretic worlds.
Our groups are equipped with a dimension function on subsets enabling the most basic computations; the expert in finite group theory will be delighted to read that matching involutions against cosets, for instance, is possible. On the other hand, no analysis, no geometry, and no number theory are available. But the existence of a rudimentary dimension function is a common though thin structural layer extending the pure group structure.

It remains to say which sets are subject to having a dimension. These sets are called the definable sets; the definable class is the model-theoretic analogue of the constructible class in algebraic geometry. In a group $G$ with no extra structure, one would consider the collection of subsets of the various $G^n$ obtained by allowing group equations, (finite) boolean combinations, projections, and then by also allowing quotients by equivalence relations of the same form. This setting is a little too tight in general and model theorists enlarge the basic case of group equations by allowing other primary relations, that is, by working in an abstract structure extending the group structure.

A group of finite Morley rank is such an extended group structure with an integer-valued dimension function on its definable sets. As for the properties of the dimension function itself, they are so natural they do not need to be described.

Although we have given no definition we hope to have motivated the Cherlin-Zilber conjecture, which surmises that infinite simple groups of finite Morley rank are groups of points of algebraic groups. The conjecture goes back to the seventies.

\paragraph{Relations with finite group theory.}
A consequence of the classification of the simple, periodic, linear groups \cite{TClassification} is the locally finite version of the Cherlin-Zilber conjecture: infinite simple \emph{locally finite} groups of finite Morley rank are algebraic. \cite{TClassification} heavily relying on the classification of the finite simple groups means that conventional group theory can help elucidate problems in model theory.

A proof of the classification of the simple, periodic, linear groups in odd characteristic without using the classification of the finite simple groups but some of its methods, such as component analysis and signaliser functors, is in \cite{BClassification}. Similar techniques carried to the model-theoretic context provide the \emph{locally finite} version of the Cherlin-Zilber conjecture under an assumption standing for characteristic oddness \cite{BSimplelocally}, still without using the classification of the finite simple groups. Let us now forget about local finiteness.
All this suggests to ask whether conversely to the above, model theory may shed light on conventional group theory, and wether finite group theorists can learn something from logicians.

\cite{ABCSimple} gives a positive answer by proving the Cherlin-Zilber conjecture in even or mixed type, viz. when there is an infinite subgroup of exponent $2$, thus obtaining an ideal sketch of a decent fragment of the classification of the finite simple groups. Apart from this case one should not expect the conjecture to be proved in full generality. There is no evidence for a model-theoretic analogue of the Feit-Thompson ``odd order'' Theorem. Simple groups of finite Morley rank with no involutions cause major technical difficulties since most methods in the area heavily rely on $2$-local analysis. 
Actually the experts do not regard the existence of the most dramatic (potential) counterexamples to the conjecture called bad groups as entirely unlikely. But after all, not all finite simple groups are groups of Lie type, so refuting the Cherlin-Zilber conjecture would certainly not show that it is not interesting.

The present work deals with a certain class of \emph{small} groups of finite Morley rank: $N_\circ^\circ$-groups, defined in \S\ref{S:facts} by a condition borrowed from the classification of the finite simple groups. The former were called $*$-locally$_\circ^\circ$ soluble groups in \cite{DJGroups, DJSmall, DJLie}; we now change terminology to conform more closely to the standards of finite group theory.

\paragraph{Two notions of smallness.}
So let us push the analogy with finite group theory further. The classical $N$- property was introduced in \cite{TNonsolvable} where the full classification of the finite, non-soluble $N$-groups was given, and then proved in a series of subsequent papers: an $N$-group is a group $G$ all of whose so-called local subgroups are soluble, which in the finite case amounts to requiring that $N_G(A)$ be soluble for every abelian subgroup $1 \neq A \leq G$. The decorations in $N_\circ^\circ$ indicate that we shall focus on connected components, making our condition less restrictive than proper $N$-ness. According to the Cherlin-Zilber conjecture, every connected, non-soluble $N_\circ^\circ$-group should be isomorphic to $\PSL_2(\K)$ or $\SL_2(\K)$ with $\K$ an algebraically closed field. We cannot prove this, and our results will look partial when compared to \cite{TNonsolvable}.


Another, more restrictive notion of smallness in \cite{TNonsolvable} was minimal simplicity: a minimal simple group is a simple group all of whose proper subgroups are soluble. The full classification of the finite, minimal simple groups is given in \cite{TNonsolvable} as a corollary to that of the finite $N$-groups.
The finite Morley rank analogue is named minimal \emph{connected} simplicity and defined naturally in \S\ref{S:facts}.
According again to the Cherlin-Zilber conjecture, every minimal connected simple group should be isomorphic to $\PSL_2(\K)$ with $\K$ an algebraically closed field; even under the assumption that the group contains involutions, this is an open question.

Minimal connected simple groups of finite Morley rank have already been studied at length as recalled in \S\S \ref{s:result} and \ref{s:history}. These groups obviously are $N_\circ^\circ$-groups but it is not clear whether one should hope for a converse statement.
So transferring the partial, current knowledge from the minimal connected simple to the $N_\circ^\circ$ setting was a non-trivial task, undertaken in \cite{DJGroups, DJSmall, DJLie}.

This extension will hopefully fit into a revised strategy for the classification of simple groups of finite Morley rank with involutions. The last written account of a master plan was in Burdges' thesis \cite[Appendix A]{BSimple} and would need to be updated because of major advances in the general structural theory of groups of finite Morley rank, notably through results on torsion briefly touched upon in \S\ref{s:Sylow}. But interestingly enough the theory of $N_\circ^\circ$-groups has already been used and will be used again in another topic: permutation groups of finite Morley rank \cite{DActions, BDActions}.

The present work completes the transition from the minimal connected simple to the $N_\circ^\circ$ setting, and does more. We cannot provide a full classification of $N_\circ^\circ$-groups, but we delineate major cases and give strong restrictions on their groups of automorphisms.

\subsection{The Result and its Proof}\label{s:result}

The ideal goal would have been to show that $\PSL_2(\K)$ and $\SL_2(\K)$ are the only non-soluble $N_\circ^\circ$-groups of finite Morley rank. Under the assumption that there is an infinite elementary abelian $2$-subgroup, this is a straightforward corollary or subcase of \cite{ABCSimple} (see \cite[Theorem 4]{DJSmall}). In general the question is delicate and one should only hope to identify $\PSL_2(\K)$ and $\SL_2(\K)$ among such groups. This we do, and more, by giving restrictive information on the structure of potential counter-examples. In particular we show that such counter-examples would admit no infinite dihedral groups of automorphisms, which is likely to be of use in a prospective inductive setting.

As a matter of fact, the focus on outer involutive automorphisms, as opposed to inner involutions, became so prominent over the years (see \S\ref{s:history}) that we could take involutions out of the configurations, viz. our extra assumptions are not on the structure of the ``inner'' Sylow $2$-subgroup of the $N_\circ^\circ$-group under consideration but on the structure of that of an acting group; incidently, the inner $2$-structure is fairly well understood. Taking involutions out is a pleasant advance, but makes results slightly more complex to state.

Our Theorem below thus reads as follows: if a connected, non-soluble, $N_\circ^\circ$-group $G$ is a definable subgroup of some larger group of finite Morley rank (possibly equal to $G$) with a few assumptions on the action of outer involutions on $G$, then $G$ is either algebraic or one of four mutually exclusive configurations with common features; in any case the structure of the outer Sylow $2$-subgroup is well understood too. The existence of the four said configurations is a presumably difficult open question. But we do not need involutions inside $G$ to run the argument, and we are confident this will allow some form of induction.

The notations used below are all explained in \S\ref{S:facts}; a hasty expert curious about the proof will find informal remarks on methods at the end of the current subsection and a few words on its structure at the beginning of \S\ref{S:before}.

\begin{theorem*}
Let $\hat{G}$ be a connected, $U_2^\perp$ group of finite Morley rank and $G \trianglelefteq \hat{G}$ be a definable, connected, non-soluble, $N_\circ^\circ$-subgroup.

Then the Sylow $2$-subgroup of $G$ has one of the following structures: isomorphic to that of $\PSL_2(\C)$, isomorphic to that of $\SL_2(\C)$, a $2$-torus of Prüfer $2$-rank at most $2$.

Suppose in addition that for all involutions $\iota \in I(\hat{G})$, the group $C_G^\circ(\iota)$ is soluble.

Then $m_2(\hat{G}) \leq 2$, one of $G$ or $\hat{G}/G$ is $2^\perp$, and involutions are conjugate in $\hat{G}$. Moreover one of the following cases occurs:
\begin{description}
\item[$\bullet$ PSL$_2$:]
$G \simeq \PSL_2(\K)$ in characteristic not $2$; $\hat{G}/G$ is $2^\perp$;
\item[$\bullet$ CiBo$_\emptyset$:]
$G$ is $2^\perp$; $m_2(\hat{G}) \leq 1$; for $\iota \in I(\hat{G})$, $C_G(\iota) = C_G^\circ(\iota)$ is a self-normalising Borel subgroup of $G$;
\item[$\bullet$ CiBo$_1$:]
$m_2(G) = m_2(\hat{G}) = 1$; $\hat{G}/G$ is $2^\perp$; for $i \in I(\hat{G}) = I(G)$, $C_G(i) = C_G^\circ(i)$ is a self-normalising Borel subgroup of $G$;
\item[$\bullet$ CiBo$_2$:]
$\Pr_2(G) = 1$ and $m_2(G) = m_2(\hat{G}) = 2$; $\hat{G}/G$ is $2^\perp$; for $i \in I(\hat{G}) = I(G)$, $C_G^\circ(i)$ is an abelian Borel subgroup of $G$ inverted by any involution in $C_G(i)\setminus\{i\}$ and satisfies $\rk G = 3 \rk C_G^\circ(i)$;
\item[$\bullet$ CiBo$_3$:]
$\Pr_2(G) = m_2(G) = m_2(\hat{G}) = 2$; $\hat{G}/G$ is $2^\perp$; for $i \in I(\hat{G}) = I(G)$, $C_G(i) = C_G^\circ(i)$ is a self-normalising Borel subgroup of $G$; if $i \neq j$ are two involutions of $G$ then $C_G(i) \neq C_G(j)$.
\end{description}
\end{theorem*}

There is at present no hope to kill any of the non-algebraic configurations of type \textbf{CiBo} (``Centralisers of Involutions are Borel subgroups''; unlike the cardinal of the same name, these configurations are far from innocent). Three of these configurations were first and very precisely described in \cite{CJTame} under much stronger assumptions of both a group-theoretic and a model-theoretic nature, and the goal of \cite{BCJMinimal, DGroupes, DGroupes1, DGroupes2} merely was to carry the same analysis with no model-theoretic restrictions. Despite progress in technology, nothing new could be added on the \textbf{CiBo} configurations since their appearance in \cite{CJTame}.
So it is likely these potential monstrosities will linger for a while; one may even imagine that they ultimately might be proved consistent.

Beyond porting the description of non-algebraic configurations from the minimal connected simple setting \cite{DGroupes} 
to the broader $N_\circ^\circ$ context, our theorem gives strong limitations on how these potential counterexamples would embed into bigger groups. This line of thought goes back to Delahan and Nesin proving that so-called simple bad groups have no involutive automorphisms (\cite{DNSharply}; \cite[Proposition 13.4]{BNGroups}).
The question of involutive automorphisms of minimal connected simple groups has already been addressed in \cite{BCDAutomorphisms} and \cite{FAutomorphisms}; we insist that a significant part of our results was not previously known in the minimal connected setting.
This is the reason why we believe that our theorem, however partial and technical it may look, will prove relevant to the classification project.

The present result therefore replaces a number of earlier (pre)publications: \cite{BCJMinimal, DGroupes, DGroupes1, DGroupes2, DJ4alpha, BCDAutomorphisms}, the contents of which are described in \S\ref{s:history} hereafter.
(We cannot dismiss Frécon's analysis \cite[Theorem 3.1]{FAutomorphisms} as it heavily uses the solubility of centralisers of $p$-elements, a property which might fail in the $N_\circ^\circ$ case.)
%
%
\medskip

And now we wish to say a few words about the proof. One cannot adapt \cite{TNonsolvable} and subsequent papers. The expert in finite group theory will appreciate here how little structure there is on a group of finite Morley rank. A finite analogue of \textbf{CiBo}$_1$, for instance, has a cyclic Sylow $2$-subgroup; for a variety of classical reasons it has a normal $2$-complement; if an $N$-group, it is soluble.
We would be delighted to see quick arguments removing finite analogues of \textbf{CiBo}$_2$ and \textbf{CiBo}$_3$. 
In any case, however elementary they may seem, such methods are not available in our context.
Character theory, remarkably absent from \cite{TNonsolvable}, cannot be used either. Even Sylow theory (see \S\S\ref{s:semisimplicity} and \ref{s:Sylow}) is rudimentary. From finite group theory there remains of course $2$-local analysis, but we are dealing with small cases where one cannot apply the standard machinery, otherwise well acclimatised to the finite Morley rank setting.

The main group-theoretic method is then matching involutions against cosets, in the spirit of Bender as quoted in the beginning of \S\ref{s:genericity}. At times our arguments in this line are rather classical and Proposition~\ref{p:algebraicity} for instance may have a known counterpart in finite group theory, at times we suspect they are unorthodoxly convoluted like in Proposition~\ref{p:maximality}.
But this is our main method mostly because we lack a better option.
We also use a variant of local analysis \cite{BUniqueness} developed by Burdges for groups of finite Morley rank (\S\S\ref{s:unipotence} and \ref{s:Borel}). This will not surprise the expert.

As for model-theoretic methods, we see two main lines. First, we tend to focus on generic elements of groups, with the effect of smoothing phenomena. The general theory of genericity in model-theoretic contexts owes much to Cherlin and Poizat so one could refer the reader to \cite{PGroupes}, but thanks to the rank function it is a rather obvious notion here. In the same vein we often resort to connectedness arguments which from the point of view of algebraic group theory will always be straightforward. Typical of connectedness methods is Zilber's Indecomposibility Theorem \cite[Theorem 5.26]{BNGroups}.
The use of fields is the second essential feature; although Zilber's Field Theorem \cite[Theorem 9.1]{BNGroups} nominally appears only in the proofs of Propositions \ref{p:algebraicity} and \ref{p:Yanartas}, it underlies our knowledge of soluble groups, in particular the unipotence theory of \S\ref{s:unipotence} which is fundamental for the whole analysis.

The structure of the proof is briefly described in \S\ref{S:before}.


\subsection{Version History}\label{s:history}

The current subsection will be of little interest to a reader not familiar with the community of groups of finite Morley rank; we include it mostly because the present article marks the voyage's end.

The project of classifying $N_\circ^\circ$-groups with involutions started as early as 2007 under the suggestion of Borovik and yet is only the last chapter of an older story: the identification of $\PSL_2(\K)$ among small groups of odd type.
\begin{itemize}
\item
We could go back to Cherlin's seminal article on groups of small Morley rank \cite{CGroups} which identified $\PSL_2(\K)$, considered bad groups, and formulated the algebraicity conjecture. Other important results on $\PSL_2(\K)$ in the finite Morley rank context were found by Hrushovski \cite{HAlmost}, Nesin et Ali(i) \cite{NSharply, BDNCIT, DNZassenhaus}.
But we shall not go this far.
\item
Jaligot was the first to do something specifically in so-called odd type \cite{JFT}, adapting computations from \cite{BDNCIT} (we say a bit more in \S\ref{s:genericity} and \ref{s:algebraicity}).
\item
Another preprint by Jaligot \cite{JTame}, then at Rutgers University, deals with \emph{tame} minimal connected simple groups of Prüfer rank $1$. (Tameness is a model-theoretic assumption on fields arising in a group, already used for instance in \cite{DNZassenhaus}.) In this context, either the group is isomorphic to $\PSL_2(\K)$, or centralisers$^\circ$ of involutions are Borel subgroups.

Quite interestingly the tameness assumption, viz. ``no bad fields'', appears there in small capitals and bold font each time it is used; it seems clear that Jaligot already thought about removing it.
\item
Jaligot's time at Rutgers resulted in a monumental article with Cherlin \cite{CJTame} where tame minimal connected simple groups were thoroughly studied and potential non-algebraic configurations carefully described. The very structure of our Theorem reflects the result of \cite{CJTame}.
\item
A collaboration between Burdges, Cherlin, and Jaligot \cite{BCJMinimal} was significant progress towards removing tameness: minimal connected simple groups have Prüfer rank at most $2$.
\item
Using major advances by Burdges (described in \S\S\ref{s:unipotence} and \ref{s:Borel}), the author could entirely remove the tameness assumption from \cite{CJTame} and reach essentially the same conclusions. This was the subject of his dissertation under the supervision of Jaligot (\cite{DGroupes}, published as \cite{DGroupes1, DGroupes2}).
\item
A few months before the completion of the author's PhD, the present project of classifying $N_\circ^\circ$-groups of finite Morley rank was suggested by Borovik, a task the author and Jaligot undertook with great enthusiasm and which over the years resulted in the series \cite{DJGroups, DJSmall, DJLie}.

A 2008 preprint \cite{DJ4alpha} was close to fully porting \cite{DGroupes} to the $N_\circ^\circ$ context. Involutions remained confined inside the group. (This amounts to supposing $\hat{G} = G$ in the Theorem.)
\item
While a post-doc at Rutgers University the author in an unpublished joint work with Burdges and Cherlin \cite{BCDAutomorphisms} went back to the minimal connected simple case but with outer involutory automorphisms. (This amounts to supposing $G$ minimal connected simple and $2^\perp$ in the Theorem.)
\item
Delays and shifts in interests postponed both \cite{DJ4alpha} and \cite{BCDAutomorphisms}. In the Spring of 2013 the author tried to convince Jaligot that time had come to redo \cite{DJ4alpha} in full generality, that is with outer involutions. The present Theorem was an ideal statement we vaguely dreamt of but we never discussed nor even mentioned to each other anything beyond as it looked distant enough. In March and April of that year we were trying to fix earlier proofs with all possible repair patches, and mixed success.

The author recalls how Jaligot would transcribe those meetings in a small red ``Rutgers'' notebook when visiting Paris. He
did not recover 
these notes after Jaligot's untimely death.
\end{itemize}

And this is how a project started with great enthusiasm was completed in grief and sorrow; yet completed.
The author feels he is now repaying his debt for the care he received as a student, for an auspicious dissertation topic, and for all the friendly confidence his advisor trusted him with.

In short I hope that the present work is the kind of monument \'Eric's shadow begs for.
I dare print that the article is much better than last envisioned in the Spring of 2013. Offended reader, understand --- that \emph{there} precisely lies my tribute to him.




Such a reconstruction would never have been even imaginable without the hospitality of the Mathematics Institute of NYU Shanghai during the Fall of 2013. The good climate and supportive staff made it happen. At various later stages the comments of Gregory Cherlin proved unvaluable, as always. Last but not least and despite the author's lack of taste for mixing genres, Lola's immense patience is most thankfully acknowledged.
%

\section{Prerequisites and Facts}\label{S:facts}

We have tried to make the article as self-contained as possible, an uneasy task since the theory of groups of finite Morley rank combines a variety of methods. Reading the prior articles in the series \cite{DJGroups, DJSmall, DJLie} is not necessary to understand this one.
In the introduction we already mentioned three general references \cite{PGroupes, BNGroups, ABCSimple}.
Yet we highly recommend the preliminaries of a recent research article, \cite[\S 2]{ABFWeyl}; the reader may wish to first look there before picking a book from the shelves.

We denote by $d(X)$ the definable hull of $X$, i.e. the smallest definable group containing $X$. If $H$ is a definable group, we denote by $H^\circ$ its connected component. If $H$ fails to be definable we then set $H^\circ = H\cap d^\circ(H)$. These constructions behave as expected.

One more word on general terminology: the author supports linguistic minorities.

\begin{definition*}[{\cite[Definition 3.1(4)]{DJGroups}}]
A group $G$ of finite Morley rank is an $N_\circ^\circ$-group if $N_G^\circ(A)$ is soluble for every nontrivial, definable, abelian, connected subgroup $A \leq G$.
\end{definition*}

\begin{remarks*}\
\begin{itemize}
\item
The property was named $*$-local$_\circ^\circ$ solubility  in \cite{DJGroups, DJSmall, DJLie}; the $*$- prefix was a mere warning to the eye in order to distinguish from local conditions in the usual sense, the lower $_\circ$ was supposed to stand for the connectedness assumption on $A$, and the upper $^\circ$ symbolised the conclusion only being on the connected component $N_G^\circ(A)$.

We prefered to adapt Thompson's $N$- terminology from \cite{TNonsolvable} by simply adding connectedness symbols.
\item
We do require full $N_\circ^\circ$-ness in our proofs and can apparently not restrict to a certain class of local subgroups. \cite{GLNonsolvable} for instance extends Thompson's classification of the non-soluble, finite $N$-groups to the non-soluble, finite groups where only $2$-local subgroups are supposed to be soluble (i.e., when $A$ above must in addition be a $2$-group).

Such a generalisation looks impossible in our setting as will become obvious during the proof, simply because we must take too many normalisers.
\item
Many results in the present work will be stated for $N_\circ^\circ$-groups of finite Morley rank. With our definition this is redundant but as other contexts, model-theoretic in particular, give rise to a notion of a connected component, this also is safer.
\end{itemize}
\end{remarks*}

\begin{remark*}[and Definition]
An extreme case of an $N_\circ^\circ$-group $G$ is when all definable, connected, proper subgroups of $G$ are soluble; $G$ is then said to be \emph{minimal connected simple}. As opposed to past work (see \S\S\ref{s:result} and \ref{s:history}) the present article does not rely on minimal connected simplicity.

Like we said in the introduction, there is no hope to prove that $N_\circ^\circ$-groups are close to being minimal connected simple. One could expect many more configurations \cite[\S3.3]{DJGroups}.
\end{remark*}

As one imagines, involutions will play a major role. We denote by $I(G)$ the set of involutions in $G$; $i, j, k, \ell$ will stand for some of them. We also use $\iota, \kappa, \lambda$ for involutions of the bigger, ambient group $\hat{G}$.
When a group has no involutions, we call it $2^\perp$.
We shall refer to the following as ``commutation principles''.

\begin{fact}\label{f:commutation}
Suppose that there exists some involutive automorphism $\iota$ of a semidirect product $H \rtimes K$, where $K$ is $2$-divisible, and that $\iota$ centralises or inverts $H$, and inverts $K$. Then $[H, K] = 1$.
\end{fact}

\subsection{Semisimplicity}\label{s:semisimplicity}


In what follows, $p$ stands for a prime number.

\begin{fact}[{torsion lifting, \cite[Exercise 11 p.98]{BNGroups}}]
Let $G$ be a group of finite Morley rank, $H \trianglelefteq G$ be a normal, definable subgroup and $x \in G$ be such that $xH$ is a $p$-element in $G/H$. Then $d(x) \cap xH$ contains a $p$-element of $G$.
\end{fact}

Apart from the above principle, most of our knowledge of torsion relies either on the assumption that $p = 2$, on some solubility assumption, or on a $U_p^\perp$-ness assumption explained below.

\begin{itemize}
\item
To emphasize the case where $p=2$, recall that in groups of finite Morley rank maximal $2$-subgroups, also known as Sylow $2$-subgroups, are conjugate (\cite[Theorem 10.11]{BNGroups}, originating in \cite{BInvolutions}). As a matter of fact, their structure is known \cite[Corollary 6.22]{BNGroups}. If $S$ is a Sylow $2$-subgroup then $S^\circ = T \ast U_2$ where $T$ is a $2$-torus and $U_2$ a $2$-unipotent group. Let us explain the terminology:
\begin{itemize}
\item
$T$ is a sum of finitely many copies of the Prüfer $2$-group, $T \simeq \Z_{2^\infty}^d$, and $d$ is called the Prüfer $2$-rank of $T$, which we denote by $\Pr_2(T) = d$. By conjugacy, $\Pr_2(G) = \Pr_2(T)$ is well-defined. Interestingly enough, $N_G^\circ(T) = C_G^\circ(T)$ \cite[Theorem 6.16, ``rigidity of tori'']{BNGroups}; the latter actually holds for any prime.
\item
$U_2$ in turn has bounded exponent. We shall mostly deal with groups having no infinite such subgroups, and we call them $U_2^\perp$ \emph{groups}.
\end{itemize}
The $2$-rank $m_2(G)$ is the maximal rank (in the finite group-theoretic sense) of an elementary abelian $2$-subgroup of $G$; again this is well-defined by conjugacy. A $U_2^\perp$ assumption implies finiteness of $m_2(G)$; one always has $\Pr_2(G) \leq m_2(G)$; see \cite{DpRank} for some reverse inequality.
\item
Actually the same holds for any prime $p$ provided the ambient group of finite Morley rank is soluble (\cite[Theorem 6.19 and Corollary 6.20]{BNGroups}, originating in \cite{BPTores}). In case the ambient group is also connected, then Sylow $p$-subgroups are connected \cite[Theorem 9.29]{BNGroups}.
We call this fact the structure of torsion in definable, connected, soluble groups.
\item
A group of finite Morley rank is said to be $U_p^\perp$ (also: of $p^\perp$-type) if it contains no infinite, elementary abelian $p$-group.
A word on Sylow $p$-subgroups of $U_p^\perp$ groups is said in \S\ref{s:Sylow}.
\end{itemize}

For the moment we give another example of how we often rely either on some specific assumption on involutions, or on solubility.

\begin{fact}[{bigeneration, \cite[special case of Theorem 2.1]{BCGeneration}}]\label{f:bigeneration}
Let $\hat{G}$ be a $U_p^\perp$ group of finite Morley rank. Suppose that $\hat{G}$ contains a non-trivial, definable, connected, normal subgroup $G \trianglelefteq \hat{G}$ and some elementary abelian $p$-group of $p$-rank $2$, say $\hat{V} \leq \hat{G}$.
If $G$ is soluble, or if $p = 2$ and $G$ has no involutions, then $G = \langle  C_G^\circ(v): v \in \hat{V}\setminus\{1\}\rangle$.
\end{fact}

We finish with a property of repeated use.

\begin{fact}[Steinberg's torsion theorem, \cite{DSteinberg}]\label{f:Steinberg}
Let $G$ be a connected, $U_p^\perp$ group of finite Morley rank and $\zeta \in G$ be a $p$-element such that $\zeta^{p^n} \in Z(G)$. Then $C_G(\zeta)/C_G^\circ(\zeta)$ has exponent dividing $p^n$.
\end{fact}

As the argument essentially relies on the connectedness of centralisers of \emph{inner} tori obtained by Alt\i nel and Burdges \cite[Theorem 1]{ABAnalogies}, one should not expect anything similar for outer automorphisms of order $p$, not even for outer toral automorphisms.

\subsection{Sylow Theory}\label{s:Sylow}

By definition, a \emph{Sylow $p$-subgroup} of some group of finite Morley rank is a maximal, \emph{soluble} $p$-subgroup.
It turns out that for a $p$-subgroup of a group of finite Morley rank, solubility is equivalent to local solubility (in the usual sense of finitely generated subgroups being soluble) \cite[Theorem 6.19]{BNGroups}, so every soluble $p$-subgroup is contained in some Sylow $p$-subgroup all right.
But the solubility requirement is not for free: even if a group of finite Morley rank $G$ is assumed to be $U_p^\perp$, it is not known whether every $p$-subgroup of $G$ is soluble; as a matter of fact it is apparently not known whether $G$ can embed a Tarski monster.
In short, a Sylow $p$-subgroup is not necessarily a maximal $p$-subgroup, even in the $U_p^\perp$ case. We now focus on Sylow $p$-subgroups.

As suggested above, Sylow $p$-subgroups of a $U_p^\perp$ group of finite Morley rank are toral-by-finite \cite[Corollary 6.20]{BNGroups}. There is more.

\begin{fact}[{\cite[Theorem 4]{BCSemisimple}}]\label{f:Sylowconjugate}
Let $G$ be a $U_p^\perp$ group of finite Morley rank. Then Sylow $p$-subgroups of $G$ are conjugate.
\end{fact}

\begin{remarks*}
Let $\hat{G}$ be a $U_p^\perp$ group of finite Morley rank and $G \trianglelefteq \hat{G}$ be a definable, normal subgroup.
\begin{itemize}
\item
The Sylow $p$-subgroups of $G$ are exactly the traces of the Sylow $p$-subgroups of $\hat{G}$.

A Sylow $p$-subgroup of $G$ is obviously the trace of some Sylow $p$-subgroup of $\hat{G}$.
The converse is immediate by conjugacy of the Sylow $p$-subgroups in the $U_p^\perp$ group $\hat{G}$.
\item
The Sylow $p$-subgroups of $\hat{G}/G$ are exactly the images of the Sylow $p$-subgroups of $\hat{G}$. The following argument was suggested by Gregory Cherlin.

Let $\varphi$ be the projection modulo $G$.
Suppose that $\hat{S}$ is a Sylow $p$-subgroup of $\hat{G}$ but $\varphi(\hat{S})$ is not a Sylow $p$-subgroup of $\hat{G}/G$. Then by the normaliser condition \cite[Corollary 6.20]{BNGroups} there is a $p$-element $\alpha \in N_{\hat{G}/G}(\varphi(\hat{S}))\setminus\varphi(\hat{S})$, which we lift to a $p$-element $a \in \hat{G}$. Observe that $\alpha \in N_{\hat{G}/G}(\varphi(\hat{S}^\circ))$, so $\varphi([a, \hat{S}^\circ G]) = [\alpha, \varphi(\hat{S}^\circ)] \leq \varphi(\hat{S}^\circ G)$ and $a \in N_{\hat{G}}(\hat{S}^\circ G)$.

Now $N = N_{\hat{G}}(\hat{S}^\circ G)$ is definable since it is the inverse image of $N_{\hat{G}/G}(\varphi(\hat{S}^\circ))$ which is definable as the normaliser of a $p$-torus by the rigidity of tori. In particular, $N$ conjugates its Sylow $p$-subgroups, and a Frattini argument yields $N \leq \hat{S}^\circ G \cdot N_{\hat{G}}(\hat{S}) \leq G N_{\hat{G}}(\hat{S})$. Write $a = g n$ with $g \in G$ and $n \in N_{\hat{G}}(\hat{S})$; $n$ is a $p$-element modulo $G$, so lifting torsion there is a $p$-element $m \in d(n) \cap nG$. Then $m \in N_{\hat{G}}(\hat{S})$ and therefore $m \in \hat{S}$. Hence $a = gn \in nG = mG \subseteq \hat{S}G$ and $\alpha = \varphi(a) \in \varphi(\hat{S})$, a contradiction.

As a consequence the image of any Sylow $p$-subgroup of $\hat{G}$ is a Sylow $p$-subgroup of $\hat{G}/G$. The converse is now immediate, conjugating in $\hat{G}/G$.
\item
Without the $U_p^\perp$ assumption this remains quite obscure. The reader will find in \cite{PWSous-groupes, PWLiftez} a model-theoretic discussion.
\end{itemize}
\end{remarks*}

We shall refer to the many consequences of the following fact as ``torality principles''.

\begin{fact}[{\cite[Corollary 3.1]{BCSemisimple}}]
Let $\up$ be a set of primes.
Let $G$ be a connected group of finite Morley rank with a $\up$-element $x$ such that $C(x)$ is $U_{\up}^\perp$. Then $x$ belongs to any maximal $\up$-torus of $C(x)$.
\end{fact}

And now for some unrelated remarks involving some notions from \cite{CGood}. A decent torus is a definable, divisible, abelian subgroup which equals the definable hull of its torsion subgroup. Goodness is the hereditary version of decency: a good torus is a definable, connected subgroup all definable, connected subgroups of which are decent tori.

\begin{remarks*}\
\begin{itemize}
\item
Let $\hat{G}$ be a connected, $U_p^\perp$ group of finite Morley rank and $G \trianglelefteq \hat{G}$ be a definable, connected subgroup. If $\hat{T} \leq \hat{G}$ is a maximal $p$-torus of $\hat{G}$, then $T = \hat{T}\cap G$ is a maximal $p$-torus of $G$.

Let $\hat{S} \geq \hat{T}$ be a Sylow $p$-subgroup of $\hat{G}$. Then $S = \hat{S} \cap G$ is a Sylow $p$-subgroup of $G$. So $T = G \cap \hat{T} \leq G \cap \hat{S}^\circ \leq C_{S}(S^\circ) = S^\circ$ by torality principles. Hence $T \leq S^\circ \leq \hat{S}^\circ \cap G = \hat{T} \cap G = T$.
\item
This is not true for an arbitrary $p$-torus $\hat{\tau}\leq \hat{G}$: take two copies $T_1, T_2$ of $\Z_{2^\infty}$ with respective involutions $i$ and $j$; now let $\hat{G} = (T_1\times T_2)/\langle  ij\rangle$ and $G$ be the image of $T_1$. Then the intersection of (the image of) $T_2$ with $G$ is $\langle  \overline{i}\rangle$.
\item
This is not true if $\hat{G}$ is not $U_p^\perp$. Take for instance two Prüfer $p$-groups $T \simeq T' \simeq \Z_{p^\infty}$, an infinite elementary abelian $p$-group $A$, and a central product $K = T' \ast A$ with $T' \cap A = \langle  a\rangle \neq \{1\}$. Set $G = T \times A$ and $\hat{G} = T \times K$. One will find $\hat{T} = T \times T'$, but $\hat{T} \cap G = T \times \langle  a\rangle$ is not connected.
\item
Similarly, if $\hat{\Theta}$ is a good torus of $\hat{G}$ then $(\hat{\Theta}\cap G)^\circ$ is one of $G$, but connectedness of $\Theta = \hat{\Theta}\cap G$ is not granted even when $\hat{\Theta}$ is maximal;
of course connectedness holds if $G$ is $U_p^\perp$ for every prime number $p$.
\item
As for maximal decent tori, their connected intersections with subgroups need not be decent tori; in the language of the next subsection, $(0, 0)$-groups need not be homogeneous.
\end{itemize}
All this begs for a notion of reductivity which is not our present goal.
\end{remarks*}

\subsection{Unipotence}\label{s:unipotence}

Developing a suitable theory of unipotence in the context of abstract groups of finite Morley rank took some time. One needs to describe a geometric phenomenon in group-theoretic terms. The positive characteristic notion may look straightforward to the hasty reader: when $p$ is a prime number, a $p$-\emph{unipotent subgroup} is a definable, connected, nilpotent $p$-group of bounded exponent. Yet the definition is naive only in appearance. First, nilpotence is perhaps not for free, as indicated in \S\ref{s:Sylow}. Second, Baudisch has constructed a non-abelian $p$-unipotent group not interpreting a field \cite{BNew}: as a consequence, the Baudisch group does not belong to algebraic geometry (for more on field interpretation, see \cite{GHStable}). Despite these technical complications, the notion of unipotence in positive characteristic remains rather intuitive.

Matters are considerably worse in characteristic zero as there is no intrinsic way to distinguish, say, some torsion-free subgroup of $\C^\times$ from the additive group of some other field. Unpublished work by Altseimer and Berkman dated 1998 on so-called ``pseudounipotent'' and ``quasiunipotent'' subgroups, two notions which we shall not define, therefore required tameness assumptions on fields arising in the structure (see \S\ref{s:history}).

Burdges found a satisfactory unipotence theory; the point (and also the difficulty) is that one has a multiplicity of notions in characteristic zero. We do not wish to describe his construction.
For a complete exposition of Burdges' unipotence theory, see Burdges' PhD \cite[Chapter 2]{BSimple}, its first formally published expositions \cite{BSignalizer, BSylow}, or the first article in the present series \cite{DJGroups}.

A \emph{unipotence parameter} is a pair of the form $(p, \infty)$ where $p$ is a prime, or $(0, d)$ where $d$ is a non-negative integer. The case $(0, 0)$ describes decent tori. We shall denote unipotence parameters by $\rho, \sigma, \tau$. For every parameter $\rho$, there is a notion of a $\rho$-group, and of the $\rho$-generated subgroup $U_\rho(G)$ of a group $G$. Bear in mind that by definition, a $\rho$-group is always definable, connected, and nilpotent; the latter need not hold of the $\rho$-generated subgroup even if the ambient group is soluble.

\begin{notation*}
We order unipotence parameters as follows:
\[(2, \infty) \succ (3, \infty) \succ \dots \succ (p, \infty) \succ \dots \succ (0, \rk(G)) \succ \dots \succ (0, 0)\]
\end{notation*}

\begin{notation*}\
\begin{itemize}
\item
For any group of finite Morley rank $H$, $\rho_H$ will denote the greatest unipotence parameter it admits, i.e. with $U_{\rho_H}(H) \neq 1$; we simply call it \emph{the} parameter of $H$.

(Any infinite group of finite Morley rank admits a parameter, possibly $(0, 0)$: see \cite[Theorem 2.19]{BSimple}, \cite[Theorem 2.15]{BSignalizer}, or \cite[Lemma 2.6]{DJGroups}.)

Be careful that the parameter of a group equal to its $\rho$-generated subgroup can be greater than $\rho$: take a decent torus which is not good and $\rho = (0, 0)$.
(More generally a definable, connected, soluble group $H$ has parameter $(0, 0)$ iff a good torus, but $H = U_{(0, 0)}(H)$ iff $H$ is generated by its decent tori.)
\item
For $\iota$ a definable involutive automorphism of some group of finite Morley rank, let $\rho_\iota = \rho_{C^\circ(\iota)}$.
\end{itemize}
\end{notation*}

With these notations at hand let us review a few classical properties. The reader should be familiar with the following before venturing further.

\begin{fact}\label{f:unipotence}\
\begin{enumerate}[(i)]
\item
If $N$ is a connected, nilpotent group of finite Morley rank then $N = \bigast_{\rho} U_\rho(N)$ (central product) where $\rho$ ranges over all unipotent parameters (Burdges' decomposition of nilpotent groups: \cite[Theorem~2.31]{BSimple}, \cite[Corollary 3.6]{BSylow}, \cite[Fact 2.3]{DJGroups});
\item
if $H$ is a connected, soluble group of finite Morley rank, one has $U_{\rho_H}(H) \leq F^\circ(H)$ (\cite[Theorem 2.21]{BSimple}, \cite[Theorem 2.16]{BSignalizer}, \cite[Fact 2.8]{DJGroups}); incidently, the connected component of the Fitting subgroup $F^\circ(H)$ is defined and studied in \cite[\S7.2]{BNGroups}; one has $H' \leq F^\circ (H)$ \cite[Corollary 9.9]{BNGroups};
\item\label{f:unipotence:UZFneq1}
if $H$ is as above then $U_{\rho_H}(Z(F^\circ(H))) \neq 1$ (\cite[Lemma 2.26]{BSimple}, \cite[Lemma 2.3]{BSylow});
\item
if a $\sigma$-group $V_\sigma$ normalises a $\rho$-group $V_\rho$ with $\rho \preccurlyeq \sigma$ then $V_\rho V_\sigma$ is nilpotent (\cite[Lemma 4.10]{BSimple}, \cite[Proposition 4.1]{BSylow}, \cite[Fact 2.7]{DJGroups});
\item\label{f:unipotence:pushandpull}
the image and preimage of a $\rho$-group under a definable homomorphism are $\rho$-groups (push-forward and pull-back: \cite[Lemma 2.12]{BSimple}, \cite[Lemma 2.11]{BSignalizer});
\item\label{f:unipotence:rhocommutator}
if $G$ is a soluble group of finite Morley rank, $S \subseteq G$ is any subset, and $H \trianglelefteq G$ is a $\rho$-subgroup, then $[H, S]$ is a $\rho$-group (\cite[Lemma 2.32]{BSimple}, \cite[Corollary 3.7]{BSylow});
\item
generalising the latter Frécon obtained a remarkable homogeneity result we shall \emph{not} use:
\begin{quote}
if $G$ is a connected group of finite Morley rank acting definably on a $\rho$-group then $[G, H]$ is a homogeneous $\rho$-group, i.e. all its definable, connected subgroups are $\rho$-groups (\cite[Theorem 4.11]{FAround}, \cite[Fact 2.1]{DJGroups}).
\end{quote}
{\rm The last phenomenon was deemed essential in all earlier versions of the present work, but to our great surprise one actually does not need it.
Frécon has developed in \cite{FAround} even subtler notions of unipotence with respect to isomorphism types instead of unipotence parameters.
}
\end{enumerate}
\end{fact}

By definition, a Sylow $\rho$-subgroup is a maximal $\rho$-subgroup. Recall from Burdges' decidedly inspiring thesis (\cite[\S4.3]{BSimple}, oddly published only in \cite[\S3.2]{FJConjugacy}) that if $\pi$ denotes a \emph{set} of unipotence parameters, then a Carter $\pi$-subgroup of some ambient group $G$ is a definable, connected, nilpotent subgroup $L$ which is $U_\pi$-self-normalising, i.e. with $U_\pi(N_G^\circ(L)) = L$ (the $\pi$-generated subgroup is defined naturally and always definable and connected). Carter subgroups, i.e. definable, connected, nilpotent, almost-self-normalising subgroups are examples of the latter where $\pi$ is the set of all unipotence parameters. All this is very well-understood in a soluble context \cite{WNilpotent,FSous}.

\subsection{Borel Subgroups and Intersections}\label{s:Borel}

\begin{definition*}
A Borel subgroup of a group of finite Morley rank is a definable, connected, soluble subgroup which is maximal as such.
\end{definition*}

We shall refer to the following as ``uniqueness principles''.

\begin{fact}[{\cite[from Corollary 4.3]{DJGroups}}]\label{f:uniqueness}
Let $G$ be an $N_\circ^\circ$-group of finite Morley rank and $B$ be a Borel subgroup of $G$.
Let $U \leq B$ be a $\rho_B$-subgroup of $B$ with $\rho_{C_G^\circ(U)}\preccurlyeq \rho_B$. Then $U_{\rho_B}(B)$ is the only Sylow $\rho_B$-subgroup of $G$ containing $U$. Furthermore $B$ is the only Borel subgroup of $G$ with parameter $\rho_B$ containing $U$.
\end{fact}

\begin{remarks*}\
\begin{itemize}
\item
Because of our ordering on unipotence parameters and our definition of $\rho_B$, the result does hold when $\rho_B = (0, 0)$, i.e. for $B$ a good torus (cf. \cite[Remark (3) after Theorem 4.1]{DJGroups}).

It would actually suffice to preorder parameters by $(0, k+1) \succ (0, k)$, and $(p, \infty) \succ (0, 0)$ for any prime number $p$.
\item
In particular, if $G \trianglelefteq \hat{G}$ where $\hat{G}$ is another (not necessarily $N_\circ^\circ$-) group of finite Morley rank, then $N_{\hat{G}}(U) \leq N_{\hat{G}}(B)$.
\item
If $1 < U \trianglelefteq B$ is a non-trivial, normal $\rho_B$-subgroup of $B$ the result applies; we shall often use this with $U = U_{\rho_B}(Z(F^\circ(B)))$ (see Fact \ref{f:unipotence} \ref{f:unipotence:UZFneq1}).
\end{itemize}
\end{remarks*}

For reference we list below the facts from Burdges' monumental rewriting \cite[\S 9]{BSimple}, \cite{BBender} of Bender's Method \cite{BUniqueness} that we shall use.
The method was devised to study intersections of Borel subgroups; it is quite technical. It will play an important role throughout the proof of our main Maximality Proposition~\ref{p:maximality}. As a matter of fact it does not appear elsewhere in the present article apart from Step~\ref{p:algebraicity:st:parametercontrol} of Proposition~\ref{p:algebraicity}.

It must be noted that the Bender method does \emph{not} finish any job; it merely helps treat non-abelian cases on the same footing as the abelian case. This will be clear during Step~\ref{p:maximality:st:Jkappa} of Proposition~\ref{p:maximality}. So the reader who feels lost here must keep in mind the following:
\begin{itemize}
\item
non-abelian intersections of Borel subgroups complicate the details but do not alter in the least the skeleton of the proof of Proposition~\ref{p:maximality};
\item
the utter technicality is, in Burdges' own words \cite{BSimple}, ``motivated by desperation'';
\item
such non-abelian intersections are not supposed to exist in the first place.
\end{itemize}


Since Burdges' original work was in the context of minimal connected simple groups we need to quote \cite{DJGroups} which merely reproduced Burdges' work in the $N_\circ^\circ$ case.

\begin{fact}[{\cite[4.46(2)]{DJGroups}}]\label{f:nilpab}
Let $G$ be an $N_\circ^\circ$-group of finite Morley rank. Then any nilpotent, definable, connected subgroup of $G$ contained in two distinct Borel subgroups is abelian.
\end{fact}

Yet past the nilpotent case it is not always possible to prove abelianity of intersections of Borel subgroups. The purpose of the Bender method is then to extract as much information as possible from non-abelian intersections. Unfortunately ``as much as possible'' means much more than reasonable. This is the analysis of so-called \emph{maximal pairs} \cite[Definition 4.12]{DJGroups}, a terminology we shall avoid.

\begin{fact}[{from \cite[4.50]{DJGroups}}]\label{f:Bender}
Let $G$ be an $N_\circ^\circ$-group of finite Morley rank. Let $B \neq C$ be two distinct Borel subgroups of $G$. Suppose that $H = (B \cap C)^\circ$ is non-abelian.

Then the following are equivalent:
\begin{description}
[font=\normalfont]
\item[{\cite[4.50(1)]{DJGroups}}]
$B$ and $C$ are the only Borel subgroups of $G$ containing $H$;
\item[{\cite[4.50(2)]{DJGroups}}]
$H$ is maximal among connected components of intersections of distinct Borel subgroups;
\item[{\cite[4.50(3)]{DJGroups}}]
$H$ is maximal among intersections of the form $(B \cap D)^\circ$ where $D \neq B$ is another Borel subgroup;
\item[{\cite[4.50(6)]{DJGroups}}]
$\rho_{B} \neq \rho_{C}$.
\end{description}
\end{fact}

In the following, subscripts $\ell$ and $h$ stand for \emph{light} and \emph{heavy}, respectively.

\begin{fact}[{from \cite[4.52]{DJGroups}}]\label{f:Bender2}
Let $G, B_\ell, B_h, H$ be as in the assumptions \emph{and conclusions} of Fact \ref{f:Bender}.
For brevity let $\rho' = \rho_{H'}$, $\rho_\ell =  \rho_{B_\ell}$, $\rho_h = \rho_{B_h}$; suppose $\rho_\ell\prec \rho_h$.

Then the following hold:
\begin{description}[font=\normalfont]
\item[{\cite[4.52(2)]{DJGroups}}]
any Carter subgroup of $H$ is a Carter subgroup of $B_h$;
\item[{\cite[4.38, 4.51(3) and 4.52(3)]{DJGroups}}]
$U_{\rho'}(F(B_h)) = (F(B_h)\cap F(B_\ell))^\circ$ is $\rho'$-homogeneous; $\rho'$ is the least unipotence parameter in $F(B_h)$;
\item[{\cite[4.52(6)]{DJGroups}}]
$U_{\rho'}(H) \leq F^\circ(B_\ell)$ and $N_G^\circ(U_{\rho'}(H)) \leq B_\ell$;
\item[{\cite[4.52(7)]{DJGroups}}]
$U_\sigma(F(B_\ell)) \leq Z(H)$ for $\sigma \neq \rho'$;
\item[{\cite[4.52(8)]{DJGroups}}]
any Sylow $\rho'$-subgroup of $G$ containing $U_{\rho'}(H)$ is contained in $B_\ell$.
\end{description}
\end{fact}

And we finish with an addendum.

\begin{lemma}\label{l:Bender:addendum}
Let $\hat{G}$ be a connected group of finite Morley rank and $G \trianglelefteq \hat{G}$ be a definable, connected, non-soluble, $N_\circ^\circ$-subgroup. Let $B_1\neq B_2$ be two distinct Borel subgroups of $G$ such that $H = (B_1\cap B_2)^\circ$ is maximal among connected components of intersections of distinct Borel subgroups \emph{and non-abelian}.
Let $Q \leq H$ be a Carter subgroup of $H$. Then:
\begin{itemize}
\item
$N_{\hat{G}}(H) = N_{\hat{G}}(B_1) \cap N_{\hat{G}}(B_2)$;
\item
$N_{\hat{G}}(Q) \leq N_{\hat{G}}(B_1) \cup N_{\hat{G}}(B_2)$.
\end{itemize}
\end{lemma}
\begin{proof}
By \cite[4.50 (1), (2) and (6)]{DJGroups}, $B_1$ and $B_2$ are the only Borel subgroups of $G$ containing $H$, and they have distinct unipotence parameters. This proves the first item. Let $\rho'$ be the parameter of $H'$ and $Q_{\rho'} = U_{\rho'}(Q)$. Then $N_{\hat{G}}(Q) \leq N_{\hat{G}}(Q_{\rho'}) \leq N_{\hat{G}}(N_G^\circ(Q_{\rho'}))$ and three cases can occur, following \cite[4.51]{DJGroups}.
\begin{itemize}
\item
In case (4a), $N_{\hat{G}}(Q) \leq N_{\hat{G}}(H) = N_{\hat{G}}(B_1) \cap N_{\hat{G}}(B_2)$; we are done.
\item
In case (4b), $B_1$ is the only Borel subgroup of $G$ containing $N_G^\circ(Q_{\rho'})$, so $N_{\hat{G}}(Q) \leq N_{\hat{G}}(B_1)$.
\item
Case (4c) is similar to case (4b) and yields $N_{\hat{G}}(Q) \leq N_{\hat{G}}(B_2)$.
\qedhere
\end{itemize}
\end{proof}

\section{Requisites (General Lemmas)}

Our theorem requires extending some well-known facts, so let us revisit a few classics.
All lemmas below go beyond the $N_\circ^\circ$ setting.

\subsection{Normalisation Principles}

The results in the present subsection are folklore; it turns out that none was formally published. They originate either in \cite[Chapitre 2]{DGroupes} or in \cite{BFrattini}. We shall use them with no reference, merely invoking ``normalisation principles''.

\begin{lemma}[{cf. \cite[Lemmes 2.1.1 and 2.1.2]{DGroupes} and \cite[\S3.4]{DGroupes1}}]\label{l:snakes:p}
Let $\hat{G}$ be a group of finite Morley rank, $G \leq \hat{G}$ be a definable subgroup, $P\leq G$ be a Sylow $p$-subgroup of $G$, and $\hat{S} \leq N_{\hat{G}}(G)$ be a soluble $p$-subgroup normalising $G$. If $p \neq 2$ suppose that $\hat{G}$ is $U_p^\perp$. Then some $G$-conjugate of $\hat{S}$ normalises $P$.
\end{lemma}
\begin{proof}
Since $G$ is definable, $d(\hat{S}) \leq N_{\hat{G}}(G)$, so we may assume $\hat{G} = G \cdot d(\hat{S})$ and $G \trianglelefteq \hat{G}$. We may assume that $\hat{S}$ is a Sylow $p$-subgroup of $\hat{G}$. Recall that $S = \hat{S} \cap G$ is then a Sylow $p$-subgroup of $G$ (see for instance \S\ref{s:Sylow}). Since $G$ is definable and $U_p^\perp$ if $p \neq 2$, it conjugates its Sylow $p$-subgroups; there is $g \in G$ with $P = S^g$. Hence $\hat{S}^g$ normalises $\hat{S}^g \cap G = S^g = P$.
\end{proof}

\begin{remarks*}
The argument is slightly subtler than it looks.
\begin{itemize}
\item
The original version \cite[Lemmes 2.1.1 and 2.1.2]{DGroupes} made the unnecessary assumption that $\hat{S}$, there denoted $K$, be definable. Its proof used only conjugacy in $\hat{G}$; but when $K^{\hat{g}} \leq N_{\hat{G}}(P)$ for some $\hat{g} \in \hat{G}$, why should $K^{\hat{g}}$ be a $G$-conjugate of $K$? \cite{DGroupes} then used definability of $K$ to continue: we may assume $\hat{G} = G \cdot K \leq G \cdot N_{\hat{G}}(K)$, so $K^{\hat{g}}$ is actually a $G$-conjugate of $K$. Alas it is false in general that $d(\hat{S}) \leq N_{\hat{G}}(\hat{S})$ (consider the Sylow $2$-subgroup of $\PSL_2(\C)$). So without definability of $\hat{S}$ one is forced to use conjugacy \emph{inside} $G$ like we do here.
\item
In particular, if $G$ is not supposed to be definable (and one then needs to assume $G \trianglelefteq \hat{G}$ to save the beginning of the proof), the statement is not clear at all since an arbitrary subgroup of a $U_p^\perp$ group of finite Morley rank need not conjugate its Sylow $p$-subgroups, take $\PSL_2(\Z[\sqrt{3}]) \leq \PSL_2(\C)$ for instance.
But for a \emph{normal} subgroup, we do not know. This could even depend on the Cherlin-Zilber conjecture.
\end{itemize}
\end{remarks*}

Recall in the following that if $\pi$ consists of a single parameter $\rho$, then a Carter $\pi$-subgroup is exactly a Sylow $\rho$-subgroup.

\begin{lemma}[{\cite[Corollaires 2.1.5 and 2.1.6]{DGroupes}}]\label{l:snakes:soluble:Carter}
Let $\hat{G}$ be a group of finite Morley rank, $H \leq \hat{G}$ be a \emph{soluble}, definable subgroup, $\pi$ be a set of unipotence parameters, $L\leq H$ be a Carter $\pi$-subgroup of $H$, and $\hat{S} \leq N_{\hat{G}}(H)$ be a soluble $p$-subgroup normalising $H$. Suppose that $H$ is $U_p^\perp$. Then some $H$-conjugate of $\hat{S}$ normalises $L$.
\end{lemma}
\begin{proof}
We first deal with the case where $L = Q$ is a Carter subgroup of $H$; the last paragraph will handle the general case.
We may suppose that $H$ is connected; we may suppose that $\hat{G} = H \cdot d(\hat{S})$ is soluble and that $H \trianglelefteq \hat{G}$; we may suppose that $\hat{S}$ is a Sylow $p$-subgroup of $\hat{G}$.
Since $H$ is soluble it conjugates its Carter subgroups, so $\hat{G} = H \cdot N_{\hat{G}}(Q)$.

First assume that $H$ is $p^\perp$. Let $\hat{R} \leq N_{\hat{G}}(Q)$ be a Sylow $p$-subgroup of $N_{\hat{G}}(Q)$ and $\hat{R}_2 \leq \hat{G}$ be a Sylow $p$-subgroup of $\hat{G}$ containing $\hat{R}$. Now $\hat{R}H/H$ is a Sylow $p$-subgroup of $N_{\hat{G}}(Q)H/H = \hat{G}/H$ and so is $\hat{R}_2 H/H$; therefore $\hat{R}H = \hat{R}_2 H$. But $H$ is $p^\perp$, hence $\hat{R} = \hat{R}_2$ is a Sylow $p$-subgroup of $\hat{G}$, and it normalises $Q$.

If we no longer assume that $H$ is $p^\perp$, then since $H$ is $U_p^\perp$ the structure of torsion in definable, connected, soluble groups implies that Sylow $p$-subgroups of $H$ are tori. By Lemma \ref{l:snakes:p}, $\hat{S}$ normalises a Sylow $p$-subgroup $P$ of $H$, so it normalises $d(P)$ as well. Up to conjugacy in $H$, $Q$ contains $P$ and therefore centralises $P$ and $d(P)$ as well. So we may work in $N_{\hat{G}}(d(P))$ and factor out $d(P)$, which reduces to the first case.
Then $\hat{S}$ normalises some Carter subgroup $\overline{C}$ of $H/d(P)$, and normalises its preimage $\varphi^{-1}(\overline{C}) \leq H$ which is of the form $\overline{C} = C d(P)/d(P)$ for some Carter subgroup $C$ of $H$ \cite[Corollaire 5.20]{FSous}. Hence $\hat{S}$ normalises $C$ modulo $d(P) \leq C$, that is, $\hat{S}$ normalises $C$.

The reader has observed that for the moment, $\hat{S}$ normalises some Carter subgroup of $H$. But by conjugacy of such groups in $H$, there is an $H$-conjugate of $\hat{S}$ normalising $Q$.

We now go back to the general case of a Carter $\pi$-subgroup $L$ of $H$ (see \S\ref{s:unipotence} for the definition). By \cite[Corollary 5.9]{FJConjugacy} there is a Carter subgroup $Q$ of $H$ with $U_\pi(Q) \leq L \leq U_\pi(Q)\cdot U_\pi(H')$; by what we just proved and up to conjugating over $H$ we may suppose that $Q$ is $\hat{S}$-invariant. So we consider the generalised centraliser $E = E_H(U_\pi(Q))$ \cite[Définition 5.15]{FSous}, a definable, connected, and $\hat{S}$-invariant subgroup of $H$ satisfying $U_\pi(Q) \leq F^\circ(E)$ \cite[Corollaire 5.17]{FSous}; by construction of $E$ and nilpotence, $\langle  L, Q\rangle \leq E$. If $E < H$ then noting that $L$ is a Carter $\pi$-subgroup of $E$ we apply induction. So we may suppose $E = H$. But in this case $U_\pi(Q) \leq F^\circ(H)$ so actually $L \leq U_\pi(F^\circ(H))$ and equality holds as the former is a Carter $\pi$-subgroup of $H$. It is therefore $\hat{S}$-invariant.
\end{proof}

The following Lemma is entirely due to Burdges who cleverly adapted the Frécon-Jaligot construction of Carter subgroups \cite{FJExistence}. We reproduce it here with Burdges' kind permission. The lemma is not used anywhere in the present article but included for possible future reference.

\begin{lemma}[\cite{BFrattini}]
Let $\hat{G}$ be a $U_2^\perp$ group of finite Morley rank, $G \leq \hat{G}$ be a definable subgroup, and $\hat{S} \leq N_{\hat{G}}(G)$ be a $2$-subgroup. Then $G$ has an $\hat{S}$-invariant Carter subgroup.
\end{lemma}
\begin{proof}
We may assume that every definable, $\hat{S}$-invariant subquotient of $G$ of smaller rank has an $\hat{S}$-invariant Carter subgroup; we may assume that $C_{\hat{S}}(G) = 1$; we may assume that $G$ is connected.

We first find an infinite, definable, abelian, $\hat{S}$-invariant subgroup. Let $\iota \in Z(\hat{S})$ be a central involution; then $C_G^\circ(\iota) < G$. If $C_G^\circ(\iota) = 1$ then $G$ is abelian and there is nothing to prove. So we may suppose that $C_G^\circ(\iota)$ is infinite and find some $\hat{S}$-invariant Carter subgroup of $C_G^\circ(\iota)$ by induction; it contains an infinite, definable, abelian, $\hat{S}$-invariant subgroup.

Let $\rho$ be the minimal unipotence parameter such that there exists a non-trivial $\hat{S}$-invariant $\rho$-subgroup of $G$ (possibly $\rho=(0, 0)$); this makes sense since there exists an infinite, definable, abelian, $\hat{S}$-invariant subgroup. Let $P \leq G$ be a maximal $\hat{S}$-invariant $\rho$-subgroup; $P \neq 1$. Let $N = N_G^\circ(P)$.

If $N < G$ then induction applies: $N$ has an $\hat{S}$-invariant Carter subgroup $Q$. So far $PQ$ is soluble; moreover for any parameter $\sigma$, $U_\sigma(Q)$ is $\hat{S}$-invariant as well. So by definition of $\rho$ and \cite[Fact 2.7]{DJGroups}, $PQ$ is actually nilpotent, hence $PQ = Q$, $P \leq Q$, and $P \leq U_\rho(Q)$. By maximality of $P$, $P = U_\rho(Q)$ is characteristic in $Q$ so $N_G^\circ(Q) \leq N_N^\circ(Q) = Q$ and $Q$ is a Carter subgroup of $G$.

Now suppose that $N = G$, that is, $P$ is normal in $G$. By induction, $\overline{G} = G/P$ has an $\hat{S}$-invariant Carter subgroup $\overline{C}$. Let $H$ be the preimage of $\overline{C}$ in $G$; $H$ is soluble. By Lemma \ref{l:snakes:soluble:Carter}, $H$ has an $\hat{S}$-invariant Carter subgroup $Q$. Here again $PQ$ is soluble and even nilpotent, so $P \leq Q$. Since $H$ is soluble, $Q/P = PQ/P$ is a Carter subgroup of $H/P = \overline{C}$ \cite[Corollaire 5.20]{FSous}, so $Q/P = \overline{C}$ and $Q = H$. Finally $N_G^\circ(Q)/P \leq N_{\overline{G}}^\circ(\overline{C}) = \overline{C} = Q/P$, so $N_G^\circ(Q) = Q$ and $Q$ is a Carter subgroup of $G$.
\end{proof}

\begin{remarks*}\
\begin{itemize}
\item
Burdges left the highly necessary assumption that $\hat{G}$ be $U_2^\perp$ implicit from the title of his prepublication and the original statement must therefore be taken with care: the Sylow $2$-subgroup of $(\overline{\F}_2)_+ \rtimes (\overline{\F}_2)^\times$ certainly does not normalise any Carter subgroup.
\item
The assumption that $p = 2$ is used only to find an infinite, definable, abelian $\hat{S}$-invariant subgroup. It is not known whether all connected groups of finite Morley rank having a definable automorphism of order $p \neq 2$ with finitely may fixed points are soluble, a classical property of algebraic groups though.
\end{itemize}
\end{remarks*}


\subsection{Involutive Automorphisms}

The need for the present subsection is the following.
\cite[Section 5]{DJSmall} collected various well-known facts in order to provide a decomposition for a connected, soluble group of odd type under an \emph{inner} involutive automorphism. But in the present article we shall consider the case of \emph{outer} automorphisms, more precisely the action of abstract $2$-tori on our groups. So the basic discussion of \cite{DJSmall} must take place in a broader setting; this is what we do here.

\begin{notation*}
If $\alpha$ is an involutive automorphism of some group $G$, we let $G^+ = C_G(\alpha) = \{g \in G: g^\alpha = g\}$ and $G^- = \{g \in G : g^\alpha = g^{-1}\}$. We also let $\{G, \alpha\} = \{[g, \alpha]: g \in G\}$ (in context there is no risk of confusion with the usual notation for unordered pairs).
\end{notation*}

If $G$ and $\alpha$ are definable, so are $G^+$, $G^-$, and $\{G, \alpha\}$; in general only the first need be a group. $\{G, \alpha\}$ is nevertheless stable under inversion, since $[g^\alpha, \alpha] = [g, \alpha]^{-1}$.
Observe that $\{G, \alpha\} \subseteq G^-$ but equality may fail to hold: for instance if $\alpha$ centralises $G$ and $G$ contains an involution $i$, then $i \in G^+ \cap G^-$ but $i \notin \{G, \alpha\} = \{1\}$.
Notice further that $G = G^+\cdot G^-$ iff $\{G, \alpha\} \subseteq (G^-)^{\wedge 2}$ and $G = G^+\cdot \{G, \alpha\}$ iff $\{G, \alpha\} \subseteq \{G, \alpha\}^{\wedge 2}$, where $X^{\wedge 2}$ denotes the set of squares of $X$.
Finally remark that $\deg \{G, \alpha\} = \deg \alpha^G \alpha = \deg \alpha^G \leq \deg G$.

\begin{lemma}[{cf. \cite[Theorem 19]{DJSmall}}]\label{l:involutiveaction:central2torus}
Let $G$ be a group of finite Morley rank with Sylow $2$-subgroup a (possibly trivial) \emph{central} $2$-torus $S$, and $\alpha$ be a definable involutive automorphism of $G$. Then $G = G^+ \cdot \{G, \alpha\}$ where the fibers of the associated product map are in bijection with $I(\{G, \alpha\})\cup\{1\} = \Omega_2([S, \alpha])$. Furthermore one has $G = (G^+)^\circ \cdot \{G, \alpha\}$ whenever $G$ is connected.
\end{lemma}
\begin{proof}
The proof follows that of \cite[Theorem 19]{DJSmall} closely and for some parts a minor adjustement would suffice. We prefer to give a complete proof and discard \cite{DJSmall}. Bear in mind that if $a^b = a^{-1}$ for two elements of our present group $G$, then $a$ has order at most $2$ (this is \cite[Lemma 20]{DJSmall}, an easy consequence of torsion lifting). Also remember from \cite[Lemma 18]{DJSmall} that $G$ is $2$-divisible: essentially because $2$-torsion is divisible and central.

\begin{step}\label{l:involutiveaction:central2torus:st:torus}
$S\cap \{G, \alpha\} = [S, \alpha]$.
\end{step}
\begin{proofclaim}
This is the argument from \cite[Theorem 19, Step 1]{DJSmall} with one more remark. One inclusion is trivial. Now let $\zeta \in S\cap \{G, \alpha\}$, and write $\zeta = [g, \alpha]$.
Since $G$ is $2$-divisible we let $h \in H$ satisfy $h^2 = g$. Let $n = 2^k$ be the order of $\zeta$. Then $[h^2, \alpha] = [h, \alpha]^h [h, \alpha] = \zeta \in Z(G)$ so $[h, \alpha]$ and $[h, \alpha]^h$ commute. Hence $1 = \zeta^n = [h, \alpha]^n [h, \alpha]^{n h}$. It follows that $h$ inverts $[h, \alpha]^n$ which must have order at most $2$: so $\xi = [h, \alpha]^{-1}$ is a $2$-element inverted by $\alpha$, and since it is central it commutes with $h$. Finally $[\xi, \alpha] = \xi^{-2} = [h, \alpha]^2 = [h^2, \alpha] = \zeta$.
\end{proofclaim}

It follows that $I(\{G, \alpha\}) \cup\{1\} = \Omega_2([S, \alpha])$, the group generated by involutions of $[S, \alpha]$.

\begin{step}\label{l:involutiveaction:central2torus:st:decomposition}
$\{G, \alpha\}$ is $2$-divisible and $G = G^+ \cdot \{G, \alpha\}$.
\end{step}
\begin{proofclaim}
Here again this is the argument from \cite[Theorem 19, Step 2]{DJSmall}; $2$-divisibility of $\{G, \alpha\}$ was announced but not explicitly proved.

Let $x = [g, \alpha] \in \{G, \alpha\}$. Like in \cite[Theorem 19, Step 2]{DJSmall}, write the definable hull of $x$ as $d(x) = \delta \oplus \langle  \gamma\rangle$ where $\delta$ is connected and $\gamma$ has finite order; rewrite $\gamma = \varepsilon \zeta$ where $\varepsilon$ has odd order and $\zeta$ is a $2$-element; let $\Delta = \delta \oplus \langle  \varepsilon\rangle$, so that $d(x) = \Delta \oplus \langle  \zeta\rangle$ where $\Delta$ is $2$-divisible and inverted by $\alpha$.
Now let $y \in \Delta$ satisfy $y^4 = x\zeta^{-1}$. Then $[gy^2, \alpha] = [g, \alpha]^{y^2} [y^2, \alpha] = x y^{-4} = \zeta \in S\cap \{G, \alpha\} = [S, \alpha]$ by Step~\ref{l:involutiveaction:central2torus:st:torus}, so there is $\xi \in S$ with $[\xi^2, \alpha] = \zeta$. Now $[y^{-1}\xi, \alpha] = [y^{-1}, \alpha]^\xi [\xi, \alpha] = y^2 [\xi, \alpha]$ squares to $y^4 [\xi, \alpha]^2 = x \zeta^{-1} [\xi^2, \alpha] = x$. The set $\{G, \alpha\}$ is therefore $2$-divisible; as observed this implies $G = G^+ \cdot \{G, \alpha\}$.
\end{proofclaim}

\begin{step}\label{l:involutiveaction:central2torus:st:fibers}
Fibers in Step~\ref{l:involutiveaction:central2torus:st:decomposition} are in bijection with $\Omega_2([S, \alpha])$.
\end{step}
\begin{proofclaim}
Let $k = [s, \alpha]$ have order at most $2$, where $s \in S$. Fix any decomposition $\gamma = a \cdot [g, \alpha]$ with $a \in G^+$ and $g \in G$. Since $\alpha$ inverts (hence centralises) $k$, one has $k a \in G^+$. Moreover $[sg, \alpha] = [s, \alpha]^g [g, \alpha] = k^g [g, \alpha] = k [g, \alpha] \in \{G, \alpha\}$. So $a [g, \alpha]= (k a)\cdot (k[g, \alpha])$ is yet another decomposition for $\gamma$.

Conversely work as in \cite[Theorem 19, Step 3]{DJSmall}: suppose that $a x = b y$ are two decompositions, with $a, b \in G^+$ and $x = [g, \alpha], y = [h, \alpha] \in \{G, \alpha\}$. Then $(a^{-1} b)^y = (xy^{-1})^y = y^{-1}x = (yx^{-1})^\alpha = (b^{-1}a)^\alpha = b^{-1}a = (a^{-1}b)^{-1}$ so $a^{-1}b$ has order at most $2$, say $k = a^{-1}b$. More precisely, $k = xy^{-1} = [g, \alpha][h, \alpha]^{-1} = [g, \alpha] h^{-\alpha} h$ is central, so $k = h [g, \alpha] h^{-\alpha} = [gh^{-1}, \alpha] \in \{G, \alpha\}$; it follows from Step~\ref{l:involutiveaction:central2torus:st:torus} that $k \in \Omega_2([S, \alpha])$.
\end{proofclaim}

\begin{step}\label{l:involutiveaction:central2torus:st:translates}
Left $G^+$-translates of the set $(G^+)^\circ \cdot \{G, \alpha\}$ are disjoint or equal.
\end{step}
\begin{proofclaim}
Like in \cite[Theorem 19, Step 4]{DJSmall}: suppose that for $a, b \in G^+$, the sets $a(G^+)^\circ\cdot \{G, \alpha\}$ and $b(G^+)^\circ \cdot \{G, \alpha\}$ meet, in say $ag_+ [g, \alpha] = b h_+ [h, \alpha]$ with natural notations. By the proof of Step~\ref{l:involutiveaction:central2torus:st:fibers}, $k = (ag_+)^{-1}(bh_+)$ is in $\Omega_2([S, \alpha])$, therefore central in $G$ and inverted (hence centralised) by $\alpha$. So $k = (bh_+)(ag_+)^{-1} = (ag_+)(bh_+)^{-1}$. Hence for any $b\gamma_+[\gamma, \alpha] \in b(G^+)^\circ\cdot \{G, \alpha\}$ one finds:
\[b\gamma_+[\gamma, \alpha] = k^2 b\gamma_+ [\gamma, \alpha] = a(g_+ h_+^{-1}\gamma_+)([\gamma, \alpha] k)\]
Since $k \in \Omega_2([S, \alpha])$, there is $s \in S$ with $k = [s, \alpha]$. So $[\gamma s, \alpha] = [\gamma, \alpha]^s [s, \alpha] = [\gamma, \alpha] k \in \{G, \alpha\}$: hence $b\gamma_+[\gamma, \alpha] \in a (G^+)^\circ \cdot \{G, \alpha\}$. This shows $b(G^+)^\circ \{G, \alpha\} \subseteq a (G^+)^\circ \{G, \alpha\}$ and the converse inclusion holds too.
\end{proofclaim}

\begin{step}
At most $\deg G$ left $G^+$-translates of $(G^+)^\circ\cdot \{G, \alpha\}$ cover $G$. In particular, if $G$ is connected, then $G = (G^+)^\circ\cdot \{G, \alpha\}$.
\end{step}
\begin{proofclaim}
We consider such left translates. They all have rank $\rk G$ by Step~\ref{l:involutiveaction:central2torus:st:fibers}. As they are disjoint or equal by Step~\ref{l:involutiveaction:central2torus:st:translates}, at most $\deg G$ of them suffice to cover $G$.
\end{proofclaim}

This completes the proof of Lemma \ref{l:involutiveaction:central2torus}.
\end{proof}

\begin{remarks*}\
\begin{itemize}
\item
Notice the flaw in \cite[Theorem 19, Step 5]{DJSmall}, where ``at most'' is erroneously replaced by ``exactly''. The reason is that the degree of $\alpha^G$ need not be $1$ in general, all one knows is $\deg \alpha^G \leq \deg G$. For instance, let $\alpha$ invert $\Z/3\Z$. Then $\deg G = 3$ but $(G^+)^\circ \cdot G^- = G$.
\item
If $G$ is a connected group of finite Morley rank of odd type whose Sylow $2$-subgroup $S$ is central, then $S$ is a $2$-torus as $S = C_S(S^\circ) = S^\circ$ by torality principles.
\item
The Lemma fails if $S$ is not $2$-divisible, even at the abelian level: let $\alpha$ invert $\Z/4\Z$.
\end{itemize}
\end{remarks*}

As a consequence we deduce another useful decomposition which will be used repeatedly.

\begin{lemma}[{cf. \cite[Lemma 24]{DJSmall}}]\label{l:involutiveaction:soluble}
Let $H$ be a $U_2^\perp$, connected, soluble group of finite Morley rank, and $\alpha$ be a definable involutive automorphism of $H$. Suppose that $\{H, \alpha\} \subseteq F^\circ(H)$.
Then $H = (H^+)^\circ\cdot \{H, \alpha\}$ with finite fibers.
\end{lemma}
\begin{proof}
By normalisation principles, $H$ admits an $\alpha$-invariant Carter subgroup $Q$; by the theory of Carter subgroups of soluble groups, $H = Q\cdot F^\circ(H)$ \cite[Corollaire 5.20]{FSous}. Now both $Q$ and $F^\circ(H)$ are definable, connected, nilpotent, and $U_2^\perp$: so Lemma \ref{l:involutiveaction:central2torus} applies to them.
Hence $Q = (Q^+)^\circ \cdot \{Q, \alpha\} \subseteq (H^+)^\circ \cdot F^\circ(H)$, and:
\[H = Q \cdot F^\circ(H) \subseteq (H^+)^\circ \cdot F^\circ(H) \subseteq (H^+)^\circ \cdot (F^\circ(H)^+)^\circ \cdot \{F^\circ(H), \alpha\} \subseteq (H^+)^\circ\cdot \{H, \alpha\}\]

The fibers are finite: this works as in \cite[Lemma 24]{DJSmall} since if $c_1 b_1 = c_2 b_2$ with $c_i \in H^+, b_i \in \{H, \alpha\}$, then $c_2^{-1} c_1 = b_2 b_1^{-1} \in H^+$ so $b_2 b_1^{-1} = b_2^{-1} b_1$ and $b_1^2 = b_2^2$, but by assumption $b_i \in \{H, \alpha\} \subseteq F^\circ(H)$ so $b_1$ and $b_2$ differ by an element of $\Omega_2(F^\circ(H))$ (in case of hyperbolic doubt read the next remark). Unlike in Lemma \ref{l:involutiveaction:central2torus} we cannot be too precise about the cardinality of the fiber.
\end{proof}

\begin{remarks*}\
\begin{itemize}
\item
We can show $\{H, \alpha\} \subseteq \Omega_2(F^\circ(H)) \cdot \{F^\circ(H), \alpha\}$. For let $h \in H$; then $[h, \alpha] \in \{H, \alpha\} \subseteq F^\circ(H)$. Applying Lemma \ref{l:involutiveaction:central2torus} in $F^\circ(H)$, we write $[h, \alpha] = f_+ [f, \alpha]$ with $f_+ \in F^\circ(H)^+$ and $f \in F^\circ(H)$. Taking the commutator with $\alpha$ we find $[h, \alpha]^2 = [f, \alpha]^2$. But in $F^\circ(H)$, the equation $x^2 = y^2$ results in $x^{-1} \cdot x^{-1}y \cdot x = y^{-1}x = (x^{-1}y)^{-1}$ and by the first observation in the proof of Lemma \ref{l:involutiveaction:central2torus}, $x^{-1}y$ has order at most $2$. Hence $[h, \alpha] = k [f, \alpha]$ for some $k \in \Omega_2(F^\circ(H))$.
\item
Without the crucial assumption that $\{H, \alpha\} \subseteq F^\circ(H)$ one still has $H = \{H, \alpha\} \cdot (H^+)^\circ \cdot \{H, \alpha\}$ and therefore $H = H^- \cdot H^+ \cdot H^-$, but one can hardly say more.

Consider two copies $A_1 = \{a_1: a \in \C\}$, $A_2 = \{a_2: a \in \C\}$ of $\C_+$ and let $Q = \{t: t\in \C^\times\} \simeq \C^\times$ act on $A_1$ by $a_1^t= (t^2 a)_1$ and on $A_2$ by $a_2^t = (t^{-2}a)_2$. Form the group $H = (A_1\oplus A_2)\rtimes Q$. Let $\alpha$ be the definable, involutive automorphism of $H$ given by:
\[(a_1 b_2 t)^\alpha =  b_1 a_2 t^{-1}\]
that is, ``$\alpha$ swaps the $\pm 2$ weight spaces while inverting the torus''.
The reader may check that $\alpha$ is an automorphism of $H$, and perform the following computations:
\begin{itemize}
\item
$[a_1 b_2t, \alpha] = (t^2 b-t^2 a)_1(t^{-2}a - t^{-2}b)_2 t^{-2}$ (so $\{H, \alpha\} \not\subseteq F^\circ(H)$);
\item
$H^+ = \{a_1 a_2 \cdot \pm 1: a \in \C_+\}$ (incidently $(H^+)^\circ \leq F^\circ(H)$);
\item
$H^- = \{a_1(-t^2 a)_2 t: a\in \C_+, t\in \C^\times\}$ (incidently $H^- = \{H, \alpha\}$);
\item
$H^+ \cdot H^- = \{(a+b)_1 (a-t^2 b)_2\cdot \pm t: a, b \in \C_+, t \in \C^\times\}$ does \emph{not} contain $0_1 a_2 \cdot i$ for $a \neq 0$ (here $i$ is a complex root of $-1$).
\end{itemize}
\item
Rewriting \cite[Theorem 19]{DJSmall} is necessary for the argument; one cannot simply use the idea of Lemma \ref{l:involutiveaction:soluble} together with the original decomposition.

Let $Q = \C^\times$ act on $A = \C_+$ by $a^t = (t^2\cdot a)$ and form $H = A \rtimes Q$. Consider $\alpha$ the involutive automorphism doing $(at)^\alpha = (-a) t$ ($\alpha$ inverts the Fitting subgroup while centralising the Carter subgroup). The reader will check that $H^+ = Q$, $H^- = A \cdot \pm 1$, $\{H, \alpha\} = A$, and of course $H = H^+ \cdot H^-$.

Running the argument in Lemma \ref{l:involutiveaction:soluble} using the (naive) $G = G^+\cdot G^-$ decomposition of \cite[Theorem 19]{DJSmall}, one finds $Q = (Q^+)^\circ \cdot Q^-$, but $Q^- \simeq \Z/2\Z$ is not in $F^\circ(H)$. One could then wish to apply the decomposition to $F(H)$ instead, but the Sylow $2$-subgroup of the latter is not a $2$-torus.

Extending \cite[Theorem 19]{DJSmall} into Lemma \ref{l:involutiveaction:central2torus} was therefore needed for Lemma \ref{l:involutiveaction:soluble}.
\end{itemize}
\end{remarks*}

\subsection{$U_p^\perp$ Actions and Centralisers}

The need for the present subsection is Lemma \ref{l:Upiperpactions:conclusion} below but we shall digress a bit for completeness and future reference.
Let $\up$ denote a set of prime numbers. The class of $U_{\up}^\perp$ groups 
is defined naturally.
\cite[\S I.9.5]{ABCSimple} deals with two dual settings:
\begin{itemize}
\item
soluble, $\up^\perp$ groups acting on definable, connected, soluble, $U_{\up}$ groups;
\item
$\up$-groups acting on definable, connected, soluble, $\up^\perp$ groups.
\end{itemize}
We slightly refine the analysis.

\begin{notation*}\label{n:qoplus}
If $A$ and $B$ are two subgroups of some ambient abelian group, we write $A \qoplus B$ to denote the quasi-direct sum, i.e. in order to mean that $A \cap B$ is finite.
\end{notation*}

\begin{lemma}\label{l:Upiperpactions:abelian}
In a universe of finite Morley rank, let $A$ be a definable, abelian group and $R$ be a group acting on $A$ by definable automorphisms. Let $A_0 \leq A$ be a definable, $R$-invariant subgroup. Suppose one of the following:
\begin{enumerate}[(i)]
\item\label{l:Upiperpactions:abelian:finitepgrouponpperp}
$A$ is a $\up^\perp$ group and $R$ is a finite, soluble $\up$-group;
\item\label{l:Upiperpactions:abelian:infinitepgrouponconnectedpperp}
$A$ is a connected, $\up^\perp$ group, $A_0$ is connected, and $R$ is a soluble $\up$-group;
\item\label{l:Upiperpactions:abelian:infinitepgrouponUpperp}
$A$ is a connected, $U_{\up}^\perp$ group, $A_0$ is connected, and $R$ is a soluble $\up$-group;
\item\label{l:Upiperpactions:abelian:definablepperponUp}
$A$ is a $U_{\up}$-group and $R$ is a definable, soluble, $\up^\perp$ group;
\item\label{l:Upiperpactions:abelian:pperponconnectedUp}
$A$ is a connected $U_{\up}$-group and $R \leq S$ where $S$ is a definable, soluble, $\up^\perp$ group acting on $A$.
\end{enumerate}
%
%
Then $C_R(A) = C_R(A_0, A/A_0)$. In cases \ref{l:Upiperpactions:abelian:finitepgrouponpperp}, \ref{l:Upiperpactions:abelian:infinitepgrouponconnectedpperp}, \ref{l:Upiperpactions:abelian:definablepperponUp} and \ref{l:Upiperpactions:abelian:pperponconnectedUp}:
$A = [A, R] \oplus C_A(R)$, $[A, R]\cap A_0 = [A_0, R]$, and $C_A(R)$ covers $C_{A/A_0}(R)$.
In case \ref{l:Upiperpactions:abelian:infinitepgrouponUpperp}, then the properties hold provided connected components are added (where not redundant), and $\oplus$ is replaced by $\qoplus$.
In case \ref{l:Upiperpactions:abelian:infinitepgrouponconnectedpperp}, then $C_A(R)$ and $C_{A/A_0}(R)$ are connected.
%
\end{lemma}
\begin{proof}\
\begin{enumerate}[(i)]
\item
This is an extension of \cite[Corollary I.9.14]{ABCSimple} taking $A_0$ into account.

We prove that $A = [A, R] + C_A (R)$ by induction on the order of $R$. By solubility, there exist a proper subgroup $S \triangleleft R$ and an element $r \in R$ with $R = \langle  S, r\rangle$. By induction, $A = [A, S] + C_A(S)$. But $r$ normalises $C = C_A(S)$ which is a definable, $\up^\perp$-group. Consider the definable homomorphisms $\ad_r: C \to C$ and $\Tr_r: C\to C$ respectively given by:
\[\ad_r(a) = [a, r], \quad \Tr_r(a) = \sum_{r^i \in \langle  r\rangle} a^{r^i}\]
Since $\ad_r \circ \Tr_r = \Tr_r \circ \ad_r = 0$, one has $\im \ad_r \leq \ker \Tr_r$ and $\im \Tr_r \leq \ker \ad_r$. But since $\ker \Tr_r \cap \ker \ad_r$ consists of elements of order dividing $|r|$, it is trivial by assumption.
In particular $\im \ad_r \cap \ker \ad_r = 0$ so $C = \im \ad_r + \ker \ad_r = [C, r] + C_C(r) \leq [A, R] + C_A(R)$.

Let us show that $[A, R] \cap C_A(R)$ is trivial. Consider the definable homomorphism $\Tr_R: A\to A$ given by:
\[\Tr_R(a) = \sum_{r \in R} a^r\]
Since $\Tr_R$ vanishes on any subgroup of the form $[A, r]$, it vanishes on $[A, R]$; notice that it coincides with multiplication by $|R|$ on $C_A(R)$. It follows that $[A, R] \cap C_A(R)$ consists of elements of order dividing $|R|$, so by assumption it is trivial.

We shall say a bit more: $\ker \Tr_R = [A, R]$ and $\im \Tr_R = C_A(R)$.
Indeed $A = [A, R] + C_A(R)$ and $[A, R] \leq \ker \Tr_R$, so $\ker \Tr_R \leq [A, R] + C_{\ker \Tr_R}(R)$. But $C_{\ker\Tr_R}(R)$ consists of elements of order dividing $|R|$, therefore it is trivial. It follows that $\ker \Tr_R = [A, R]$.
Again $\im \Tr_R \cap \ker \Tr_R \leq C_{\ker \Tr_R}(R) = 0$, so as above $A = \im \Tr_R + \ker \Tr_R$, proving $C_A(R) \leq \im \Tr_R + C_{\ker \Tr_R}(R) = \im \Tr_R$.

We turn our attention to the definable, $R$-invariant subgroup $A_0 \leq A$. One sees that:
\[[A, R] \cap A_0 = \ker \Tr_R \cap A_0 = \ker (\Tr_R)_{|A_0} = [A_0, R]\]
and letting $\varphi$ stand for projection modulo $A_0$:
\[\varphi(C_A(R)) = \varphi \circ \Tr_R(A) = \Tr_R\circ  \varphi(A) = \Tr_R(A/A_0) = C_{A/A_0} (R)\]
Finally let $S = C_R(A_0, A/A_0)$. We apply our results to the action of $S$ on $A$ and find $A \leq [A, S] + C_A(S) \leq C_A(S)$ so $S = C_R(A)$.
\item
We reduce to case \ref{l:Upiperpactions:abelian:finitepgrouponpperp} with the following claim.

In a universe of finite Morley rank, if $G$ is a definable, connected group and $R$ is a locally finite group acting on $G$, then there is a finite subgroup $R_0 \leq R$ with $C_G(R_0) = C_G(R)$ and $[G, R_0] = [G, R]$.

The first is by the descending chain condition on centralisers: there is a finite subset $X \subseteq R$ with $C_G(X) = C_G(R)$. Now by connectedness of $G$ and Zilber's indecomposibility theorem, $[G, r]$ is definable and connected for any $r \in R$. By the ascending chain condition on definable, connected subgroups, there is a finite subset $Y \subseteq R$ such that $[G, Y] = [G, R]$. Take $R_0 = \langle   X \cup Y\rangle$, a finite subgroup of $R$.

So taking both actions on $A$ and on $A_0$ into account we may suppose $R$ to be finite; apply case \ref{l:Upiperpactions:abelian:finitepgrouponpperp} and see that $A = [A, R] \oplus C_A(R)$ implies connectedness of the latter.
\item
Here again we may suppose $R$ to be finite. Now read the proof of case \ref{l:Upiperpactions:abelian:finitepgrouponpperp} again, replacing ``trivial'' by ``finite'' and adding connected components where necessary.
\item
This is esentially \cite[Facts 1.15 and 1.16]{CDSmall}; also see \cite[Corollary I.9.11]{ABCSimple}.

Let $H = A \rtimes R$, a definable, soluble group with $A \leq F(H)$. Then for $q \notin \up$, $U_q(R) \leq F(H) \leq C_H(A)$ and likewise, $U_{(0, k)}(R) \leq C_H(A)$ for $k > 0$. So $R^\circ$ acts as a good torus which we may replace with a finite, normal subgroup of $R$; then we may suppose that $R$ itself is finite.

Considering the complement of $\up$ in the set of primes, we may apply case \ref{l:Upiperpactions:abelian:finitepgrouponpperp}.
\item
We reduce to case \ref{l:Upiperpactions:abelian:definablepperponUp} with the following claim.

In a universe of finite Morley rank, if $G$ is a definable, connected group and $S$ is a definable group acting on $G$, then any subgroup $R\leq S$ satisfies $C_G(R) = C_G(d(R))$ and $[G, R] = [G, d(R)]$.

The first is by definability of centralisers. The second is as in \cite[Lemma 1.14]{CDSmall}: let $X = \{s \in d(R): [G, s] \leq [G, R]\}$.
Since $[G, R]$ is definable by connectedness of $G$ and Zilber's indecomposibility theorem, so is its normaliser in $d(R)$. Hence $d(R)$ normalises $[G, R]$; the definable set $X$ is actually a subgroup of $d(R)$. So $d(R) \leq d(X)$ and $[G, d(R)] = [G, R]$.
\qedhere
\end{enumerate}
\end{proof}

\begin{remark*}
The Lemma does not hold for $U_{\up}^\perp$, non-connected $A$ since it fails at the finite level: let $R = \Z/2\Z$ act by inversion on $A = \Z/4\Z$; one has $C_A(R) = 2A = [A, R]$.
\end{remark*}

After obtaining the following Lemma the author realised it was already proved by Burdges and Cherlin using a different argument.

\begin{lemma}[{cf. \cite[Proposition I.9.12]{ABCSimple}; also \cite[Lemma 2.5]{BCGeneration}}]\label{l:Upiperpactions:centralisers}
In a universe of finite Morley rank, let $G$ be a definable group, $R$ be a soluble $\up$-group acting on $G$ by definable automorphisms, and let $H \trianglelefteq G$ be a definable, connected, soluble, $U_{\up}^\perp$, $R$-invariant subgroup. Then $C_{G/H}^\circ(R) = C_G^\circ(R)H/H$.
\end{lemma}
\begin{proof}
As in Lemma \ref{l:Upiperpactions:abelian}, using chain conditions and local finiteness, we may assume that $R$ is finite. Let $L = \varphi^{-1}(C_{G/H}^\circ(R))$, where $\varphi$ denotes projection modulo $H$. Since $\varphi$ is surjective, $\varphi(L) = C_{G/H}^\circ(R)$ which is connected and a finite extension of $\varphi(L^\circ)$: so $\varphi(L) = \varphi(L^\circ)$ and $L = L^\circ H = L^\circ$ by connectedness of $H$. Hence $L$ itself is connected.
We now proceed by induction on the solubility class of $H$.

First suppose that $H$ is abelian; we proceed by induction on the solubility class of $R$. \begin{itemize}
\item
First suppose that $R = \langle  r\rangle$. Be careful that the definable map $\Tr_r: G\to G$ given by:
\[\Tr_r (g) = \prod_{i = 0}^{|r|-1}g^{r^i}\]
is \emph{not} a group homomorphism, but $(\Tr_r)_{|H}$ is one.

Since $[L, r] \leq H \cap \Tr_r^{-1}(0) = \ker (\Tr_r)_{|H}$, one has by connectedness and Zilber's indecomposibility theorem $[L, r] \leq \ker^\circ (\Tr_r)_{|H} = [H, r]$ by the proof of Lemma \ref{l:Upiperpactions:abelian}. Bear in mind that $H$ is abelian; it follows that $L \leq H C_G(r)$, so by connectedness $L \leq H C_G^\circ(r)$, as desired.
\item
Now suppose $R = \langle  S, r\rangle$ with $S\triangleleft R$. By induction, $L \leq H C_G^\circ(S)$ and since $H \leq L$, one has $L \leq H C_L^\circ(S)$. Let $G_S = C_G^\circ(S)$ and $H_S = C_H^\circ(S)$; also let $\varphi_S$ be the projection $G_S \to G_S/H_S$, and $L_S = \varphi_S^{-1}(C_{G_S/H_S}^\circ(r))$.

By the cyclic case, $L_S \leq H_S C_{G_S}^\circ(r) \leq H C_G^\circ(R)$. But $[C_L^\circ(S), r] \leq H \cap C_G^\circ(S)$ so by connectedness $[C_L^\circ(S), r] \leq C_H^\circ(S) = H_S$. It follows that $C_L^\circ(S) \leq L_S \leq H C_G^\circ(R)$ and $L \leq H C_L^\circ(S) \leq H C_G^\circ(R)$.
\end{itemize}
We now let $K = H'$, which is a definable, connected, $R$-invariant subgroup normal in $G$. Let $\varphi_K: G\to G/K$ and $\psi: G/K \to G/H$ be the standard projections, so that $\varphi = \psi \varphi_K$. By induction, $\varphi_K(C_G^\circ(R)) = C_{\varphi_K(G)}^\circ(R)$. But $\varphi_K(H) \trianglelefteq \varphi_K(G)$ and $\varphi_K(H)$ is abelian, so by the abelian case we just covered, $\psi(C_{\varphi_K(G)}^\circ(R)) = C_{\psi\varphi_K(G)}^\circ(R)$. Therefore:
\[\varphi(C_G^\circ(R)) = \psi(\varphi_K(C_G^\circ(R)) = \psi(C_{\varphi_K(G)}^\circ(R)) = C_{\psi\varphi_K(G)}^\circ(R) = C_{\varphi(G)}^\circ(R)\qedhere\]
\end{proof}

The following inductive consequence will not be used in the present work.

\begin{lemma}[{cf. \cite[Proposition I.9.13]{ABCSimple}}]
In a universe of finite Morley rank, let $H$ be a definable, connected, soluble, $U_{\up}^\perp$ group and $R$ be a soluble $\up$-group acting on $H$ by definable automorphisms.
Then $H = [H, R] C_H^\circ(R)$.
\end{lemma}

Now let $\rho$ denote a unipotence parameter. We wish to generalise \cite[Lemma 3.6]{BSignalizer} relaxing the $\up^\perp$ assumption to $U_{\up}^\perp$. This will considerably simplify some arguments; in particular we shall no longer care whether Burdges' unipotent radicals of Borel subgroups contain involutions or not when taking centralisers. This will spare us the contortions of \cite[Lemmes 5.2.33, 5.2.39, 5.3.20, 5.3.23]{DGroupes}.

\begin{lemma}[{cf. \cite[Lemma 3.6]{BSignalizer}}]\label{l:Upiperpactions:conclusion}
In a universe of finite Morley rank, let $U$ be a definable, $U_{\up}^\perp$, $\rho$-group and $R$ be a soluble $\up$-group acting on $U$ by definable automorphisms. Then $C_U^\circ(R)$ is a $\rho$-group.
\end{lemma}
\begin{proof}
The proof is by induction on the nilpotence class of $U$. First suppose that $U$ is abelian. Then by Lemma \ref{l:Upiperpactions:abelian} one has $U = [U, R] \qoplus C_U^\circ(R)$. Let $K$ stand for the finite intersection. Then $C_U^\circ(R)/K \simeq U/[U, R]$ which by push-forward \cite[Lemma 2.11]{BSignalizer} is a $\rho$-group. It follows that $C_U^\circ(R)$ itself is a $\rho$-group. (Since we could not locate a proof of this trivial fact in the literature, here it goes: let $V = C_U^\circ(R)$ and $\varphi: V \to V/K$ be the standard projection. By pull-back \cite[Lemma 2.11]{BSignalizer}, $\varphi(U_\rho(V)) = V/K = \varphi(V)$, and since $\ker \varphi$ is finite, $\rk U_\rho(V) = \rk V$. By connectedness, $V = U_\rho(V)$.)

Now let $1 < A \triangleleft U$ be an abelian definable, connected, characteristic subgroup. By induction, $C_A^\circ(R)$ and $C_{U/A}^\circ(R)$ are $\rho$-groups. Now by Lemma \ref{l:Upiperpactions:centralisers},
\begin{align*}
C_{U/A}^\circ(R) & \simeq C_U^\circ(R) A/A\\
& \simeq C_U^\circ(R)/(A\cap C_U^\circ(R))\\
& \simeq \left(C_U^\circ(R)/C_A^\circ(R)\right)/\left((A\cap C_U^\circ(R))/C_A^\circ(R)\right)\\
& = \left(C_U^\circ(R)/C_A^\circ(R)\right)/L\end{align*}
where $L = (A\cap C_U^\circ(R))/C_A^\circ(R)$ is finite. Since $C_{U/A}^\circ(R)$ is a $\rho$-group, so is $C_U^\circ(R)/C_A^\circ(R)$. But $C_A^\circ(R)$ is a $\rho$-group, so by pull-back, so is $C_U^\circ(R)$.
\end{proof}


One could of course do the same with a set of unipotence parameters instead of a single parameter $\rho$.

\begin{remark*}
As opposed to the usual setting of $\up^\perp$ groups \cite[Lemma 3.6]{BSignalizer}, connectedness of $C_U(R)$ is not granted in the $U_{\up}^\perp$ case: think of an involutive automorphism inverting a $\rho$-group which contains a non-trivial $2$-torus.
\end{remark*}

As a consequence, if inside a group of odd type some involution $i$ acts on a $\sigma$-group $H$  with $\rho_{C(i)} \prec \sigma$, then $i$ inverts $H$. We shall use this fact with no reference.


\subsection{Carter $\pi$-Subgroups}

The maybe not-so-familiar notion of a Carter $\pi$-subgroup was recalled in \S\ref{s:unipotence}. Bear in mind that by definition, $\pi$-groups are nilpotent.

\begin{lemma}\label{l:piCarter}
Let $H$ be a connected, soluble group of finite Morley rank, $\pi$ be a set of parameters such that $U_\pi(H') = 1$, and $L \leq H$ be a maximal $\pi$-subgroup. Then there is a Carter subgroup $Q \leq H$ of $H$ with $L = U_\pi(Q)$.
\end{lemma}
\begin{proof}
It suffices to show that for any $\pi$-subgroup $L \leq H$ there is a Carter subgroup $Q$ of $H$ with $L \leq Q$.

If $|\pi| = 1$ then we are actually dealing with a single unipotence parameter $\rho$, and the result follows from the theory of Sylow $\rho$-subgroups (\cite[Lemma 4.19]{BSimple}, \cite[Theorem 5.7]{BSylow}).
If $|\pi| > 1$, write Burdges' decomposition of $L = L_\rho \ast M$, where $\rho$ is any unipotence parameter occurring in $L$, $L_\rho = U_\rho(L)$, and $M$ is a $(\pi\setminus\{\rho\})$-group. By induction there is a Carter subgroup $Q$ of $H$ with $L_\rho \leq Q$.

Now consider the generalised centraliser (a tool we already used in the proof of Lemma \ref{l:snakes:soluble:Carter}) $E = E_H(L_\rho) \geq \langle  Q, M\rangle$.
If $E < H$ then by induction on the Morley rank $L$ is contained in some Carter subgroup of $E$. Since $Q \leq E$, the former also is a Carter subgroup of $H$.

So we may assume $E = H$, and therefore $L_\rho \leq F^\circ(H)$ \cite[Corollaire 5.17]{FSous}. Actually we may assume this for any parameter $\rho$, meaning $L \leq F^\circ(H)$. Now $Q$ acts on $U_\pi(F^\circ(H))$ so $[Q, U_\pi(F^\circ(H))] \leq U_\pi(H') = 1$ and $L \leq U_\pi(F^\circ(H)) \leq N_H(Q) = Q$.
\end{proof}

\subsection{$W_p^\perp$ Groups}

Weyl groups of minimal connected simple groups have been abundantly discussed \cite{ABAnalogies, BCSemisimple, BDWeyl, FAutomorphisms}.
We do not feel utterly interested now; as a consequence we shall not even define Weyl groups.
Instead we shall develop a more limited view which will suffice for our purposes. This line is very much in the spirit of \cite{BPTores}, the influence of which on later work should not be concealed.

\begin{notation*}
Let $G$ be a $U_p^\perp$ group of finite Morley rank. Let $W_p(G) = S/S^\circ$ for any Sylow $p$-subgroup $S$ of $G$ (these are conjugate by \cite[Theorem 4]{BCSemisimple}, our Fact \ref{f:Sylowconjugate}, so this is well-defined).
\end{notation*}

\begin{lemma}\label{l:W2perp:factor}
Let $G$ be a $U_p^\perp$ group of finite Morley rank.
\begin{enumerate}[(i)]
\item\label{l:W2perp:factor:Hconnected}
If $H \leq G$ is a definable, connected subgroup, then $W_p(H) \hookrightarrow W_p(G)$.
\item\label{l:W2perp:factor:Hnormal}
If $H \trianglelefteq G$ is a definable, normal subgroup, then $W_p(G) \twoheadrightarrow W_p(G/H)$.
\item\label{l:W2perp:factor:Hconnectednormal}
If $H \trianglelefteq G$ is a definable, connected, normal subgroup, then $W_p(G/H) \simeq W_p(G)/W_p(H)$.
\item\label{l:W2perp:factor:Hcentral}
If $G$ is connected and $H \leq Z(G)$ is a central subgroup, then $W_p(G/H) \simeq W_p(G)$.
\end{enumerate}
\end{lemma}
\begin{proof}\
\begin{enumerate}[(i)]
\item
Let $S_H$ be a Sylow $p$-subgroup of $H$ and extend it to a Sylow $p$-subgroup $S_G$ of $G$. To $w\in W_p(H)$ associate $h S_G^\circ \in W_p(G)$ where $h \in S_H$ is such that $h S_H^\circ = w$. This is a well-defined group homomorphism as $S_H^\circ \leq S_G^\circ$. It is injective since if $h \in S_H\cap S_G^\circ$, then $h \in C_{S_H}(S_H^\circ) = S_H^\circ$ by torality principles and connectedness of $H$.
\item
Let $S_H \leq S_G$ be as above and denote projection modulo $H$ by $\overline{\phantom{G}}$; we know that $\Sigma = \overline{S_G} \simeq S_G/S_H$ is a Sylow $p$-subgroup of $G/H$.
To $w \in W_p(G)$ associate $\overline{g} \Sigma^\circ \in W_p(G/H)$ where $g \in S_G$ is such that $g S_G^\circ = w$. This is a well-defined group homomorphism as $\overline{S_G^\circ} = \Sigma^\circ$. It is clearly surjective.
\item
Suppose in addition that $H$ is connected. With notations as in the argument for Claim \ref{l:W2perp:factor:Hnormal}, if $w$ is in the kernel then $g \in S_G^\circ H$, and we may suppose $g \in H$ (the converse is obvious). Hence the kernel coincides with the image of $W_p(H)$ in $W_p(G)$ given by Claim \ref{l:W2perp:factor:Hconnected}.
\item
By Claim \ref{l:W2perp:factor:Hnormal} the map $W_p(G) \to W_p(G/H)$ is a surjective group homomorphism; now if $g S_G^\circ \in W_p(G)$ lies in the kernel, since $H$ is central in $G$ one finds $g \in S_G\cap (H S_G^\circ) \leq C_{S_G}(S_G^\circ) = S_G^\circ$ by torality principles and connectedness of $G$. So the map is injective and $W_p(G) \simeq W_p(G/H)$.
\qedhere
\end{enumerate}
\end{proof}

\begin{remarks*}\
\begin{itemize}
\item
In Claims \ref{l:W2perp:factor:Hconnected} and \ref{l:W2perp:factor:Hconnectednormal}, connectedness of $H$ is necessary: consider $\Z/2\Z$ inside $\Z_{2^\infty}$, then inside $\SL_2(\C)$.
\item
As a consequence, if $G$ is connected and $H \trianglelefteq G$ is a definable, normal subgroup, then $W_p(G/H) \simeq W_p((G/H^\circ)/(H/H^\circ)) \simeq W_p(G/H^\circ) \simeq W_p(G)/W_p(H^\circ)$.
\item
Lemma \ref{l:W2perp:factor} could be used as a qualifying test for tentative notions of the Weyl group.
\end{itemize}
\end{remarks*}

We wish to suggest a bit of terminology.

\begin{definition*}
A $U_p^\perp$ group of finite Morley rank is $W_p^\perp$ if its Sylow $p$-subgroups are connected.
\end{definition*}

As a consequence of Lemma \ref{l:W2perp:factor}, when $H \trianglelefteq G$ where both are definable and connected, if $H$ and $G/H$ are $W_p^\perp$ then so is $G$.
We aim at saying a bit more about extending tori. The following result is not used anywhere in the present article.

\begin{lemma}\label{l:W2perp:extension}
Let $\hat{G}$ be a connected, $U_p^\perp$ group of finite Morley rank and $G \trianglelefteq \hat{G}$ be a definable, connected subgroup. Suppose that $\hat{G}/G$ is $W_p^\perp$. Let $\hat{S} \leq \hat{G}$ be a Sylow $p$-subgroup and $S = \hat{S} \cap G$.
Then there exist:
\begin{itemize}
\item
a $p$-torus $\hat{T} \leq \hat{G}$ with $\hat{S} = S \rtimes \hat{T}$ (semidirect product);
\item
a $p$-torus $\hat{\Theta} \leq \hat{G}$ with $\hat{S} = S \qtimes \hat{\Theta}$ (central product over a finite intersection).
\end{itemize}
\end{lemma}
\begin{proof}
We know that $S$ is a Sylow $p$-subgroup of $G$ and that $\hat{S}/S \simeq \hat{S}G/G$ is a Sylow $p$-subgroup of $\hat{G}/G$; as the latter is $W_p^\perp$ it is a $p$-torus.
In particular $\hat{S} = \hat{S}^\circ S$. Note that $S \cap \hat{S}^\circ \leq C_{S}(S^\circ) = S^\circ$ by torality principles and connectedness of $G$.

Bear in mind that $p$-tori are injective as $\Z$-modules. Inside $\hat{S}^\circ$ take a direct complement $\hat{T}$ of $S^\circ$, so that $\hat{S}^\circ = S^\circ \oplus \hat{T}$. Then $\hat{S} = S \hat{S}^\circ = S \hat{T}$, but $S \cap \hat{T} \leq S \cap \hat{S}^\circ \cap \hat{T} \leq S^\circ \cap \hat{T} = 1$. Hence $\hat{S} = S \rtimes \hat{T}$.

We now consider the action of $\hat{S}$ on $\hat{S}^\circ$; observe that $\hat{S}$ as a pure group has finite Morley rank, so Lemma \ref{l:Upiperpactions:abelian} applies and yields $\hat{S}^\circ = [\hat{S}^\circ, \hat{S}] \qoplus C_{\hat{S}^\circ}^\circ(\hat{S})$. Since $\hat{S}/S$ is a $p$-torus, it is abelian, so $[\hat{S}^\circ, \hat{S}] \leq \hat{S}' \leq S$, and by Zilber's indecomposibility theorem $[\hat{S}^\circ, \hat{S}] \leq S^\circ$.
Inside $C_{\hat{S}^\circ}^\circ(\hat{S})$ take a direct complement $\hat{\Theta}$ of $C_{S^\circ}^\circ(\hat{S})$, so that $C_{\hat{S}^\circ}^\circ(\hat{S}) = C_{S^\circ}^\circ(\hat{S}) \oplus \hat{\Theta}$. Then $\hat{S} = S \hat{S}^\circ = S C_{\hat{S}^\circ}^\circ(\hat{S}) = S \hat{\Theta}$, and $\hat{\Theta} \leq C_{\hat{S}^\circ}^\circ(\hat{S})$ commutes with $S$. Moreover $(S \cap \hat{\Theta})^\circ \leq (C_S(\hat{S}) \cap \hat{\Theta})^\circ \leq C_{S^\circ}^\circ(\hat{S}) \cap \hat{\Theta} = 1$ by construction, so $\hat{S} = S \qtimes \hat{\Theta}$.
\end{proof}

\begin{remark*}
One may not demand that $\hat{S} = S \times \hat{T}$ (direct product). 
Consider the two groups $\SL_2(\C)$ with involution $i$ and $\C^\times$ with involution $j$. Let $\hat{G} = (\SL_2(\C) \times \C^\times)/\langle  ij\rangle$ and $\varphi: \SL_2(\C) \times \C^\times \to \hat{G}$ be the standard projection.
Let $G = \varphi(\SL_2(\C))\simeq \SL_2(\C)$ and $\hat{\Theta} = \varphi(\C^\times) \simeq \C^\times$.
Fix any Sylow $2$-subgroup $\hat{S}$ of $\hat{G}$. Then with $S = \hat{S} \cap G$ one has $S \hat{\Theta} = S \qtimes \hat{\Theta} = \hat{S}$, and $S \cap \hat{\Theta} = \langle  \varphi(i)\rangle$.

If one asks for a semidirect complement $\hat{T}$, the latter must contain its own involution, which will be $\varphi(ab)$ (or possibly $\varphi(iab)$, a similar case), where $a \in \varphi^{-1}(S)\leq \SL_2(\C)$ satisfies $a^2 = i$ and $b^2 = j$ in $\C^\times$.
Remember that inside a fixed Sylow $2$-subgroup of $\SL_2(\C)$, every element of order four (be it toral inside the fixed Sylow $2$-subgroup or not) is inverted by another element of order four.
So let $\zeta \in \varphi^{-1}(S)$ invert $a$. Then:
\[\varphi(\zeta^{ab}) = \varphi(\zeta^a) = \varphi(i \zeta) \neq \varphi(\zeta)\]
so the action of $\hat{T}$ on $S$ is always non-trivial.

One may not demand $\hat{S} = S \times \hat{T}$, and in any case nothing can apparently prevent $d(\hat{T})$ from intersecting $G$ non-trivially, so the question is rather pointless.
\end{remark*}

\subsection{A Counting Lemma}

The following quite elementary Lemma was devised in Cappadocia in 2007 as an explanation of \cite[Corollaire 5.1.7]{DGroupes} (or \cite[Corollaire 4.7]{DGroupes2}). It will be used only once.

\begin{lemma}[Göreme]\label{l:Goreme}
Let $G$ be a connected, $U_2^\perp$, $W_2^\perp$ group of finite Morley rank. Then the number of conjugacy classes of involutions is odd (or zero).
\end{lemma}
\begin{proof}
By torality principles, every class is represented in a fixed Sylow $2$-subgroup $S = S^\circ$. We group involutions of $S^\circ$ by classes $\gamma_k$, and assume we find an even number of these: $I(S^\circ) = \sqcup_{k = 1}^{2m} \gamma_k$. Since the number of involutions in $S^\circ$ is however odd, some class, say $\gamma$, has an even number of involutions. Now $N = N_G(S)$ acts on $\gamma$; by definition of a conjugacy class and by a classical fusion control argument \cite[Lemma 10.22]{BNGroups}, $N$ acts transitively on $\gamma$. Hence $[N: C_N(\gamma)] = |\gamma|$ is even. Lifting torsion, there is a non-trivial $2$-element $\zeta$ in $N \setminus C_N(\gamma)$. Since $S \trianglelefteq N$, one has $\zeta \in S = S^\circ \leq C_N(\gamma)$, a contradiction.
\end{proof}

The author hoped to be able to use this Lemma without any form of bound on the Prüfer $2$-rank. He failed as one shall see in Step~\ref{t:st:PrhatG=1} of the Theorem. The general statement remains as a relic of happier times past.

\section{The Proof --- Before the Maximality Proposition}\label{S:before}

\begin{theorem*}
Let $\hat{G}$ be a connected, $U_2^\perp$ group of finite Morley rank and $G \trianglelefteq \hat{G}$ be a definable, connected, non-soluble, $N_\circ^\circ$-subgroup.

Then the Sylow $2$-subgroup of $G$ has one of the following structures: isomorphic to that of $\PSL_2(\C)$, isomorphic to that of $\SL_2(\C)$, a $2$-torus of Prüfer $2$-rank at most $2$.

Suppose in addition that for all involutions $\iota \in I(\hat{G})$, the group $C_G^\circ(\iota)$ is soluble.

Then $m_2(\hat{G}) \leq 2$, one of $G$ or $\hat{G}/G$ is $2^\perp$, and involutions are conjugate in $\hat{G}$. Moreover one of the following cases occurs:
\begin{description}
\item[$\bullet$ PSL$_2$:]
$G \simeq \PSL_2(\K)$ in characteristic not $2$; $\hat{G}/G$ is $2^\perp$;
\item[$\bullet$ CiBo$_\emptyset$:]
$G$ is $2^\perp$; $m_2(\hat{G}) \leq 1$; for $\iota \in I(\hat{G})$, $C_G(\iota) = C_G^\circ(\iota)$ is a self-normalising Borel subgroup of $G$;
\item[$\bullet$ CiBo$_1$:]
$m_2(G) = m_2(\hat{G}) = 1$; $\hat{G}/G$ is $2^\perp$; for $i \in I(\hat{G}) = I(G)$, $C_G(i) = C_G^\circ(i)$ is a self-normalising Borel subgroup of $G$;
\item[$\bullet$ CiBo$_2$:]
$\Pr_2(G) = 1$ and $m_2(G) = m_2(\hat{G}) = 2$; $\hat{G}/G$ is $2^\perp$; for $i \in I(\hat{G}) = I(G)$, $C_G^\circ(i)$ is an abelian Borel subgroup of $G$ inverted by any involution in $C_G(i)\setminus\{i\}$ and satisfies $\rk G = 3 \rk C_G^\circ(i)$;
\item[$\bullet$ CiBo$_3$:]
$\Pr_2(G) = m_2(G) = m_2(\hat{G}) = 2$; $\hat{G}/G$ is $2^\perp$; for $i \in I(\hat{G}) = I(G)$, $C_G(i) = C_G^\circ(i)$ is a self-normalising Borel subgroup of $G$; if $i \neq j$ are two involutions of $G$ then $C_G(i) \neq C_G(j)$.
\end{description}
\end{theorem*}

The proof requires eight propositions all strongly relying on the $N_\circ^\circ$ assumption, the deepest of which will be the maximality Proposition~\ref{p:maximality}.
Let us briefly describe the global outline. More detailed information will be found before each proposition.

In Proposition~\ref{p:2structure} (\S\ref{s:2structure}) we determine the $2$-structure of $N_\circ^\circ$-groups by elementary methods. Proposition~\ref{p:genericity} (\S\ref{s:genericity}) is a classical rank computation required both by the Algebraicity Proposition~\ref{p:algebraicity} (\S\ref{s:algebraicity}) which identifies $\PSL_2(\K)$ through reconstruction of its BN-pair, and by the Maximality Proposition~\ref{p:maximality} which shows that in non-algebraic configurations centralisers$^\circ$ of involutions are Borel subgroups.
The proof may be of interest to the expert in finite group theory; perhaps he will find something unexpected there.
Proposition~\ref{p:maximality} will take all of \S\ref{S:maximality} but actually requires two more technical preliminaries: Propositions \ref{p:DevilsLadder} (\S\ref{s:DevilsLadder}) and \ref{p:Yanartas} (\S\ref{s:Yanartas}), which deal with actions of involutions and torsion, respectively. After Proposition~\ref{p:maximality} things go faster. We study the action of an infinite dihedral group in Proposition~\ref{p:dihedral} (\S\ref{s:dihedral}) and a strong embedding configuration in Proposition~\ref{p:strongembedding} (\S\ref{s:strongembedding}). Both are rather classical, methodologically speaking; Proposition~\ref{p:dihedral} is more involved than Proposition~\ref{p:strongembedding}; they can be read in any order but both rely on Maximality. The final assembling takes place in \S\ref{s:theorem} where all preliminary propositions  \ref{p:2structure}, \ref{p:genericity}, \ref{p:DevilsLadder} and \ref{p:Yanartas} reappear as independent themes.

The resulting architecture surprised the author. In the original, minimal connected simple setting one proceeded by first bounding the Prüfer $2$-rank \cite{BCJMinimal} and then studying the remaining cases \cite{DGroupes1, DGroupes2}. There maximality propositions had to be proved three times in order to complete the analysis. The reason for such a clumsy treatment, with one part of the proof being repeated over and over again, was that torsion arguments were systematically based on some control on involutions.
Here we do the opposite. By providing careful torsion control in Proposition~\ref{p:Yanartas} and relaxing our expectations on conjugacy classes of involutions we shall be able to run maximality without prior knowledge of the Prüfer $2$-rank. This seems to be the right level both of elegance and generality. Bounding the Prüfer $2$-rank then follows by adapting a small part of \cite{BCJMinimal}.

Before the curtain opens one should note that bounding the Prüfer $2$-rank of $\hat{G}$ a priori is possible if one assumes $G$ to be $2^\perp$ as Burdges noted for \cite{BCDAutomorphisms}. We do not follow this line.

\subsection{The $2$-Structure Proposition}\label{s:2structure}

Proposition~\ref{p:2structure} hereafter comes directly from \cite[Chapitre 4 and Addendum]{DGroupes}, published as \cite[\S 2]{DGroupes2}. It is the most elementary of our propositions, and together with the Strong Embedding Proposition~\ref{p:strongembedding} one of the two not requiring almost-solubility of centralisers of involutions.

\begin{proposition}[$2$-Structure]\label{p:2structure}
Let $G$ be a connected, $U_2^\perp$, $N_\circ^\circ$-group of finite Morley rank. Then the Sylow $2$-subgroup of $G$ has the following form:
\begin{itemize}
\item
connected, i.e. a possibly trivial $2$-torus;
\item
isomorphic to that of $\PSL_2(\C)$;
\item
isomorphic to that of $\SL_2(\C)$, in which case $C_G^\circ(i)$ is non-soluble for any involution $i$ of $G$.
\end{itemize}
\end{proposition}
\begin{proof}
If the Prüfer rank is $0$ this is a consequence of the analysis of degenerate type groups \cite{BBCInvolutions}. If it is $1$, this is well-known, see for reference \cite[Proposition 27]{DJSmall}. Notice that if the Sylow $2$-subgroup is as in $\SL_2(\C)$ and $i$ is any involution, then by torality principles all Sylow $2$-subgroups of $C_G(i)$ are in $C_G^\circ(i)$, but none is connected: this, and the structure of torsion in connected, soluble groups of finite Morley rank prevents $C_G^\circ(i)$ from being soluble.

So we suppose that the Prüfer $2$-rank is at least $2$ and show that a Sylow $2$-subgroup $S$ of $G$ is connected. Let $G$ be a minimal counterexample to this statement. Then $G$ is non-soluble. Since $G$ is an $N_\circ^\circ$-group, $Z(G)$ is finite, but we actually may suppose that $G$ is centreless. For if the result holds of $G/Z(G)$, then $SZ(G)/Z(G)$ is a Sylow $2$-subgroup of $G/Z(G)$, and therefore connected, so that $S \leq S^\circ Z(G) \cap S \leq C_S(S^\circ) = S^\circ$ by torality principles. So we may assume $Z(G) = 1$.

Still assuming that the Prüfer $2$-rank is at least $2$ we let $\zeta \in S\setminus S^\circ$ have minimal order, so that $\zeta^2 \in S^\circ$. Let $\Theta_1 = C_{S^\circ}^\circ(\zeta)$. If $\Theta_1 \neq 1$ then $\langle  S^\circ, \zeta\rangle \leq C_G(\Theta_1)$ which is connected by \cite[Theorem 1]{ABAnalogies} and soluble since $G$ is an $N_\circ^\circ$-group. The structure of torsion in such groups yields $\zeta \in S^\circ$, a contradiction. So $\Theta_1 = C_{S^\circ}^\circ(\zeta) = 1$ and $\zeta$ therefore inverts $S^\circ$. In particular $\zeta$ centralises the group $\Omega = \Omega_2(S^\circ)$ generated by involutions of $S^\circ$, and $\Omega$ normalises $C_G^\circ(\zeta)$.
By normalisation principles $\Omega$ normalises a maximal $2$-torus $T$ of $C_G^\circ(\zeta)$; by torality principles, $\zeta \in T$ and $T$ has the same Prüfer $2$-rank as $S$. Now $|\Omega|\geq 4$ so there is $i \in \Omega$ such that $\Theta_2 = C_T^\circ(i)$ is non-trivial. Then $\langle  T, i\rangle \leq C_G(\Theta_2)$ which is soluble and connected as above, implying $i \in T$. This is not a contradiction yet, but now $\zeta \in T \leq C_G^\circ(i)$ and of course $S^\circ \leq C_G^\circ(i)$. Hence $C_G^\circ(i) < G$ is a smaller counterexample, a contradiction. Connectedness is proved.
\end{proof}

\begin{remark*}
One can show that if $\alpha \in G$ is a $2$-element with $\alpha^2 \neq 1$, then $C_G(\alpha)$ is connected.

For let $\alpha \in G$ have order $2^k$ with $k > 1$. By Steinberg's torsion theorem (our Fact \ref{f:Steinberg}), $C_G(\alpha)/C_G^\circ(\alpha)$ has exponent dividing $2^k$. Using torality principles, fix a maximal $2$-torus $T$ of $G$ containing $\alpha$.
If the Sylow $2$-subgroup of $G$ is connected, then $T$ is a Sylow $2$-subgroup of $G$ included in $C_G^\circ(\alpha)$: hence $C_G(\alpha) = C_G^\circ(\alpha)$.
If the Sylow $2$-subgroup of $G$ is isomorphic to that of $\PSL_2(\C)$ or to that of $\SL_2(\C)$, then any $2$-element $\zeta \in C_G(\alpha)$ normalising $T$ centralises $\alpha$ of order at least $4$, so it also centralises $T$. It follows from torality principles that $\zeta \in T \leq C_G^\circ(\alpha)$, and $C_G(\alpha)$ is connected again.

We shall not use this remark.
\end{remark*}

\subsection{The Genericity Proposition}\label{s:genericity}

\begin{quote}
\itshape
Considerations concerning the distribution of involutions in the cosets of a given subgroup are often useful in the study of groups of even order.
\end{quote}

So wrote Bender in the beginning of \cite{BFiniteLarge}. The first instance of this method in the finite Morley rank context seems to be \cite[after Lemma 7]{BDNCIT} which with \cite{BNCIT2} aimed at identifying $\SL_2(\K)$ in characteristic $2$. Jaligot brought it to the odd type setting \cite{JFT}. The present subsection is the cornerstone of Propositions \ref{p:algebraicity} and \ref{p:maximality} and is used again when conjugating involutions in Step~\ref{t:st:conjugacy} of the final argument. We introduce subsets of a group $H$ describing the distribution of involutions in the translates of $H$.

\begin{notation*}
For $\kappa$ an involutive automorphism and $H$ a subgroup of some ambient group, we let $T_H(\kappa) = \{h \in H : h^\kappa = h^{-1}\}$. (This set is definable as soon as $\kappa$ and $H$ are.)
\end{notation*}

The following is completely classical; the proof will not surprise the experts and is included for the sake of self-containedness. It will be applied only when $H$ is a Borel subgroup of $G$.

\begin{proposition}[Genericity]\label{p:genericity}
Let $\hat{G}$ be a connected, $U_2^\perp$ group of finite Morley rank and $G \trianglelefteq \hat{G}$ be a definable, connected, non-soluble, $N_\circ^\circ$-subgroup.

Suppose that $\hat{G} = G\cdot d(\hat{S}^\circ)$ for some maximal $2$-torus $\hat{S}^\circ$ of $\hat{G}$.

Let $\iota \in I(\hat{G})$ and $H \leq G$ be a definable, infinite, soluble subgroup of $G$. Then $K_H = \{\kappa \in \iota^{\hat{G}} \setminus N_{\hat{G}}(H) : \rk T_H(\kappa) \geq \rk H - \rk C_G(\iota)\}$ is generic in $\iota^{\hat{G}}$.
\end{proposition}
\begin{proof}
The statement is invariant under conjugating $\hat{S}^\circ$ so by torality principles we may assume $\iota \in \hat{S}^\circ$; in particular $\iota^{\hat{G}} = \iota^G$.
We shall first show that $\iota^{\hat{G}}\setminus N_{\hat{G}}(H)$ is generic in $\iota^{\hat{G}}$. \cite[Lemmas 2.16 and 3.33]{DJGroups} were supposed to do this, but they only apply when $\iota \in G$. Minor work must be added.

Suppose that $\iota^{\hat{G}}\setminus N_{\hat{G}}(H)$ is not generic in $\iota^{\hat{G}}$. Then by a degree argument, $\iota^{\hat{G}}\cap N_{\hat{G}}(H)$ is generic in $\iota^{\hat{G}}$. Inside $\hat{G}$ apply \cite[Lemma 2.16]{DJGroups} with $X = \iota^{\hat{G}}$ and $M = N_{\hat{G}}(H)$: $X\cap M$ contains a definable, $\hat{G}$-invariant subset $X_1$ which is generic in $X$. Note that $X$ is infinite as otherwise $\iota$ inverts $\hat{G}$, so $X_1$ is infinite as well.
We cannot directly apply \cite[Lemma 3.33]{DJGroups} as $\hat{G}$ itself need not be $N_\circ^\circ$. So let $X_2 = \{\kappa\lambda: \kappa, \lambda \in X_1\}$, which is an infinite, $\hat{G}$-invariant subset of $N_{\hat{G}}(H)$. Since $X_1 \subseteq \iota^{\hat{G}} = \iota^G \subseteq \iota G = G \iota$, $X_2$ is actually a subset of $G$. Hence $X_2 \subseteq N_G(H)$.
The latter need not be soluble but is a finite extension of $N_G^\circ(H)$, which is. Since $X_2$ is infinite and has degree $1$, there is a generic subset $X_3$ of $X_2$ which is contained in some translate $nN_G^\circ(H)$ of $N_G^\circ(H)$, where $n \in N_G(H)$. Then $X_3 \subseteq N_G^\circ(H)\cdot \langle  n\rangle$ which is a definable, soluble group we denote by $M_2$; $X_3$ itself may fail to be $G$-invariant. But $X_2$ is a $G$-invariant subset such that $X_3 \subseteq X_2\cap M_2$ is generic in $X_2$. By \cite[Lemma 3.33]{DJGroups} applied in $G = G^\circ$ to $X_2$ and $M_2$, $G$ is soluble: a contradiction.

The end of the proof is rather worn-out. Consider the definable function $\varphi: \iota^{\hat{G}}\setminus N_{\hat{G}}(H) \to G\cdot\langle  \iota\rangle/H$ which maps $\kappa$ to $\kappa H$. The domain has rank $\rk \iota^{\hat{G}} = \rk \iota^G = \rk G - \rk C_G(\iota)$. The image set has rank at most $\rk G - \rk H$. So the generic fiber has rank at least $\rk H - \rk C_G(\iota)$. But if $\kappa, \lambda$ lie in the same fiber, then $\kappa H = \lambda H$ and $\kappa \lambda \in T_H(\kappa)$. Hence for generic $\kappa$, $\rk T_H(\kappa) \geq \rk \varphi^{-1}(\varphi(\kappa)) \geq \rk H - \rk C_G(\iota)$.
\end{proof}

As it turns out, the algebraic properties of $T_H(\kappa)$ are not always as good as one may wish, and one then focuses on the following sets instead.

\begin{notation*}
For $\kappa$ an involutive automorphism and $H$ a subgroup of some ambient group, we let $\T_H(\kappa) = \{h^2 \in H : h^\kappa = h^{-1}\} \subseteq T_H(\kappa)$. (This set is definable as soon as $\kappa$ and $H$ are.)
\end{notation*}

There is no a priori estimate on $\rk \T_H(\kappa)$, and Proposition~\ref{p:Yanartas} will remedy this.
The $\T$ sets were denoted $\tau$ in \cite{DGroupes}; interestingly enough, they were already used in \cite[Notation 7.4]{BCJMinimal}.

\subsection{The Algebraicity Proposition}\label{s:algebraicity}

We now return to the historical core of the subject. 

Identifying $\SL_2(\K)$ is a classical topic in finite group theory. Proposition~\ref{p:algebraicity} may be seen as a very weak form of the Brauer-Suzuki-Wall Theorem \cite{BSWCharacterization} in odd characteristic. However \cite{BSWCharacterization} heavily relied on character theory, a tool not available in and perhaps not compatible in spirit with the context of groups of finite Morley rank. (One may even interpret the expected failure of the Feit-Thompson theorem in our context as evidence for this thesis.) A character-free proof of outstanding elegance was found by Goldschmidt. Yet his article \cite{GElements} dealt only with the characteristic $2$ case, and ended on the conclusive remark:
\begin{quote}
\itshape
Finally, some analogues of Theorem 2 [Goldschmidt's version of BSW] may hold for odd primes but [\dots] this problem seems to be very difficult.
\end{quote}
Bender's investigations in odd characteristic \cite{BBrauer} and \cite{BFiniteDihedral} both require some character theory. We do not know of a general yet elementary identification theorem for $\PSL(2,q)$ with odd $q$, and hope that the present paper will help ask the question.

In the finite Morley rank context various results identifying $\PSL_2(\K)$ exist, starting with Cherlin's very first article in the field \cite{CGroups} and Hrushovski's generalisation \cite{HAlmost}. For groups of even type \cite{BDNCIT, BNCIT2} provide identification using heavy rank computations. In a different spirit, the reworking of Zassenhaus' classic \cite{ZKennzeichnung} by Nesin \cite{NSharply} and its extension \cite{DNZassenhaus} identify $\PSL_2(\K)$ among $3$-transitive groups; the latter gives a very handy statement.

Most of the ideas in the proof below are in \cite{DGroupes1} and in many other articles before. Only two points need be commented on.
\begin{itemize}
\item
First, we shift from the tradition as in \cite{CJTame, DGroupes1} of invoking the results on permutation groups Nesin had ported to the finite Morley rank context (\cite{DNZassenhaus}, see above).

We decided to use final identification arguments based on the theory of Moufang sets instead. At that point of the analysis the difference may seem essentially cosmetic but the Moufang setting is in our opinion more appropriate as it focuses on the BN-pair. We now rely on recent work by Wiscons \cite{WGroups}.

(Incidently, Nesin had started thinking about BN-pairs in jail \cite{NSplit} but was released before reaching an identification theorem for $\PSL_2(\K)$ in this context; not returning to gaol he apparently never returned to the topic.)
\item
Second, we refrained from using Frécon homogeneity.
This makes the proof only marginally longer in Step~\ref{p:algebraicity:st:K}. The reasons for doing so were consistency with not using it in Proposition~\ref{p:maximality}, and the mere challenge as it was thought a few years ago to be unavoidable.
\end{itemize}

\begin{proposition}[Algebraicity]\label{p:algebraicity}
Let $\hat{G}$ be a connected, $U_2^\perp$ group of finite Morley rank and $G \trianglelefteq \hat{G}$ be a definable, connected, non-soluble, $N_\circ^\circ$-subgroup.
Suppose that for all $\iota \in I(\hat{G})$, $C_G^\circ(\iota)$ is soluble.

Suppose that there exists $\iota \in I(\hat{G})$ such that $C_G^\circ(\iota)$ is contained in two distinct Borel subgroups. Then $G$ has the same Sylow $2$-subgroup as $\PSL_2(\K)$. If in addition $\iota \in G$, then $G\simeq \PSL_2(\K)$, where $\K$ is an algebraically closed field of characteristic not $2$.
\end{proposition}

\begin{proof}
Since $\hat{G}$ is connected, every involution $\iota$ is toral: say $\iota \in \hat{S}^\circ$ a $2$-torus. We may therefore assume that $\hat{G} = G\cdot d(\hat{S}^\circ)$, so that the standard rank computations of the Genericity Proposition~\ref{p:genericity} apply. Moreover, $\hat{G}/G$ is connected and abelian, hence $W_2^\perp$.

\begin{notationinproof}\label{p:algebraicity:n:B,K,kappa}\
\begin{itemize}
\item
Let $B \geq C_G^\circ(\iota)$ be a Borel subgroup of $G$ maximising $\rho_B$; let $\rho = \rho_B$.
\item
Let $K_B = \{\kappa \in \iota^{\hat{G}}\setminus N_{\hat{G}}(B) : \rk T_B(\kappa) \geq \rk B - \rk C_G^\circ(\iota)\}$; by the Genericity Proposition~\ref{p:genericity}, $K_B$ is generic in $\iota^{\hat{G}}$.
\item
Let $\kappa \in K_B$.
\end{itemize}
\end{notationinproof}

Note that it is not clear at this point whether $\iota$ normalises $B$.

\begin{step}\label{p:algebraicity:st:parametercontrol}
$U_\rho(C_G^\circ(\iota)) = 1$. If $U \leq B$ is a non-trivial $\rho$-group, $H \leq G$ is a definable, connected subgroup of $G$ containing $U$, and $\lambda \in \iota^{\hat{G}}$ normalises $H$, then $\lambda$ normalises $B$.
\end{step}
\begin{proofclaim}
For this proof letting $Y_B = U_\rho(Z(F^\circ(B)))$ will spare a few parentheses; by Fact \ref{f:unipotence} \ref{f:unipotence:UZFneq1}, $Y_B \neq 1$.

Suppose $U_\rho(C_G^\circ(\iota)) \neq 1$. Let $D\neq B$ be a Borel subgroup of $G$ containing $C_G^\circ(\iota)$ and maximising $H = (B\cap D)^\circ$: such a Borel subgroup exists by assumption on $C_G^\circ(\iota)$. By construction $\rho_D \succcurlyeq \rho_\iota = \rho_B \succcurlyeq \rho_D$, so all are equal. If $H$ is not abelian then by \cite[4.50(3) and (6) (our Fact \ref{f:Bender})]{DJGroups} $\rho_B \neq \rho_D$, a contradiction.
Hence $H$ is abelian, and in particular $C_G^\circ(\iota) \leq H \leq C_G^\circ(U_\rho(H))$ which is a soluble group; by definition of $B$, the parameter of $C_G^\circ(U_\rho(H))$ is $\rho$. It follows from uniqueness principles (Fact \ref{f:uniqueness}) that $U_\rho(H)$ is contained in a unique Sylow $\rho$-subgroup of $G$. This must be $U_\rho(B) = U_\rho(D)$, so $B = D$: a contradiction.

We just proved $\rho_\iota \prec \rho$.
It follows that for any $\sigma \succcurlyeq \rho$, any $\iota$-invariant $\sigma$-group is inverted by $\iota$.
Now let $U$, $H$, and $\lambda$ be as in the statement. There is a Sylow $\rho$-subgroup $V$ of $H$ containing $U$. By normalisation principles $\lambda$ has an $H$-conjugate $\mu$ normalising $V$: so $\mu$ inverts $V \geq U$.

Let $C = C_G^\circ(U)$, a definable, connected, soluble group. Since $U \leq U_\rho(B)$, one has $Y_B \leq C$. So there is a Sylow $\rho$-subgroup $W$ of $C$ containing $Y_B$. As $\mu$ inverts $U$ it normalises $C$; by normalisation principles $\mu$ has a $C$-conjugate $\nu$ normalising $W$: so $\nu$ inverts $W \geq Y_B$. Now $\nu$ also inverts $U_{\rho_C}(C)$, and commutation principles (our Fact \ref{f:commutation}) yield $[U_{\rho_C}(C), Y_B] = 1$, whence $U_{\rho_C}(C) \leq C_G^\circ(Y_B) \leq B$. At this point it is clear that $\rho_C = \rho$ and $U_\rho(B)$ is the only Sylow $\rho$-subgroup of $G$ containing $U$ by uniqueness principles.

On the other hand $\mu$ inverts $U_{\rho_H}(H)$ and $U$, so by commutation principles $[U_{\rho_H}(H), U] = 1$ and $U_{\rho_H}(H) \leq C$, meaning that $\rho_H = \rho$ as well. Hence $\lambda$ inverts $U_{\rho_H}(H) = U_\rho(H) \geq U$. Since $U_\rho(B)$ is the only Sylow $\rho$-subgroup of $G$ containing $U$, $\lambda$ normalises $B$.
\end{proofclaim}

\begin{notationinproof}
Let $L_\kappa = B\cap B^\kappa$ and $\Theta_\kappa = \{\ell \in L_\kappa: \ell\ell^\kappa \in L_\kappa'\}$.
\end{notationinproof}

\begin{step}\label{p:algebraicity:st:Thetakappa}
$L_\kappa$ and $\Theta_\kappa$ are infinite, definable, $\kappa$-invariant, abelian-by-finite groups. Moreover $\Theta_\kappa^\circ \subseteq T_B(\kappa) \subseteq \Theta_\kappa$.
\end{step}
\begin{proofclaim}
$L_\kappa'$ is finite since we otherwise let $H = C_G^\circ(L_\kappa') \geq U_\rho(Z(F^\circ(B)))$ which is definable, connected, and soluble since $G$ is an $N_\circ^\circ$-group: Step~\ref{p:algebraicity:st:parametercontrol} shows that $\kappa$ normalises $B$, contradicting its choice in Notation~\ref{p:algebraicity:n:B,K,kappa}. It follows that $L_\kappa^\circ$ is abelian and $L_\kappa$ is abelian-by-finite. $\Theta_\kappa$ is clearly a definable, $\kappa$-invariant subgroup of $L_\kappa$, so it is abelian-by-finite as well. By construction $T_B(\kappa) \subseteq \Theta_\kappa$, and $\Theta_\kappa$ is therefore infinite.

We now consider the action of $\kappa$ on $\Theta_\kappa^\circ$ and find according to Lemma \ref{l:Upiperpactions:abelian} a decomposition $\Theta_\kappa^\circ = C_{\Theta_\kappa^\circ}^\circ(\kappa) \qoplus [\Theta_\kappa^\circ, \kappa]$. Now the definable function $\varphi: C_{\Theta_\kappa^\circ}^\circ(\kappa) \to L_\kappa'$ which maps $t$ to $t t^\kappa = t^2$ is a group homomorphism, so by connectedness and since $L_\kappa'$ is finite, $C_{\Theta_\kappa^\circ}^\circ(\kappa)$ has exponent $2$: it is trivial. So $\kappa$ inverts $\Theta_\kappa^\circ$, meaning $\Theta_\kappa^\circ \subseteq T_B(\kappa)$.
\end{proofclaim}

\begin{notationinproof}
Let $U \leq [U_\rho(Z(F^\circ(B))), \Theta_\kappa^\circ]$ be a non-trivial, $\Theta_\kappa^\circ$-invariant $\rho$-subgroup minimal with these properties.
\end{notationinproof}

\begin{step}\label{p:algebraicity:st:K}
$U$ does exist and $C_U^\circ(\iota) = 1$; $C_{\Theta_\kappa^\circ}(U)$ is finite and there exists an algebraically closed field structure $\K$ with $U \simeq \K_+$ and $\Theta_\kappa^\circ/C_{\Theta_\kappa^\circ}(U) \simeq \K^\times$. Moreover $G$ has the same Sylow $2$-subgroup as $\PSL_2(\K)$.
\end{step}
\begin{proofclaim}
Here again we let $Y_B = U_\rho(Z(F^\circ(B))) \neq 1$.

If $\Theta_\kappa^\circ$ centralises $Y_B$ then the $\kappa$-invariant, definable, connected, soluble group $C_G^\circ(\Theta_\kappa^\circ)$ contains $Y_B$ and Step~\ref{p:algebraicity:st:parametercontrol} forces $\kappa$ to normalise $B$, against its choice in Notation~\ref{p:algebraicity:n:B,K,kappa}. Hence $[Y_B, \Theta_\kappa^\circ] \neq 1$; it is a $\rho$-group (Fact \ref{f:unipotence} \ref{f:unipotence:rhocommutator}; no need for Frécon homogeneity here).

We show that $C_U^\circ(\iota) = 1$; be careful that $\iota$ need not normalise $U$ nor even $B$. Yet if $C_U^\circ(\iota)$ is infinite then Step~\ref{p:algebraicity:st:parametercontrol} applied to $C_G^\circ(C_U^\circ(\iota)) \geq Y_B$ forces $\iota$ to normalise $B$, and then $\iota$ inverts $U_\rho(B) \geq U \geq C_U^\circ(\iota)$: a contradiction.

Suppose that $C_{\Theta_\kappa^\circ}(U)$ is infinite; Step~\ref{p:algebraicity:st:parametercontrol} applied to $C_G^\circ(C_{\Theta_\kappa^\circ}(U)) \geq U$ forces $\kappa$ to normalise $B$: a contradiction.
We now wish to apply Zilber's Field Theorem. It may look like we fall short of $\Theta_\kappa^\circ$-minimality but fear not. Follow for instance the proof in \cite[Theorem 9.1]{BNGroups}. It suffices to check that any non-zero $r$ in the subring of $\End(U)$ generated by $\Theta_\kappa^\circ$ is actually an automorphism. But by push-forward \cite[Lemma 2.11]{BSignalizer} $\im r \simeq U/\ker r$ is a non-trivial, $\Theta_\kappa^\circ$-invariant $\rho$-subgroup. By minimality of $U$ as such, $r$ is surjective. In particular $\ker r$ is finite.
Suppose it is non-trivial and form, like in \cite[Theorem 9.1]{BNGroups}, the chain $(\ker r^n)$. Each term is $\Theta_\kappa^\circ$-central by connectedness, so $C_U^\circ(\Theta_\kappa^\circ)$ contains an infinite torsion subgroup $A$. If there is some torsion unipotence then $A = U$ by minimality as a $\rho$-group, and $\Theta_\kappa^\circ$ centralises $U$: a contradiction.
So $A$ contains a non-trivial $q$-torus for some prime number $q$. This means that there is a $q$-torus in $[Y_B, \Theta_\kappa^\circ] \leq B'$ which contradicts, for instance, \cite[Proposition 3.26]{FHall}.
Hence every $r \in \langle  \Theta_\kappa^\circ\rangle_{\End(U)}$ is actually an automorphism of $U$: field interpretation applies (it also follows, a posteriori, that $U$ is $\Theta_\kappa^\circ$-minimal all right).

A priori $\Theta_\kappa^\circ/C_{\Theta_\kappa^\circ}(U)$ simply embeds into $\K^\times$. But one has by Step~\ref{p:algebraicity:st:Thetakappa} and the definition of $\kappa$:
\begin{align*}
\rk \Theta_\kappa^\circ/C_{\Theta_\kappa^\circ}(U) & = \rk \Theta_\kappa^\circ = \rk T_B(\kappa) \geq \rk B - \rk C_G^\circ(\iota) = \rk B - \rk C_B^\circ(\iota) = \rk \iota^B\\ & \geq \rk \iota^U = \rk U - \rk C_U(\iota) = \rk U = \rk \K_+
\end{align*}
It follows that $\Theta_\kappa^\circ/C_{\Theta_\kappa^\circ}(U) \simeq \K^\times$.
At this point $\Theta_\kappa^\circ$ contains a non-trivial $2$-torus. By the $2$-structure Proposition~\ref{p:2structure} and in view of the assumption on centralisers of involutions, the Sylow $2$-subgroup of $G$ is either connected or isomorphic to that of $\PSL_2(\K)$. Suppose it is connected. Then $G$ is $W_2^\perp$; since $\hat{G}/G$ is as well, so is $\hat{G}$ by Lemma \ref{l:W2perp:factor}. This contradicts the fact that $\kappa$ inverts the $2$-torus of $\Theta_\kappa^\circ$.
\end{proofclaim}

For the rest of the proof we now suppose that $\iota$ lies in $G$. So we may assume $\hat{G} = G$. Bear in mind that since the Prüfer $2$-rank is $1$ by Step~\ref{p:algebraicity:st:K}, all involutions are conjugate.

\begin{notationinproof}\
\begin{itemize}
\item
Let for consistency of notations $i = \iota \in G$ and $k = \kappa \in G$. (By torality principles, $i \in C_G^\circ(i) \leq B$.)
\item
Let $j_k$ be the involution in $\Theta_k^\circ$.
\end{itemize}
\end{notationinproof}

Since $i, j_k$ are in $B$ they are $B$-conjugate. In particular $C_G^\circ(j_k) \leq B$.

\begin{step}\label{p:algebraicity:st:centralisers}
$\Theta_k^\circ = C_G^\circ(j_k)$. Moreover $\rk U = \rk C_G^\circ(i) = \rk \Theta_k$, $\rk B \leq 2 \rk U$, and $\rk G \leq \rk B + \rk U$.
\end{step}
\begin{proofclaim}
One inclusion is clear by abelianity of $\Theta_k^\circ$ obtained in Step~\ref{p:algebraicity:st:Thetakappa}.
Now let $N = N_G^\circ(C_G^\circ(k, j_k))$. Since $L_k^\circ$ is abelian by Step~\ref{p:algebraicity:st:Thetakappa}, so are $C_G^\circ(j_k) \leq L_k^\circ$ and its conjugate $C_G^\circ(k)$. Hence $\Theta_k^\circ \leq C_G^\circ(j_k) \leq N$ and by torality $k \in C_G^\circ(k) \leq N$. So $N$ contains a non-trivial $2$-torus and an involution inverting it: by the structure of torsion in definable, connected, soluble groups, $N$ is not soluble. Since $G$ is an $N_\circ^\circ$-group, one has $C_G^\circ(k, j_k) = 1$, so $k$ inverts $C_G^\circ(j_k)$. Hence $C_G^\circ(j_k) \leq \Theta_k^\circ$.

We now compute ranks.
By Steps \ref{p:algebraicity:st:K} and \ref{p:algebraicity:st:centralisers}, $\rk C_G^\circ(i) = \rk \Theta_k^\circ = \rk \K^\times = \rk \K_+ = \rk U$. By definition of $k \in K_B$ and Step~\ref{p:algebraicity:st:Thetakappa}, $\rk \Theta_k^\circ = \rk T_B(k) \geq \rk B - \rk C_B(i)$, so $\rk B \leq 2 \rk U$.

Now remember that $k$ varies in a set $K_B$ generic in $i^G$. Let $f: K_B \to i^B$ be the definable function mapping $k$ to $j_k$. If $j_k = j_\ell$ then $\ell \in C_G(j_k)$ and the latter has the same rank as $C_G(i)$ so we control fibers. Hence:
\[\rk G - \rk C_G(i) = \rk i^G = \rk K_B \leq \rk i^B + \rk C_G(i) = \rk i^B + \rk C_B(i) = \rk B\]
that is, $\rk G \leq \rk B + \rk C_G(i)$.
\end{proofclaim}

For the end of the proof $k$ will stay fixed; conjugating again in $B$ we may therefore suppose that $j_k = i$.

\begin{notationinproof}
Let $N = C_G(i)$ and $H = B \cap N$.
\end{notationinproof}

\begin{step}
$(B, N, U)$ forms a split BN-pair of rank $1$ (see \cite{WGroups} if necessary).
\end{step}
\begin{proofclaim}
We must check the following:
\begin{itemize}
\item
$G = \langle  B, N\rangle$;
\item
$[N:H] = 2$;
\item
for any $\omega \in N\setminus H$, one has $H = B \cap B^\omega$, $G = B\sqcup B\omega B$, and $B^\omega\neq B$;
\item
$B = U\rtimes H$.
\end{itemize}
First, $H = B \cap N = C_B(i) = C_B^\circ(i)$ by Steinberg's torsion theorem and the structure of torsion in $B$. By the structure of the Sylow $2$-subgroup obtained in Step~\ref{p:algebraicity:st:K}, $H < N$, so using Steinberg's torsion theorem again $[N:H] = 2$. Hence for any $\omega \in N \setminus H = Hk$ one has $B^\omega = B^k \geq H^k = H$ and $H \leq B\cap B^k$. Now by the structure of torsion in $B$, the intersection $B\cap B^k$ centralises the $2$-torus in the abelian group $(B\cap B^k)^\circ = L_k^\circ$ so $B\cap B^k \leq C_B(i) = H$.

Recall that the action of $H = C_G^\circ(i) = \Theta_k^\circ$ on $U$ induces a field structure; in particular $H \cap U \leq C_U(\Theta_k^\circ) = 1$. So $U\cdot H = U\rtimes H$ has rank $2\rk U \geq \rk B$ by Step~\ref{p:algebraicity:st:centralisers} and therefore $B = U\rtimes H$.

It remains to obtain the Bruhat decomposition. But first note that if $C_{N_G(B)}(i) > C_B(i)$ then $C_{N_G(B)}(i) = N$ contains $k$, which contradicts $k \notin N_G(B)$ from Notation~\ref{p:algebraicity:n:B,K,kappa}. So $C_{N_G(B)}(i) = C_B(i)$ and since $B$ conjugates its involutions a Frattini argument yields $N_G(B) \subseteq B\cdot C_{N_G(B)}(i) = B$.

Finally let $g \in G\setminus B$; $g$ does not normalise $B$. Let $X = (U\cap B^g)^\circ$ and suppose $X \neq 1$. In characteristic $p$ this contradicts uniqueness principles. In characteristic $0$, $U \simeq \K_+$ is minimal \cite[Corollaire 3.3]{PGroupes}, so $X = U$; at this point $U = U_\rho(B^g) = U^g$, a contradiction again. In any case $X = 1$. In particular $UgB$ has rank $\rk U + \rk B = \rk G$ by Step~\ref{p:algebraicity:st:centralisers} and $UgB$ is generic in $G$. This also holds of $UkB$ so $g \in BkB$ and $G = B\sqcup BkB = B\sqcup B\omega B$ for any $\omega \in N\setminus H$. This certainly implies $G = \langle  B, N\rangle$.
\end{proofclaim}

We finish the proof with \cite[Theorem 1.2]{WGroups} or \cite[Theorem 2.1]{DMTSpecial}, depending on the characteristic. If $U$ has exponent $p$, then $U_p(H) = 1$ as $H \simeq \K^\times$, so \cite[Theorem 1.2]{WGroups} applies. If not, then $U$ is torsion-free: we use \cite[Theorem 2.1]{DMTSpecial} instead. In any case, $G/\cap_{g \in G} B^g \simeq \PSL_2(\K)$ for some field structure $\K$ which a priori need not be the same as in Step~\ref{p:algebraicity:st:K} but could easily be proved to. Since $\cap_{g \in G} B^g$ is a normal, soluble subgroup, it is finite as $G$ is an $N_\circ^\circ$-group, and therefore central by connectedness. But central extensions of finite Morley rank of quasi-simple algebraic groups are known \cite[Corollary 1]{ACCentral}, so $G\simeq \SL_2(\K)$ or $\PSL_2(\K)$, and the first is impossible by assumption on the centralisers of involutions.
\end{proof}

\begin{remark*}
In order to prove non-connectedness of the Sylow $2$-subgroup of $G$, one only needs solubility of $C_G^\circ(\iota)$ regardless of how centralisers of involutions in other classes may behave. But in order to continue one needs much more.
\begin{itemize}
\item
One cannot work with $j_\kappa$ as all our rank computations rely on the equality $\rk C_G(j_\kappa) = \rk C_G(\iota)$, for which there is no better reason than conjugacy with $\iota$. This certainly implies $\iota \in G$ to start with.
\item
One cannot entirely drop $\iota$ and focus on $j_\kappa$, since there is no reason why $C_G^\circ(j_\kappa)$ should be soluble.
\end{itemize}
\end{remark*}

\subsection{The Devil's Ladder}\label{s:DevilsLadder}

Proposition~\ref{p:DevilsLadder} comes from \cite[Proposition 5.4.9]{DGroupes} and was realised (somewhere in Turkey, in 2007) to be more general; the name was given after a Ligeti study. The first lucid uses were in \cite{DJ4alpha} and \cite{BCDAutomorphisms}. Both the statement and the proof have undergone considerable change since: in 2013  
 the argument still took three pages.

We shall climb the ladder three times: in order to control torsion which is the very purpose of Proposition~\ref{p:Yanartas}, at a rather convoluted moment in Step~\ref{p:maximality:st:conjugacy} of Maximality Proposition~\ref{p:maximality}, and in order to conjugate involutions in the very end of the proof of our Theorem, Step~\ref{t:st:conjugacy}. It may be viewed as an extreme form of Proposition~\ref{p:algebraicity}, Step~\ref{p:algebraicity:st:parametercontrol}; the effective contents of the argument are not perfectly intuitive but for a contradiction proof it suffices to stand firm longer than the group.


\begin{proposition}[The Devil's Ladder]\label{p:DevilsLadder}
Let $\hat{G}$ be a connected, $U_2^\perp$, $W_2^\perp$ group of finite Morley rank and $G \trianglelefteq \hat{G}$ be a definable, connected, non-soluble, $N_\circ^\circ$-subgroup.
Suppose that for all $\iota \in I(\hat{G})$, $C_G^\circ(\iota)$ is soluble.

Let $\kappa, \lambda \in I(\hat{G})$ be two involutions. Suppose that for all $\mu \in I(\hat{G})$ such that $\rho_\mu \succ \rho_\kappa$, $C_G^\circ(\mu)$ is a Borel subgroup of $G$.

Let $B \geq C_G^\circ(\kappa)$ be a Borel subgroup of $G$ and $1 \neq X \leq F^\circ(B)$ be a definable, connected subgroup which is centralised by $\kappa$ and inverted by $\lambda$.

Then $C_G^\circ(X) \leq B$ and $B$ is the only Borel subgroup of $G$ of parameter $\rho_B$ containing $C_G^\circ(X)$; in particular $\kappa$ and $\lambda$ normalise $B$.
\end{proposition}
\begin{proof}
First observe that $\kappa \in C_{\hat{G}}(X)$ which is $\lambda$-invariant, so by normalisation principles $\lambda$ has a $C_{\hat{G}}(X)$-conjugate $\lambda'$ which normalises some Sylow $2$-subgroup of $C_{\hat{G}}(X)$ containing $\kappa$. By the $W_2^\perp$ assumption the Sylow $2$-subgroup of $\hat{G}$ is abelian, so $[\kappa, \lambda'] = 1$; also observe that $\lambda'$ inverts $X$.
Let $C = C_G^\circ(X)$, a definable, connected, and soluble group since $G$ is an $N_\circ^\circ$-group.

First suppose $\rho_C \succ \rho_\kappa$. Then $\kappa$ inverts $U_{\rho_C (C)}$, which is therefore abelian. Since the four-group $\langle  \kappa, \lambda'\rangle$ normalises $U_{\rho_C}(C)$, one of the two involutions $\lambda'$ or $\kappa \lambda'$, call it $\mu$, satisfies $Y = C_{U_{\rho_C}(C)}^\circ(\mu) \neq 1$. Note that $Y$ is a $\rho_C$-group. Let $D = C_G^\circ(Y) \geq U_{\rho_C}(C)$; it is a definable, connected, soluble, $\kappa$-invariant subgroup. Since $\rho_D \succcurlyeq \rho_C \succ \rho_\kappa$, $\kappa$ inverts $U_{\rho_D}(D)$. On the other hand, $Y \leq C_G^\circ(\mu)$ so $\rho_\mu \succ \rho_\kappa$ and by assumption, $C_G^\circ(\mu)$ is a Borel subgroup of $G$, say $B_\mu$. Since $\kappa$ and $\mu$ commute, $\kappa$ normalises $B_\mu$ and since $\rho_\mu \succ \rho_\kappa$, $\kappa$ inverts $U_{\rho_\mu}(B_\mu) \trianglelefteq B_\mu$. It also inverts $Y \leq B_\mu$, so by commutation principles $[U_{\rho_\mu}(B_\mu), Y] = 1$ and $U_{\rho_\mu}(B_\mu) \leq C_G^\circ(Y) = D$.

We are still assuming $\rho_C \succ \rho_\kappa$. The involution $\kappa$ inverts $U_{\rho_D}(D) \trianglelefteq D$ and $U_{\rho_\mu}(B_\mu) \leq D$; so by commutation principles $[U_{\rho_\mu}(B_\mu), U_{\rho_D}(D)] = 1$ and $U_{\rho_D}(D) \leq N_G^\circ( U_{\rho_\mu}(B_\mu)) = B_\mu$. At this stage it is clear that $\rho_D = \rho_\mu$ and $U_{\rho_\mu}(B_\mu) = U_{\rho_D}(D)$. In particular $D \leq N_G^\circ(U_{\rho_\mu}(B_\mu)) = B_\mu$. As a conclusion,
\[X \leq C_G^\circ(U_{\rho_C}(C)) \leq C_G^\circ(Y) = D \leq B_\mu = C_G^\circ(\mu)\]
against the fact that $\mu$ inverts $X$.

This contradiction shows that $\rho_C \preccurlyeq \rho_\kappa$. Now $X \leq F^\circ(B)$, so $U_{\rho_B}(Z(F^\circ(B))) \leq C_G^\circ(X) = C$; hence $\rho_B \preccurlyeq \rho_C \preccurlyeq \rho_\kappa \preccurlyeq \rho_B$ and equality holds. Since by uniqueness principles $U_{\rho_B}(B)$ is the only Sylow $\rho_B$-subgroup of $G$ containing $U_{\rho_B}(Z(F^\circ(B)))$, it also is unique as such containing $U_{\rho_C}(C)$. Hence $N_{\hat{G}}(C) \leq N_{\hat{G}} (U_{\rho_C}(C)) \leq N_{\hat{G}}(B)$. Therefore $\kappa$ and $\lambda$ normalise $B$.
\end{proof}

\subsection{Inductive Torsion Control}\label{s:Yanartas}

It will be necessary to control torsion in the $T_B(\kappa)$-sets. In \cite{DGroupes} this was redone for each conjugacy class of involutions by \emph{ad hoc} arguments which could, in high Prüfer rank, get involved (the ``Birthday Lemmas'' \cite[Lemmes 5.3.9 and 5.3.10]{DGroupes} published as \cite[Lemmes 6.9 and 6.10]{DGroupes2}). We proceed more uniformly although some juggling is required. Like in \cite{DGroupes2} the argument will be applied twice: to start the proof of the Maximality Proposition~\ref{p:maximality}, and later to conjugate involutions in Step~\ref{t:st:conjugacy} of the final argument. This accounts for the disjunction in the statement.

There was nothing equally technical in \cite{BCDAutomorphisms} as controlling involutions there was trivial. An inner version of the argument was found in Yanarta\c{s} in the Spring of 2007 and added to \cite{DJ4alpha}. Externalising involutions is no major issue.

\begin{proposition}[Inductive Torsion Control]\label{p:Yanartas}
Let $\hat{G}$ be a connected, $U_2^\perp$, $W_2^\perp$ group of finite Morley rank and $G \trianglelefteq \hat{G}$ be a definable, connected, non-soluble, $N_\circ^\circ$-subgroup.
Suppose that for all $\iota \in I(\hat{G})$, $C_G^\circ(\iota)$ is soluble.

Let $\iota \in I(\hat{G})$ and $B \geq C_G^\circ(\iota)$ be a Borel subgroup. Suppose that for all $\mu \in I(\hat{G})$ such that $\rho_\mu \succ \rho_\iota$, $C_G^\circ(\mu)$ is a Borel subgroup of $G$.
Let $\kappa \in I(\hat{G}) \setminus N_{\hat{G}}(B)$ be such that $T_B(\kappa)$ is infinite.

Suppose either that $B = C_G^\circ(\iota)$ or that $\iota$ and $\kappa$ are $\hat{G}$-conjugate. Then $\T_B(\kappa)$ has the same rank as $T_B(\kappa)$, and contains no torsion elements.
\end{proposition}
\begin{proof}
First remember that since $\hat{G}$ is $W_2^\perp$, if some involution $\omega \in I(\hat{G})$ inverts a toral element $t \in \hat{G}$, then $t^2 =1$. One may indeed take a maximal decent torus $\hat{T}$ of $\hat{G}$ containing $t$; then $\omega$ normalises $C_{\hat{G}}^\circ(t)$ which contains $\hat{T}$ and its $2$-torus $\hat{T}_2$, so by normalisation principles $\omega$ has a $C_{\hat{G}}^\circ(t)$-conjugate $\omega'$ normalising $\hat{T}_2$. By the $W_2^\perp$ assumption, the latter already is a Sylow $2$-subgroup of $\hat{G}$, whence $\omega' \in \hat{T}_2 \leq C_{\hat{G}}^\circ(t)$. It follows that $\omega$ centralises $t$; it also inverts it by assumption, so $t^2 = 1$.

The proof starts here.

We first show that $B$ has no torsion unipotence. The argument is a refinement of Step~\ref{p:algebraicity:st:Thetakappa} of Proposition~\ref{p:algebraicity}. Suppose that there is a prime number $p$ with $U_p(B) \neq 1$. Let $L_\kappa = B \cap B^\kappa$ (be careful that we do \emph{not} consider the connected component). Since $C_G^\circ(L_\kappa')$ contains both $U_p(Z(F^\circ(B))$ and $U_p(Z(F^\circ(B^\kappa)$, uniqueness principles imply that $L_\kappa'$ is finite. Unfortunately $L_\kappa$ need not be abelian so let us introduce:
\[\Theta_\kappa = \{\ell \in L_\kappa : \ell \ell ^\kappa \in L_\kappa'\}\]
which is a definable, $\kappa$-invariant subgroup of $B$ containing $T_B(\kappa)$; in particular it is infinite. Also note that $\Theta_\kappa^\circ$ is abelian. Now let $A \leq U_p (B)$ be a $\Theta_\kappa^\circ$-minimal subgroup. $\Theta_\kappa^\circ$ cannot centralise $A$ since otherwise $C_G^\circ(\Theta_\kappa^\circ) \geq \langle   A, A^\kappa\rangle$, against uniqueness principles. So by Zilber's field theorem the action induces an algebraically closed field of characteristic $p$ structure. By Wagner's theorem on fields \cite[consequence of Corollary 9]{WFields} $\Theta_\kappa^\circ$ contains a $q$-torus $T_q$ for some $q \neq p$. Up to taking the maximal $q$-torus of $\Theta_\kappa^\circ$ we may assume that $\kappa$ normalises $T_q$. Write if necessary $T_q$ as the sum of a $\kappa$-centralised and a $\kappa$-inverted subgroup; by the first paragraph of the proof, $\kappa$ centralises $T_q$. So for any $t \in T_q$ one has $t t^\kappa = t^2 \in L_\kappa'$, therefore $T_q \leq L_\kappa'$ against finiteness of the latter.

We have disposed of torsion unipotence inside $B$, and every element of prime order in $B$ is toral by the structure of torsion in definable, connected, soluble groups. By the first paragraph of the proof, no element of finite order $\neq 2$ of $B$ is inverted by \emph{any} involution (this will be used in the next paragraph with an involution distinct from $\kappa$). In particular $d(t^2)$ is torsion-free for any $t\in T_B(\kappa)$; hence the definable hull of any element of $\T_B(\kappa)$ is torsion-free.

We now show that $T_B(\kappa)$ can contain but finitely many involutions (possibly none). Suppose that it contains infinitely many. Since $B$ has only finitely many conjugacy classes of involutions, there are $i, j \in T_B(\kappa)$ which are $B$-conjugate. Now $i \in B$ so $\{B, i\} \subseteq F^\circ(B)$; by Lemma \ref{l:involutiveaction:soluble} (although \cite[Lemma 24]{DJSmall} would do here) $B = B^{+_i} \cdot \{B, i\}$ so there is $x \in \{B, i\} \subseteq (F^\circ(B))^{-_i}$ with $j = i^x$. Since $i$ inverts $x$, $d(x^2)$ is torsion-free. Also, $1 \neq ij = i i^x = x^2 \in F^\circ(B)$.
Let $X = d(x^2)$ which is an abelian, definable, connected, infinite subgroup; like $ij$ it is centralised by $\kappa$ and inverted by $i$. There are two cases.
\begin{itemize}
\item
If $B = C_G^\circ(\iota)$ then $\iota$ centralises $X$ whereas $\kappa i$ inverts it (yes, $\kappa$ and $i$ do commute). Since $X \leq F^\circ(B)$ with $C_G^\circ(\iota) \leq B$, the Devil's Ladder, Proposition~\ref{p:DevilsLadder}, applied to the pair $(\iota, \kappa i)$ leads to $\kappa i \in N_{\hat{G}}(B)$ and $\kappa \in N_{\hat{G}}(B)$: a contradiction.
\item
If $\kappa$ is $\hat{G}$-conjugate to $\iota$, say $\kappa = \iota^\gamma$ for some $\gamma \in \hat{G}$, we work in $B^\gamma \geq C_G^\circ(\kappa)$. Since $\kappa$ centralises $X$, $X \leq B^\gamma$. Since $i \in C_{\hat{G}}(\kappa) \cap B \leq C_G(\kappa)$, and by connectedness of the Sylow $2$-subgroup of $\hat{G}$, one has $i \in C_G^\circ(\kappa) \leq B^\gamma$. Since $i$ inverts $X$, $X \leq F^\circ(B^\gamma)$. Finally by conjugacy $\rho_\kappa = \rho_\iota$ so climbing the Devil's Ladder for the pair $(\kappa, i)$ we find $C_G^\circ(X) \leq B^\gamma = B^{\gamma\kappa}$. Since $X \leq F^\circ(B)$ this implies $U_{\rho_B}(Z(F^\circ(B))) \leq C_G^\circ(X) \leq B^\gamma$. Uniqueness principles now yield $B = B^\gamma$. Hence $\kappa\in N_{\hat{G}}(B)$: a contradiction.
\end{itemize}

We conclude to rank equality. Let $i_1, \dots, i_n$ be the finitely many involutions in $T_B(\kappa)$ (possibly $n = 0$) and set $i_0 = 1$. If $t \in T_B(\kappa)$ then the torsion subgroup of $d(t)$ is some $\langle  i_m\rangle$, so $d(i_m t)$ is $2$-divisible, and $i_m t \in \T_B(\kappa)$. Hence $T_B(\kappa) \subseteq \cup i_m \T_B(\kappa)$, which proves $\rk \T_B(\kappa) = \rk T_B(\kappa)$.
\end{proof}

\begin{remarks*}\
\begin{itemize}
\item
One needs $T_B(\kappa)$ to be infinite only to show $U_p(B) = 1$; if one were to assume the latter, the rest of the argument would still work with finite $T_B(\kappa)$, and yield $\T_B(\kappa) = \{1\}$.
\item
The fact that $U_p(B) = 1$ is a strong indication of the moral inconsistency of the configuration.
\end{itemize}
\end{remarks*}

\section{The Proof --- The Maximality Proposition}\label{S:maximality}

Proposition~\ref{p:maximality} is the technical core of the present article; we would be delighted to learn of a finite group-theoretic analogue. It was first devised in the context of minimal connected simple groups of odd type \cite{DGroupes}, then ported to $N_\circ^\circ$-groups of odd type \cite{DJ4alpha}, and to actions on minimal connected simple groups of degenerate type \cite{BCDAutomorphisms}.
The main idea and the final contradiction have not changed but every generalisation has required new technical arguments. So neither of the above mentioned adaptations was routine; nor was combining them.
We can finally state a general form.

\begin{proposition}[Maximality]\label{p:maximality}
Let $\hat{G}$ be a connected, $U_2^\perp$ group of finite Morley rank and $G \trianglelefteq \hat{G}$ be a definable, connected, $W_2^\perp$, non-soluble, $N_\circ^\circ$-subgroup.
Suppose that for all $\iota \in I(\hat{G})$, $C_G^\circ(\iota)$ is soluble.

Then for all $\iota \in I(\hat{G})$, $C_G^\circ(\iota)$ is a Borel subgroup of $G$.
\end{proposition}

\begin{proof}
The proof is longer and more demanding than others in the article, but one should be careful to distinguish two levels.
\begin{itemize}
\item
At a superficial level, all arguments resorting to local analysis in $G$ and to the Bender method (Steps \ref{p:maximality:st:abelianity} and \ref{p:maximality:st:Jkappa}) would be much shorter and more intuitive if one knew that Borel subgroups of $G$ have abelian intersections. There is no hope to prove such a thing but it may be a good idea to have a quick look at the structure of the proof in this ideally-behaved case.
\item
At a deeper level, assuming abelianity of intersections does not make the statement of the proposition obvious and the reader is invited to think about it. Even with abelian intersections of Borel subgroups there would still be something to prove;  this certainly uses the $T_B(\kappa)$ sets and rank computations of \S\ref{s:genericity} as nothing else is available. As a matter of fact, even under abelian assumptions, we cannot think of a better strategy than the following.
\end{itemize}
The long-run goal (Step~\ref{p:maximality:st:concentration}) is to collapse the configuration by showing that $G$-conjugates of some subgroup of $G$ generically lie inside $B$. This form of contradiction was suggested by Jaligot to the author then his PhD student for \cite{DGroupes1}. It is typical of Jaligot's early work in odd type \cite[Lemme~2.13]{JFT}. (The author's original argument based on the distribution of involutions was both doubtful and less elegant; even recently he could feel the collapse in terms of involutions, but failed to write it down properly.)

Controlling generic $G$-conjugates of an arbitrary subgroup is not an easy task. The surprise (Step~\ref{p:maximality:st:conjugacy}) is that the $T_B(\kappa)$ sets, or more precisely the $\T_B(\kappa)$ sets, form the desired family. Seeing this requires a thorough analysis of $\T_B(\kappa)$, and embedding it into some abelian subgroup of $B$ with pathological rigidity properties (Step~\ref{p:maximality:st:Jkappa}).
The crux of the argument involves some intersection of Borel subgroups. Interestingly enough, abelian intersections could be removed from \cite{DGroupes, DGroupes1, DGroupes2, DJ4alpha} by a somehow artificial observation on torsion; abelian intersections started playing a non-trivial role in \cite{BCDAutomorphisms} but as a result the global proof then divided into two parallel lines. We could find a more uniform treatment, although the proof of Step~\ref{p:maximality:st:Jkappa} still divides into two along the line of abelianity.

The beginning of the argument (Steps \ref{p:maximality:st:abelianity}, \ref{p:maximality:st:action}, \ref{p:maximality:st:uniqueness}) simply prepares for the analysis, showing that $\T_B(\kappa)$ behaves like a semisimple group. Of course controlling torsion with Proposition~\ref{p:Yanartas} is essential in the first place; studying torsion separately thus allowing inductive treatment was the main success of \cite{DJ4alpha}. The proof starts here.

\subsection{The Reactor}

Since $\hat{G}$ is connected, by torality principles every involution has a conjugate in some fixed $2$-torus $\hat{S}^\circ$. We may therefore assume that $\hat{G} = G\cdot d(\hat{S}^\circ)$, so that the standard rank computations of the Genericity Proposition~\ref{p:genericity} apply. Moreover, $\hat{G}/G$ is connected and abelian, hence $W_2^\perp$. Since $G$ is $W_2^\perp$ as well, so is $\hat{G}$ by Lemma \ref{l:W2perp:factor}.

We then proceed by descending induction on $\rho_\iota$ and fix some involution $\iota_0 \in I(\hat{G})$ such that for any $\mu \in I(\hat{G})$ with $\rho_\mu \succ \rho_{\iota_0}$, $C_G^\circ(\mu)$ is a Borel subgroup. Notice that induction will not be used as such in the current proof but merely in order to apply Propositions \ref{p:DevilsLadder} and \ref{p:Yanartas}.

Be warned that there will be some running ambiguity on $\iota_0$ starting from Notation~\ref{p:maximality:n:IB} onwards, the resolution being in the proof of Step~\ref{p:maximality:st:conjugacy}.

\begin{notationinproof}\
\begin{itemize}
\item
Let $B \geq C_G^\circ(\iota_0)$ be a Borel subgroup of $G$ and suppose $B > C_G^\circ(\iota_0)$; let $\rho = \rho_B$.
\item
Let $K_B = \{\kappa \in \iota_0^{\hat{G}}\setminus N_{\hat{G}}(B) : \rk T_B(\kappa) \geq \rk B - \rk C_G^\circ(\iota_0)\}$; by the Genericity Proposition~\ref{p:genericity}, $K_B$ is generic in $\iota_0^{\hat{G}}$.
\item
Let $\kappa \in K_B$.
\item
For the moment we simply write $\T = \T_B(\kappa)$.
\end{itemize}
\end{notationinproof}

By Inductive Torsion Control (Proposition~\ref{p:Yanartas}), one has $\rk \T \geq \rk B - \rk C_G^\circ(\iota_0)$, and $\T$ contains no torsion elements.

\begin{step}[uniqueness]\label{p:maximality:st:uniqueness}\
\begin{enumerate}[(i)]
\item\label{p:maximality:st:uniqueness:i}
$B$ is the only Borel subgroup of $G$ containing $C_G^\circ(\iota_0)$.
\item\label{p:maximality:st:uniqueness:ii}
$N_{\hat{G}}(B)$ contains a Sylow $2$-subgroup $\hat{S}_0$ of $\hat{G}$.
\item\label{p:maximality:st:uniqueness:iii}
If $\lambda \in I(\hat{G}) \cap N_{\hat{G}}(B)$, then $[B, \lambda] \leq F^\circ(B)$ and $B = B^{+_\lambda} \cdot (F^\circ(B))^{-_\lambda}$ with finite fibers.
\item\label{p:maximality:st:uniqueness:iv}
$(N_G(B))^{-_\lambda} \subseteq B$.
\end{enumerate}
\end{step}
\begin{proofclaim}
Since $G$ is $W_2^\perp$, by Algebraicity Proposition~\ref{p:algebraicity} there is a unique Borel subgroup of $G$ containing $C_G^\circ(\iota_0)$; in particular $C_{\hat{G}}(\iota_0)$ normalises $B$. By torality principles, $N_{\hat{G}}(B)$ contains a full Sylow $2$-subgroup $\hat{S}_0$ of $\hat{G}$, which is a $2$-torus as $\hat{G}$ is $W_2^\perp$.
Now let $\lambda \in I(\hat{G}) \cap N_{\hat{G}}(B)$. Conjugating in $N_{\hat{G}}(B)$ we may suppose $\lambda \in \hat{S}_0$. Then $\hat{B} = B\cdot d(\hat{S}_0)$ is a definable, connected, soluble group, so $\hat{B}' \leq F^\circ(\hat{B})$. Using Zilber's indecomposibility theorem, $[B, \lambda] \leq [B, \hat{S}_0] \leq (B\cap F^\circ(\hat{B}))^\circ \leq F^\circ(B)$. So Lemma \ref{l:involutiveaction:soluble} yields $B = (B^+)^\circ\cdot \{B, \lambda\}$. Of course $\{B, \lambda\} \subseteq (F^\circ(B))^{-_\lambda}$.

It remains to prove \ref{p:maximality:st:uniqueness:iv}.
The $2$-torus $\hat{S}_0$ also acts on $N_G(B)$, so it centralises the finite set $N_G(B)/B$. It follows that if $n \in (N_G(B))^{-_\lambda}$, then $nB = n^\lambda B = n^{-1} B$, that is, $n^2 \in B$. If $G$ has no involutions then neither does $N_G(B)/B$ by torsion lifting. But if $G$ does have involutions, then by torality principles $B \geq C_G^\circ(\iota_0)$ already contains a maximal $2$-torus of $G$, which is a Sylow $2$-subgroup of $G$: hence in that case again, $N_G(B)/B$ has no involutions. In any case $n \in B$, which proves $(N_G(B))^{-_\lambda} \subseteq B$.
\end{proofclaim}

The most important claims for the moment are \ref{p:maximality:st:uniqueness:i} and \ref{p:maximality:st:uniqueness:iii}. Claim \ref{p:maximality:st:uniqueness:iv} will play its role in the sole final Step but was more conveniently proved here.

\subsection{The Fuel}

Controlling $\iota_0^G \cap N_{\hat{G}}(B)$ was claimed to be essential in \cite[after Corollaire 5.37]{DGroupes2}. We can actually do without but this will result in some counterpoint of involutions with a final chord at the very end of the proof of Step~\ref{p:maximality:st:conjugacy}.

\begin{notationinproof}\label{p:maximality:n:IB}
Let $I_B = \{\iota \in \iota_0^{\hat{G}}: C_G^\circ(\iota) \leq B\}$.
\end{notationinproof}

\begin{remarks*}
$I_B = \iota_0^{N_G(B)}$ and any maximal $2$-torus $\hat{S} \leq N_{\hat{G}}(B)$ intersects $I_B$, two facts we shall use with no reference. A proof and an observation follow.
\begin{itemize}
\item
If $\iota \in I_B$ then there is $x \in \hat{G} = G\cdot d(\hat{S}_0)$ with $\iota = \iota_0^x$, where $\hat{S}_0$ is a $2$-torus containing $\iota_0$; one may clearly assume $x \in G$. Now by Uniqueness Step~\ref{p:maximality:st:uniqueness} \ref{p:maximality:st:uniqueness:i} and definition of $I_B$, $B^x$ is the only Borel subgroup of $G$ containing $C_G^\circ(\iota) \leq B$, whence $x \in N_G(B)$ and $I_B \subseteq \iota_0^{N_G(B)}$. The converse inclusion is obvious.

By Step~\ref{p:maximality:st:uniqueness} \ref{p:maximality:st:uniqueness:ii}, $N_{\hat{G}}(B)$ contains a Sylow $2$-subgroup of $\hat{G}$ so any maximal $2$-torus $\hat{S} \leq N_{\hat{G}}(B)$ is in fact a Sylow $2$-subgroup of $N_{\hat{G}}(B)$, and contains an $N_{\hat{G}}(B)$-conjugate $\iota$ of $\iota_0$; then $\iota \in \hat{S} \cap I_B$.
\item
On the other hand it is not clear at all whether equality holds in $I_B \subseteq \iota_0^G\cap N_{\hat{G}}(B)$. 
As a matter of fact we cannot show that $B$ is self-normalising in $G$; this is easy when $G$ is $2^\perp$ but not in general. At this point, using $C_G^\circ(\iota) < B$, there is a lovely little argument showing that $C_G(\iota)$ is connected
(which is not obvious if $G < \hat{G}$ as Steinberg's torsion theorem no longer applies), but one cannot go further. Moreover, self-normalisation techniques \emph{à la} \cite{ABFWeyl} do not work in the $N_\circ^\circ$ context.

The first claim below will remedy this.
\end{itemize}
\end{remarks*}

\begin{step}[action]\label{p:maximality:st:action}\
\begin{enumerate}[(i)]
\item\label{p:maximality:st:action:i}
If $\lambda \in \iota_0^G \cap N_{\hat{G}}(B)$ but $\lambda \notin I_B$, then $\lambda$ inverts $U_\rho(Z(F^\circ(B)))$.
\item\label{p:maximality:st:action:ii}
$[U_\rho(Z(F^\circ(B))), \T] \neq 1$.
\end{enumerate}
\end{step}
\begin{proofclaim}
For this proof letting $Y_B = U_\rho(Z(F^\circ(B)))$ will spare a few parentheses.

Let $\lambda$ be as in the statement and suppose that $X = C_{Y_B}^\circ(\lambda)$ is non-trivial. Then $X$ is a $\rho$-group. By uniqueness Step~\ref{p:maximality:st:uniqueness} \ref{p:maximality:st:uniqueness:iii}, $B = B^{+_\lambda} \cdot (F^\circ(B))^{-_\lambda}$; obviously both terms normalise $X$ so $X \trianglelefteq B$. It follows from uniqueness principles that $U_\rho(B)$ is the only Sylow $\rho$-subgroup of $G$ containing $X$. Since $X \leq C_G^\circ(\lambda)$ is contained in some conjugate $B^x$ of $B$, we have $U_\rho(B^x) = U_\rho(B)$ so $C_G^\circ(\lambda) \leq B$ and $\lambda \in I_B$: a contradiction.

We move to the second claim. Suppose that $\T$ centralises $Y_B$. Let $C = C_G^\circ(\T)$, a definable, connected, soluble, $\kappa$-invariant subgroup; let $U$ be a Sylow $\rho$-subgroup of $C$ containing $Y_B$. By normalisation principles $\kappa$ has a $C$-conjugate $\lambda$ normalising $U$, and inverting $\T$. Since $U_\rho(B)$ is the only Sylow $\rho$-subgroup of $G$ containing $Y_B$, $\lambda$ normalises $B$. Hence $\lambda \in \iota_0^{\hat{G}} \cap N_{\hat{G}}(B) = \iota_0^G \cap N_{\hat{G}}(B)$. We see two cases.

First suppose $\lambda \notin I_B$. Then by claim \ref{p:maximality:st:action:i}, $\lambda$ inverts $Y_B$. If $\rho_C = \rho$, then apply uniqueness principles: $U_\rho(B)$ is the only Sylow $\rho$-subgroup of $G$ containing $Y_B$, so it also is the only Sylow $\rho$-subgroup of $G$ containing $U_\rho(C)$. As the latter is $\kappa$-invariant, so is $B$: a contradiction. Therefore $\rho_C \succ \rho$. It follows that $\lambda$ inverts $U_{\rho_C}(C)$, whence $[U_{\rho_C}(C), Y_B] = 1$ by commutation principles. This forces $U_{\rho_C}(C) \leq C_G^\circ(Y_B) \leq B$, against $\rho_C \succ \rho$.

So $\lambda \in I_B$, i.e. $C_G^\circ(\lambda) \leq B$. But by uniqueness Step~\ref{p:maximality:st:uniqueness} \ref{p:maximality:st:uniqueness:iii}, $\T \subseteq (F^\circ(B))^{-_\lambda}$, and therefore $\T \subseteq F^\circ(B)\cap F^\circ(B)^\kappa$. Since all elements in $\T$ are torsion-free by the Torsion Control Proposition~\ref{p:Yanartas}, one even has $\T \subseteq (F^\circ(B)\cap F^\circ(B)^\kappa)^\circ$. The latter is abelian by \cite[4.46(2) (our Fact \ref{f:Bender})]{DJGroups}, and $\T$ is therefore a definable, connected, abelian subgroup.
Now always by the Torsion Control and Genericity Propositions \ref{p:Yanartas} and \ref{p:genericity}, and by the decomposition of $B$ obtained in Step~\ref{p:maximality:st:uniqueness} \ref{p:maximality:st:uniqueness:iii}, one has:
\[\rk \T = \rk T_B(\kappa) \geq \rk B - \rk C_G^\circ(\iota_0) = \rk B - \rk C_G^\circ(\lambda) = \rk B - \rk C_B^\circ(\lambda) = \rk (F^\circ(B))^{-\lambda}\]
A definable set contains at most one definable, connected, generic subgroup, so $\T$ is the only definable, connected, generic group included in $(F^\circ(B))^{-\lambda}$: hence $N_{\hat{G}}((F^\circ(B))^{-\lambda}) \leq N_{\hat{G}}(\T)$ and $B^{+_\lambda}$ normalises $\T$. Moreover $\T\cap B^{+_\lambda} = 1$ since $\lambda$ inverts $\T$ and $\T$ contains no torsion elements. So $\T \cdot B^{+_\lambda} = \T \rtimes B^{+_\lambda}$ is a definable subgroup of rank $\geq \rk (F^\circ(B_\lambda))^{-_\lambda} + \rk B^{+_\lambda} = \rk B$ by Step~\ref{p:maximality:st:uniqueness}~\ref{p:maximality:st:uniqueness:iii}. Hence $B = \T \rtimes B^{+_\lambda}$ normalises $\T$; $B = N_G^\circ(\T)$ since $G$ is an $N_\circ^\circ$-group. In particular $\kappa$ normalises $B$: a contradiction.
\end{proofclaim}

Claim \ref{p:maximality:st:action:i} will be used only once more, in the next Step.

\subsection{The Fuel, refined}

\begin{step}[abelianity]\label{p:maximality:st:abelianity}\
\begin{enumerate}[(i)]
\item\label{p:maximality:st:abelianity:i}
If $\iota \in I_B$ then $\T \cap C_G(\iota) = 1$.
\item\label{p:maximality:st:abelianity:ii}
There is no definable, connected, soluble, $\kappa$-invariant group containing $U_\rho(Z(F^\circ(B)))$ and $\T$.
\item\label{p:maximality:st:abelianity:iii}
$\T$ is a definable, abelian, torsion-free group.
\end{enumerate}
\end{step}
\begin{proofclaim}
The first claim is easy. Let $\iota \in I_B$ and $t \in \T\setminus\{1\}$ be such that $t^\iota = t$. Then $\iota \in C_{\hat{G}}(t)$ which is $\kappa$-invariant; by normalisation principles and abelianity of the Sylow $2$-subgroup, $\kappa$ has a $C_{\hat{G}}(t)$-conjugate $\lambda$ commuting with $\iota$.
By uniqueness Step~\ref{p:maximality:st:uniqueness} \ref{p:maximality:st:uniqueness:i}, $B$ is the only Borel subgroup of $G$ containing $C_G^\circ(\iota)$, so $\lambda$ normalises $B$. Recall from Inductive Torsion Control, Proposition~\ref{p:Yanartas}, that $d(t)$ is torsion-free. By uniqueness Step~\ref{p:maximality:st:uniqueness} \ref{p:maximality:st:uniqueness:iii}, $t^\lambda = t^\kappa = t^{-1}$ forces $t^2 = [t^{-1}, \lambda] \in F^\circ(B)$ and $t\in F^\circ(B)$. We then apply the Devil's Ladder, Proposition~\ref{p:DevilsLadder}, to the action of $\langle  \iota, \kappa\rangle$ on $d(t)$ and find that $\kappa$ normalises $B$: a contradiction.

As the proof of the second claim is a little involved let us first see how it entails the third one. Suppose that $X = (F^\circ(B)\cap F^\circ(B)^\kappa)^\circ$ is non-trivial and let $H = N_G^\circ(X)$; then $G$ being an $N_\circ^\circ$-group and the second claim yield a contradiction. Hence $X = 1$ which proves abelianity of $(B\cap B^\kappa)^\circ$. Then, since elements of $\T\subseteq B\cap B^\kappa$ contain no torsion in their definable hulls by Proposition~\ref{p:Yanartas}, one has $\T\subseteq (B\cap B^\kappa)^\circ$ and $\T$ is therefore an abelian group, obviously definable and torsion-free. So we now proceed to proving the second claim. Here again we let $Y_B = U_\rho(Z(F^\circ(B)))$.

Let $L$ be a definable, connected, soluble, $\kappa$-invariant group containing $Y_B$ and $\T$. We shall show that $Y_B$ and $\T$ commute, which will contradict Step~\ref{p:maximality:st:action} \ref{p:maximality:st:action:ii}. To do this we proceed piecewise in the following sense. Bear in mind that for $t\in \T$, $d(t)$ is torsion-free by Inductive Torsion Control, Proposition~\ref{p:Yanartas}, so one may take Burdges' decomposition of the definable, connected, abelian group $d(t)$. As a result, the set $\T$ is a union of products of various abelian $\tau$-groups for various parameters $\tau$. We shall show that each of them centralises $Y_B$, which will be the contradiction.

So we let $\tau$ be a parameter and $\Theta$ be an abelian $\tau$-group included in the set $\T$. If $\tau = \rho$ then we are done as $\Theta \leq U_\rho(B)$. So suppose $\tau \prec \rho$ and prepare to use the Bender method (\S\ref{s:Borel}).
Since $L \geq \langle  Y_B, \T\rangle$, $L$ is not abelian by Step~\ref{p:maximality:st:action} \ref{p:maximality:st:action:ii}.

Let $C \leq G$ be a Borel subgroup of $G$ containing $N_G^\circ(L') \geq L$ and maximising $\rho_C$. Notice that:
\[U_{\rho_C}(Z(F^\circ(C))) \leq C_G^\circ(F^\circ(C)) \leq C_G^\circ(C') \leq C_G^\circ(L') \leq N_G^\circ(L')\]
so by uniqueness principles and definition of $C$, $C$ is actually the only Borel subgroup of $G$ containing $N_G^\circ(L')$. As the latter is $\kappa$-invariant, so is $C$; in particular $C\neq B$. Moreover $Y_B \leq C$ so uniqueness principles force $\rho_C \succ \rho$, and $H = (B\cap C)^\circ \geq \langle  Y_B, \T\rangle$ is non-abelian. So we are under the assumptions of Fact \ref{f:Bender2} with $B_\ell = B$ and $B_h = C$.

We determine the linking parameter $\rho'$, i.e. the only parameter of the homogeneous group $H'$ \cite[4.51(3) (Fact \ref{f:Bender2})]{DJGroups}. Fact \ref{f:unipotence} \ref{f:unipotence:rhocommutator} (no need for Frécon homogeneity here) shows that the by Step~\ref{p:maximality:st:action}~\ref{p:maximality:st:action:ii} non-trivial commutator $[Y_B, \T]$ is a $\rho$-subgroup of $H'$, hence $\rho' = \rho$.

We now construct a most remarkable involution. Let $V_\rho \leq C$ be a Sylow $\rho$-subgroup of $C$ containing $Y_B$. Since $\kappa$ normalises $C$, it has by normalisation principles a $C$-conjugate $\lambda$ normalising $V_\rho$. By uniqueness principles, $U_\rho(B)$ is the only Sylow $\rho$-subgroup of $G$ containing $Y_B$, so $\lambda$ normalises $B$. If $\lambda\notin I_B$ then by Step~\ref{p:maximality:st:action} \ref{p:maximality:st:action:i} $\lambda$ inverts $Y_B$; since $\rho_C \succ \rho$ it certainly inverts $U_{\rho_C}(C)$ as well, whence by commutation principles $[Y_B, U_{\rho_C}(C)] = 1$ and $U_{\rho_C}(C) \leq C_G^\circ(Y_B) \leq B$, contradicting $\rho_C \succ \rho$. Hence $\lambda \in I_B$; it normalises $B$ and $C$ (hence $H$).

We return to our abelian $\tau$-group $\Theta$ included in the set $\T$, with $\tau \prec \rho$. Let $V_\tau \leq H$ be a Sylow $\tau$-subgroup of $H$ containing $\Theta$. By normalisation principles $\lambda$ has an $H$-conjugate $\mu$ normalising $V_\tau$. We shall prove that $\mu$ actually centralises $V_\tau$; little work will remain after that. Observe that $V_\tau$ is a definable, connected, nilpotent group contained in two different Borel subgroups of $G$ so by \cite[4.46(2) (Fact \ref{f:nilpab})]{DJGroups} it is abelian. By the commutator argument of Fact \ref{f:unipotence} \ref{f:unipotence:rhocommutator} or the simpler push-forward argument of Fact \ref{f:unipotence} \ref{f:unipotence:pushandpull} (no need for Frécon homogeneity here), $[V_\tau, \mu]$ is a $\tau$-group inverted by $\mu$.

Now note that $\mu$, like $\lambda$, is in $I_B$, and normalises $B$ and $C$. Moreover by Step~\ref{p:maximality:st:uniqueness} \ref{p:maximality:st:uniqueness:iii}, $[V_\tau, \mu] \leq F^\circ(B)$. We shall prove that $[V_\tau, \mu] \leq F^\circ(C)$ as well by making it commute with all of $F^\circ(C)$, checking it on each term of Burdges' decomposition of $F^\circ(C)$. Keep Fact \ref{f:Bender2} in mind.

First, by \cite[4.38]{DJGroups}, $\rho' = \rho$ is the least parameter in $F^\circ(C)$; we handle it as follows. Recall that $[V_\tau, \mu] \leq F^\circ(B)$ is a $\tau$-group, so $[V_\tau, \mu] \leq U_\tau(F^\circ(B))$. By \cite[4.52(7)]{DJGroups} and since $\rho'=\rho\neq \tau$, the latter is in $Z(H)$. But by \cite[4.52(3)]{DJGroups}, $U_\rho(F^\circ(C)) = U_{\rho'}(F^\circ(C)) = (F^\circ(B)\cap F^\circ(C))^\circ \leq H$, so $[V_\tau, \mu]$ does commute with $U_\rho(F^\circ(C))$. Now let $\sigma \succ \rho$ be another parameter. Remember that $\mu$ normalises $C$; since $\mu \in \iota_0^{\hat{G}}$, $\sigma \succ \rho_\mu$ and $\mu$ inverts $U_\sigma(F^\circ(C))$. It inverts $[V_\tau, \mu]$ as well so commutation principles force $[V_\tau, \mu]$ to centralise $U_\sigma(F^\circ(C))$.

As a consequence $[V_\tau, \mu] \leq C$ centralises $F^\circ(C)$. Unfortunately this is not quite enough to apply the Fitting subgroup theorem as literally stated in \cite[Proposition 7.4]{BNGroups} due to connectedness issues. The first option is to note that with exactly the same proof as in \cite[Proposition~7.4]{BNGroups}: in any connected, soluble group $K$ of finite Morley rank one has $C_K^\circ(F^\circ(K)) \leq F^\circ(K)$. Another option is to observe that by \cite[4.52(1)]{DJGroups}, $F^\circ(C)$ has no torsion unipotence: in particular, the torsion in $F(C)$ is central in $C$ \cite[2.14]{DJGroups}. Altogether $[V_\tau, \mu]$ commutes with $F(C)$ and we then use the Fitting subgroup theorem stated in \cite[Proposition 7.4]{BNGroups} to conclude $[V_\tau, \mu] \leq F(C)$. Either way we find $[V_\tau, \mu] \leq F^\circ(C)$, and we already knew $[V_\tau, \mu] \leq F^\circ(B)$.
By connectedness $[V_\tau, \mu] \leq (F^\circ(B)\cap F^\circ(C))^\circ$. But the latter as we know \cite[4.51(3)]{DJGroups} is $\rho' = \rho$-homogeneous: since $\rho \succ \tau$, this shows $[V_\tau, \mu] = 1$.

In particular $\mu \in I_B$ centralises $\Theta \leq V_\tau$. By claim \ref{p:maximality:st:action:i}, $\Theta = 1$ which certainly commutes with $Y_B$. This contradiction finishes the proof of claim \ref{p:maximality:st:action:ii}.
\end{proofclaim}

\begin{remark*}
It is possible to avoid using the Devil's Ladder in the proof of claim \ref{p:maximality:st:abelianity:i}. Postpone and finish the proof of claim \ref{p:maximality:st:abelianity:ii} as follows:
\begin{quote}
In particular $\mu \in I_B$ centralises $\Theta$, so $\mu \in C_{\hat{G}}(\Theta)$ which is $\kappa$-invariant. By normalisation principles and abelianity of the Sylow $2$-subgroup, $\kappa$ has a $C_{\hat{G}}(\Theta)$-conjugate $\nu$ commuting with $\mu$. Since $\mu \in I_B$, by uniqueness Step~\ref{p:maximality:st:uniqueness} \ref{p:maximality:st:uniqueness:i}, $\nu$ normalises $B$. By Step~\ref{p:maximality:st:uniqueness}~\ref{p:maximality:st:uniqueness:iii}, $\Theta = [\Theta, \nu] \leq F^\circ(B)$ commutes with $Y_B$. Hence all of $\T$ commutes with $Y_B$, against Step~\ref{p:maximality:st:action} \ref{p:maximality:st:action:ii}.
\end{quote}
Then prove claim \ref{p:maximality:st:abelianity:i}:
\begin{quote}
Now let $t \in \T\setminus\{1\}$ be centralised by $\iota \in I_B$. Like in the previous paragraph, $t\in F^\circ(B)$; $t$ has infinite order and is inverted by $\kappa$. But we proved in the third claim that $(F^\circ(B)\cap F^\circ(B)^\kappa)^\circ = 1$, a contradiction.
\end{quote}
\end{remark*}

Both claims \ref{p:maximality:st:abelianity:i} and \ref{p:maximality:st:abelianity:iii} are crucial. Claim \ref{p:maximality:st:abelianity:ii} is a gadget used in the proof of claim \ref{p:maximality:st:abelianity:iii} and in the next step.

\subsection{The Core}

\begin{notationinproof}\label{p:maximality:n:J}\
\begin{itemize}
\item
Let $\pi$ be the set of parameters occurring in $\T$.
\item
Let $J_\kappa = U_\pi(C_B^\circ(\T))$ (one has $\T \leq J_\kappa$ by Step~\ref{p:maximality:st:abelianity} \ref{p:maximality:st:abelianity:iii}).
\end{itemize}
\end{notationinproof}

We feel extremely uncomfortable with the next step. The question of why to maximise over $C_B^\circ(\T)$ is a mystery and always was. Nine years before writing these lines, the author then a PhD student produced an incorrect study of some similar maximal intersection configuration, and after noticing a well-hidden flaw had to reassemble his proof by trying all possible maximisations. Exactly the same happened to him again. We feel like one piece of the puzzle is still missing, or more confusingly that we are playing with incomplete sets of pieces from distinct puzzles. There are many ways to get it wrong and the step works by miracle.


\begin{step}[rigidity]\label{p:maximality:st:Jkappa}
$J_\kappa$ is an abelian Carter $\pi$-subgroup of $B$. There is a maximal $2$-torus $\hat{S}$ of $\hat{G}$ contained in $N_{\hat{G}}(B) \cap N_{\hat{G}}(J_\kappa)$, and for any $\iota \in I_B \cap \hat{S}$, one has
$C_{U_\pi(N_G^\circ(J_\kappa))}^\circ(\iota) \leq C_G^\circ(\T)$.
\end{step}
\begin{proofclaim}
First of all, observe that by torality principles there is a maximal $2$-torus $\hat{S}_0$ of $\hat{G}$ containing $\iota_0$; by uniqueness Step~\ref{p:maximality:st:uniqueness} \ref{p:maximality:st:uniqueness:i} $\hat{S}_0$ normalises $B$. Bear in mind that any maximal $2$-torus in $N_{\hat{G}}(B)$ contains an involution in $I_B$.

We need more stucture now, so let $C \neq B$ be a Borel subgroup of $G$ containing $C_B^\circ(\T)$ and maximising $H = (B\cap C)^\circ$. There is such a Borel subgroup indeed since $C_G^\circ(\T)$ is $\kappa$-invariant whereas $B$ is not. As one expects there are two cases and we first deal with the abelian one. The other will be more involved technically, but there will be no more complications of this kind when we are done.

Suppose that $H$ is abelian. Since $H \geq C_B^\circ(\T) \geq \T$ by abelianity of the latter, Step~\ref{p:maximality:st:abelianity} \ref{p:maximality:st:abelianity:iii}, and since $H$ is supposed to be abelian as well, $H = C_B^\circ(\T) \leq N_G^\circ(J_\kappa)$. We now consider $N_G^\circ(J_\kappa)$. It is not clear at all whether $B$ contains $N_G^\circ(J_\kappa)$ but one may ask.

If ($H$ is abelian and) $B$ happens to be the only Borel subgroup of $G$ containing $N_G^\circ(J_\kappa)$, then:
\[U_\pi\left(N_{C_G^\circ(\T)}^\circ(J_\kappa)\right) \leq U_\pi\left( N_{C_B^\circ(\T)}^\circ(J_\kappa)\right) = U_\pi\left(C_B^\circ(\T)\right) = J_\kappa\]
and $J_\kappa \leq C_G^\circ(\T)$ is a Carter $\pi$-subgroup of $C_G^\circ(\T)$. As the latter is $\kappa$-invariant, by normalisation principles $\kappa$ has a $C_G^\circ(\T)$-conjugate $\lambda$ normalising $J_\kappa$. But our current assumption that $B$ is the only Borel subgroup of $G$ containing $N_G^\circ(J_\kappa)$ forces $\lambda$ to normalise $B$ as well. By Step~\ref{p:maximality:st:uniqueness} \ref{p:maximality:st:uniqueness:iii} and since $\lambda$ like $\kappa$ inverts the $2$-divisible group $\T$, we find $\T = [\T, \lambda] \leq F^\circ(B)$ which contradicts Step~\ref{p:maximality:st:action}~\ref{p:maximality:st:action:ii}.

So (provided $H$ is abelian) $B$ is not the only Borel subgroup of $G$ containing $N_G^\circ(J_\kappa)$: let $D \neq B$ one such. Then $C_B^\circ(\T) = H \leq N_B^\circ(J_\kappa) \leq (B\cap D)^\circ$ so by maximality of $H$, $H = (B\cap D)^\circ = N_B^\circ(J_\kappa)$ and $J_\kappa = U_\pi(C_B^\circ(\T)) = U_\pi(H)$ is a Carter $\pi$-subgroup of $B$. By normalisation principles there is a $B$-conjugate $\hat{S}$ of $\hat{S}_0$ normalising $J_\kappa$. For $\iota \in \hat{S} \cap I_B$ one has $C_G^\circ(\iota) \leq B$ and:
\[C_{U_\pi\left(N_G^\circ(J_\kappa)\right)}^\circ(\iota) \leq N_B^\circ(J_\kappa) = H \leq C_G^\circ(\T)\]
It is not easy to say more as $N_G^\circ(J_\kappa)$ need not be nilpotent, but we are done with the proof in the abelian case.

We now suppose that $H$ is not abelian. However $H \geq C_B^\circ(\T)$ so if $D \neq B$ is a Borel subgroup of $G$ containing $H$, one has by definition of the latter $H = (B\cap D)^\circ$. By \cite[4.50(3) and (6) (Fact~\ref{f:Bender})]{DJGroups}, we are under the assumptions of Fact \ref{f:Bender2}. Keep it at hand. Let $Q \leq H$ be a Carter subgroup of $H$. Let $\rho'$ denote the parameter of the homogeneous group $H'$. Studying $J_\kappa$ certainly means asking about $\rho'$ and $\pi$.

Here is a useful principle: if $\sigma$ is a \emph{set} of parameters not containing $\rho'$, $V_\sigma \leq H$ is a $\sigma$-subgroup of $H$, and $\hat{S} \leq N_{\hat{G}}(B) \cap N_{\hat{G}}(V_\sigma) \cap N_{\hat{G}}(C)$ is a $2$-torus, then $\hat{S}$ centralises $V_\sigma$. It is easily proved: first, $V_\sigma$ being nilpotent by definition of a $\sigma$-group and contained in two distinct Borel subgroups, is abelian by Fact~\ref{f:nilpab}. Now let $\hat{B} = B\cdot d(\hat{S})$, a definable, connected, soluble subgroup of $\hat{G}$. Then by Zilber's indecomposibility theorem, $[B, \hat{S}] \leq (F^\circ(\hat{B})\cap B)^\circ \leq F^\circ(B)$ and likewise in $C$. Hence $[V_\sigma, \hat{S}]\leq (F^\circ(B) \cap F^\circ(C))^\circ$ which is $\rho'$-homogeneous \cite[4.52(3)]{DJGroups}. As $\rho'\notin \sigma$, we have $[V_\sigma, \hat{S}] = 1$ by (Fact \ref{f:unipotence} \ref{f:unipotence:pushandpull} or \ref{f:unipotence:rhocommutator}), and $\hat{S}$ centralises $V_\sigma$.

The argument really starts here. First, $\rho' \in \pi$. Otherwise by lemma \ref{l:piCarter}, $\T$ is included in a Carter subgroup of $H$; we may assume $\T \leq Q$, and in particular by abelianity of $Q$ (Fact \ref{f:nilpab}), $Q \leq C_G^\circ(\T)$. By Lemma \ref{l:Bender:addendum}, $N_{\hat{G}}(Q) \leq N_{\hat{G}}(B) \cup N_{\hat{G}}(C)$. So there are two cases (yes, this does work for groups).
\begin{itemize}
\item
First suppose ($\rho'\notin \pi$ and) $N_{\hat{G}}(Q) \leq N_{\hat{G}}(C)$. In particular $N_B^\circ(Q) \leq N_H^\circ(Q) = Q$ and $Q$ is a Carter subgroup of $B$. By normalisation principles, $\hat{S}_0$ has a $B$-conjugate $\hat{S}$ in $N_{\hat{G}}(B) \cap N_{\hat{G}}(Q) \leq N_{\hat{G}}(B)\cap N_{\hat{G}}(U_\pi(Q)) \cap N_{\hat{G}}(C)$. As we noted $\hat{S}$ must centralise $U_\pi(Q) \geq \T$. But there is an involution $\iota \in \hat{S}\cap I_B$, and this contradicts Step~\ref{p:maximality:st:abelianity} \ref{p:maximality:st:abelianity:i}.
\item
Hence (still assuming $\rho'\notin \pi$) one has $N_{\hat{G}}(Q) \leq N_{\hat{G}}(B)$. Then $N_{C_G^\circ(\T)}^\circ(Q) \leq N_{C_B^\circ(\T)}^\circ(Q) \leq N_H^\circ(Q) = Q$ and $Q \leq C_G^\circ(\T)$ is a Carter subgroup of $C_G^\circ(\T)$. As the latter is $\kappa$-invariant, by normalisation principles $\kappa$ has a $C_G^\circ(\T)$-conjugate $\lambda$ normalising $Q$. Now since $N_{\hat{G}}(Q) \leq N_{\hat{G}}(B)$, $\lambda$ normalises $B$. Then $\T$ is inverted by $\lambda$ and $2$-divisible, whence $\T = [\T, \lambda] \leq [B, \lambda] \leq F^\circ(B)$ by Step~\ref{p:maximality:st:uniqueness} \ref{p:maximality:st:uniqueness:iii}, contradicting Step~\ref{p:maximality:st:action} \ref{p:maximality:st:action:ii}.
\end{itemize}

So we have proved $\rho' \in \pi$. On the other hand $\rho_B = \rho \notin \pi$ as otherwise $C_G^\circ(U_\rho(\T))$ would contradict Step~\ref{p:maximality:st:abelianity} \ref{p:maximality:st:abelianity:ii}. Suppose for a second $\rho_C \succ \rho_B$; then since $\rho \neq \rho'$, one has $U_\rho(Z(F^\circ(B))) \leq Z(H) \leq C_G^\circ(\T)$ \cite[4.52(7)]{DJGroups}, against Step~\ref{p:maximality:st:action} \ref{p:maximality:st:action:ii}. Since parameters differ \cite[4.50(6)]{DJGroups} one has $\rho_B \succ \rho_C$. In particular \cite[4.52(2)]{DJGroups}, $Q$ is a Carter subgroup of  $B$.

We now show that $\T$ is $\rho'$-homogeneous, i.e. $\pi = \{\rho'\}$. Let $\sigma = \pi \setminus\{\rho'\}$. Since $H'$ is $\rho'$-homogeneous, by Lemma \ref{l:piCarter} we may assume that $U_\sigma(\T) \leq Q$. Now $U_{\rho'}(H) = U_{\rho'}(F^\circ(H))$ is a Sylow $\rho'$-subgroup of $B$ \cite[implicit but clear in 4.52(6)]{DJGroups}. By normalisation principles $\hat{S}_0$ has a $B$-conjugate $\hat{S}$ in $N_{\hat{G}}(B)\cap N_{\hat{G}}(U_{\rho'}(H)) \leq N_{\hat{G}}(B) \cap N_{\hat{G}}(C)$ \cite[4.52(6)]{DJGroups}. Hence $\hat{S}$ normalises $H$. But $Q$ is a Carter subgroup of $H$ so by normalisation principles over $H$, $\hat{S}$ has an $H$-conjugate $\hat{S}_1$ in $N_{\hat{G}}(B) \cap N_{\hat{G}}(C) \cap N_{\hat{G}}(Q)$. By our initial principle, $\hat{S}_1$ centralises $U_\sigma(Q) \geq U_\sigma(\T)$. Since $\hat{S}_1$ contains an involution in $I_B$, we find $U_\sigma(\T) = 1$ by Step~\ref{p:maximality:st:abelianity} \ref{p:maximality:st:abelianity:i}, as desired. Hence $\T$ is $\rho'$-homogeneous.

As a conclusion $\pi = \{\rho'\}$ and $J_\kappa = U_{\rho'}(C_B^\circ(\T)) \leq U_{\rho'}(H)$. The latter is an abelian Sylow $\rho'$-subgroup of $B$ \cite[implicit but clear in 4.52(6) and noted above]{DJGroups}. Also, $\T \leq U_{\rho'}(H) \leq C_B^\circ(\T)$ and $J_\kappa = U_{\rho'}(H)$. We constructed a maximal $2$-torus $\hat{S} \leq N_{\hat{G}}(B) \cap N_{\hat{G}}(J_\kappa)$ a minute ago.

Finally fix $\iota \in \hat{S} \cap I_B$. We aim at showing that $C_{U_{\rho'}(N_G^\circ(J_\kappa))}^\circ(\iota) \leq C_G^\circ(\T)$. Recall that $\hat{S}$ normalises $C$. By normalisation principles $\hat{S}$ normalises some Sylow $\rho'$-subgroup $V_{\rho'}$ of $C$. Then with Lemma \ref{l:involutiveaction:central2torus} under the action of $\iota$, $V_{\rho'} = (V_{\rho'}^+)^\circ \cdot \{V_{\rho'}, \iota\}$. Now $(V_{\rho'}^+)^\circ$ is a $\rho'$-subgroup of $(B \cap C)^\circ = H$, so $(V_{\rho'}^+)^\circ \leq J_\kappa \leq F^\circ(C)$ \cite[4.52(6)]{DJGroups}. Letting $\hat{C} = C\cdot d(\hat{S})$ one easily sees as we often did that $\{V_{\rho'}, \iota\} \subseteq F^\circ(C)$. So $V_{\rho'} \leq F^\circ(C)$ and $V_{\rho'} \leq U_{\rho'}(F^\circ(C))$. Conjugating Sylow $\rho'$-subgroups in $C$ this means that $U_{\rho'}(F^\circ(C))$ is actually the only Sylow $\rho'$-subgroup of $C$. But by \cite[4.52(8)]{DJGroups} any Sylow $\rho'$-subgroup of $G$ containing $U_{\rho'}(H)$ is contained in $C$. This means that $U_{\rho'}(F^\circ(C))$ is the only Sylow $\rho'$-subgroup of $G$ containing $U_{\rho'}(H) = J_\kappa$.

As a conclusion, any Sylow $\rho'$-subgroup of $N_G^\circ(J_\kappa)$ lies in $U_{\rho'}(F^\circ(C))$. Hence, paying attention to the fact that $\iota$ normalises the nilpotent $\rho'$-group $U_{\rho'}(F^\circ(C))$:
\[C_{U_\pi\left(N_G^\circ(J_\kappa)\right)}^\circ(\iota) \leq C_{U_{\rho'}(F^\circ(C))}^\circ(\iota) \leq U_{\rho'}(H) = J_\kappa \leq C_G^\circ(\T)\qedhere\]
\end{proofclaim}

We shall use the Bender method no more.

\subsection{The Reaction}

\begin{notationinproof}\
\begin{itemize}
\item
We now write $\T_\kappa$ for $\T_B(\kappa)$, as the involution $\kappa$ will vary in $K_B$.
\item
Let $Y = \{B, \iota_0\}$.
\end{itemize}
\end{notationinproof}

\begin{step}[conjugacy]\label{p:maximality:st:conjugacy}
\begin{enumerate}[(i)]
\item\label{p:maximality:st:conjugacy:normal}
$Y$ is a normal subgroup of $B$.
\item\label{p:maximality:st:conjugacy:rk}
$\rk B = \rk C_G(\iota_0) + \rk Y$.
\item\label{p:maximality:st:conjugacy:TI}
Any element of $Y \setminus\{1\}$ lies in finitely many $G$-conjugates of $Y$.
\item\label{p:maximality:st:conjugacy:conjugate}
$\T_\kappa$ and $Y$ are $G$-conjugate.
\end{enumerate}
\end{step}
\begin{proofclaim}
As a matter of fact we let $Y_\iota = \{B, \iota\}$ for any $\iota \in I_B$. Since $I_B = \iota_0^{N_G(B)}$, such sets are $N_G(B)$-conjugate to $Y_{\iota_0} = Y$.

Let $\iota \in \hat{S} \cap I_B$; do not forget that there is such an involution all right. Under the action of $\iota$ we may write $J_\kappa = J_\kappa^+ \qoplus [J_\kappa, \iota]$. By Step~\ref{p:maximality:st:abelianity} \ref{p:maximality:st:abelianity:i}, $\T_\kappa \cap J_\kappa^+ = 1$. So using the very definition of $\kappa \in K_B$ this yields the rank estimate:
\[\rk [J_\kappa, \iota] = \rk J_\kappa - \rk J_\kappa^+ \geq \rk \T_\kappa \geq \rk B - \rk C_G^\circ(\iota_0) = \rk B - \rk C_B^\circ(\iota) = \rk \iota^B \geq \rk \iota^{J_\kappa} = \rk [J_\kappa, \iota]\]
Equality follows. In particular $[J_\kappa, \iota] \subseteq Y_\iota$ is generic in $Y_\iota$. Since a definable set of degree $1$ contains at most one definable, generic subgroup, one has $C_B(\iota) \leq N_B(Y_\iota) \leq N_B([J_\kappa, \iota])$. On the other hand since $\hat{G}$ is $W_2^\perp$, $[J_\kappa, \iota]$ has no involutions; it is disjoint from $C_B(\iota)$. Hence $[J_\kappa, \iota] \cdot C_B(\iota) = [J_\kappa, \iota] \rtimes C_B(\iota)$ is a generic subgroup of $B$. It follows $B = [J_\kappa, \iota]\rtimes C_B(\iota)$. At this stage it is clear that $Y_\iota = [J_\kappa, \iota]$ is a normal subgroup of $B$ contained in $F^\circ(B)$, and the same holds of $Y$ by $N_G(B)$-conjugacy. Moreover $\rk Y_\iota = \rk \T_\kappa$; we are not done.

Consider the definable function $f: \T_\kappa \to Y_\iota$ which maps $t$ to $[t, \iota]$; as $J_\kappa$ is abelian, it is a group homomorphism. Bearing in mind that $\T_\kappa \cap C_{J_\kappa}(\iota) = 1$ by Step~\ref{p:maximality:st:abelianity} \ref{p:maximality:st:abelianity:i} and in view of the equality of ranks, $f$ is actually a group isomorphism; we are not done.

Let us show that any non-trivial element of $Y = Y_{\iota_0}$ lies in finitely many $G$-conjugates. For if $a \in Y\setminus\{1\}$ then by the isomorphism $\T_\kappa \simeq Y$ and Inductive Torsion Control, Proposition~\ref{p:Yanartas}, $a$ has infinite order: $C = C_G^\circ(a) \geq \langle  U_\rho(Z(F^\circ(B))), Y\rangle$ is therefore soluble, and $\iota_0$-invariant. If $\rho_C \succ \rho_B$ then $\iota_0$ inverts both $U_{\rho_C}(C)$ and $Y$, and commutation principles yield $[U_{\rho_C}(C), Y] = 1$ whence $U_{\rho_C}(C) \leq N_G^\circ(Y) = B$, a contradiction. Hence $\rho_C \preccurlyeq \rho_B$ and equality follows. Now uniqueness principles show that $U_\rho(B)$ is the only Sylow $\rho$-subgroup of $G$ containing $U_\rho(C)$. If $a \in Y^g$ with $g \in G$ then $U_\rho(B^g)$ is the only Sylow $\rho$-subgroup of $G$ containing $U_\rho(C)$ likewise, so $g \in N_G(B)$. Since $B \leq N_G(Y) \leq N_G(B)$, this can happen only for a finite number of conjugates of $Y$; we are not done.

It remains to conjugate $\T_\kappa$ to $Y$.
We claim that $J_\kappa \leq C_G^\circ(\T_\kappa)$ is a Carter $\pi$-subgroup of $C_G^\circ(\T_\kappa)$, where $\pi$ is as in Notation~\ref{p:maximality:n:J}. For let $N = U_\pi(N_G^\circ(J_\kappa))$ and $N_1 = U_\pi(N\cap C_G^\circ(\T_\kappa))$. We wish to decompose under the action of $\iota$, the involution we fixed at the beginning of the proof. Be very careful however that $\iota$ need not normalise $N_1$. But since $\hat{S}$ normalises $J_\kappa$ it also normalises $N$. Then $\hat{N} = N\cdot d(\hat{S})$ is yet another definable, connected, soluble group, so $\{N, \iota\} \subseteq (\hat{N}'\cap N)^\circ \leq F^\circ(N)$, and Lemma \ref{l:involutiveaction:soluble} applies to $N$. Now take $n_1 \in N_1$ and write its decomposition $n_1 = p n $ \emph{inside} $N$, with $p \in (N^+)^\circ$ and $n \in \{N, \iota\}$. Then $p \in C_{U_\pi(N_G^\circ(J_\kappa))}^\circ(\iota) \leq C_G^\circ(\T_\kappa)$ by Step~\ref{p:maximality:st:Jkappa}. So $n \in C_G^\circ(\T_\kappa)$.
On the other hand, for any $t \in \T_\kappa$ one has using a famous identity:
\begin{align*}
1 & = [[\iota, n^{-1}], t]^n \cdot [[n, t^{-1}], \iota]^t \cdot [[t, \iota], n]^\iota\\
& = [n^{-2}, t]^n \cdot [[t, \iota], n]^\iota\\
& = [[t, \iota], n]^\iota
\end{align*}
Hence $n$ commutes with $[\T_\kappa, \iota] = Y_\iota$ and $n \in N_G(N_G^\circ(Y_\iota)) = N_G(B)$. Since $p \in C_G^\circ(\iota) \leq B$, one has $n_1 = pn \in N_G(B)$, meaning $N_1 \leq N_G^\circ(B) = B$. Now $N_1 \leq U_\pi(N_B^\circ(J_\kappa))$ and since $J_\kappa$ is a Carter $\pi$-subgroup of $B$, $N_1 \leq J_\kappa$. Therefore $J_\kappa$ is a Carter $\pi$-subgroup of $C_G^\circ(\T_\kappa)$.

Stretto. This extra rigidity has devastating consequences. By normalisation principles, $\kappa$ has a $C_G^\circ(\T_\kappa)$-conjugate $\lambda$ normalising $J_\kappa$. If $\lambda$ normalises $B$ then $\T_\kappa \leq [J_\kappa, \lambda] \leq F^\circ(B)$ by Step~\ref{p:maximality:st:uniqueness} \ref{p:maximality:st:uniqueness:iii}, which contradicts $[U_\rho(Z(F^\circ(B))), \T_\kappa] \neq 1$ from Step~\ref{p:maximality:st:action} \ref{p:maximality:st:action:ii}. So $\lambda$ does not normalise $B$. On the other hand $T_\lambda(B)$ contains $\T_\kappa$ so $\lambda \in K_B$. In particular, everything we said so far of $\kappa$ holds of $\lambda$: by rank equality, $\T_\lambda = \T_\kappa$.

By conjugacy of Sylow $2$-subgroups, $\lambda$ has an $N_{\hat{G}}(J_\kappa)$-conjugate $\mu$ in $\hat{S}$. Remember that we took $\hat{G} = G\cdot d(\hat{S}^\circ)$, so $N_{\hat{G}}(J_\kappa) = N_G(J_\kappa) \cdot d(\hat{S})$ and $\mu = \lambda^n$ for some $n \in N_G(J_\kappa)$. Moreover $\mu \in \hat{S}$ commutes with the involution $\iota$ we fixed earlier in the proof. Let $X = C_{Y_\iota}^\circ(\mu) \leq F^\circ(B)$.
\begin{itemize}
\item
Suppose $X = 1$. Then $\mu$ inverts $Y_\iota$, so:
\[Y_\iota \leq [J_\kappa, \mu] = [J_\kappa, \lambda^n] = [J_\kappa, \lambda]^n \leq \T_\lambda^n = \T_\kappa^n\]
and equality follows from the equality of ranks.
\item
Suppose $X \neq 1$. We apply the Devil's Ladder, Proposition~\ref{p:DevilsLadder}, to the action of $\langle  \mu, \iota\rangle$ on $X$ inside $B_\mu$, the only Borel subgroup of $G$ containing $C_G^\circ(\mu)$ by Uniqueness Step~\ref{p:maximality:st:uniqueness} \ref{p:maximality:st:uniqueness:i}. We find $B_\mu \geq C_G^\circ(X) \geq U_\rho(Z(F^\circ(B)))$. Uniqueness principles force $U_\rho(B_\mu) = U_\rho(B)$, which means $\mu \in I_B \cap \hat{S}$. In particular, everything we said in this proof of $\iota$ holds of $\mu$, so:
\[Y_\mu = [J_\kappa, \mu] = [J_\kappa, \lambda^n] = [J_\kappa, \lambda]^n \leq \T_\lambda^n = \T_\kappa^n\]
and equality follows from the equality of ranks.
\end{itemize}
In any case, $\T_\kappa$ is $G$-conjugate to $Y$: we are done.
\end{proofclaim}

Notations and Steps from \ref{p:maximality:n:IB} to \ref{p:maximality:st:Jkappa} may be forgotten.

\subsection{Critical Mass}

\begin{step}[the collapse]\label{p:maximality:st:concentration}
\end{step}

We first estimate $\rk \{\T_\kappa: \kappa \in K_B\}$. The set under consideration is definable all right as a subset of $\{Y^g: g \in G\} = G/N_G(Y)$ by Step~\ref{p:maximality:st:conjugacy} \ref{p:maximality:st:conjugacy:conjugate}. If $\T_\kappa = \T_\lambda$ then there is $g \in G$ with $\T_\kappa = Y^g$. In particular, $\kappa$ and $\lambda$ lie in $N_{\hat{G}}(N_G^\circ(Y^g)) = N_{\hat{G}}(B^g)$ by Step~\ref{p:maximality:st:conjugacy} \ref{p:maximality:st:conjugacy:normal}. Since $\kappa$ and $\lambda$ are $G$-conjugate, $\kappa \lambda \in N_G(B^g)$. Now $\kappa$ inverts $\kappa \lambda$ so by Step~\ref{p:maximality:st:uniqueness} \ref{p:maximality:st:uniqueness:iv}, $\kappa \lambda \in B^g$, and $\lambda \in \kappa T_{B^g}(\kappa)$. The latter has the same rank as $Y$ by Proposition~\ref{p:Yanartas} and Step~\ref{p:maximality:st:conjugacy} \ref{p:maximality:st:conjugacy:conjugate}. It follows that $\rk \{\T_\kappa : \kappa \in K_B\} \geq \rk K_B - \rk Y = \rk G - \rk C_G(\iota) - \rk Y$.

We move to something else. Let $\cF$ be a definable family of conjugates of $Y$. Since an element in $Y$ lies in only finitely many conjugates by Step~\ref{p:maximality:st:conjugacy} \ref{p:maximality:st:conjugacy:TI}, $\rk \bigcup_\cF = \rk \cF + \rk Y$. We first apply this to $\cF_1 = \{\T_\kappa: \kappa \in K_B\}$, finding:
\[\rk \bigcup_{\cF_1} = \rk \bigcup_{\kappa\in K_B} \T_\kappa \geq \rk G - \rk C_G(\iota_0) - \rk Y + \rk Y = \rk G - \rk C_G(\iota_0)\]
We now apply it to $\cF_2 = \{Y^g: g \in G/N_G(Y)\}$, finding:
\[\rk \bigcup_{\cF_2} = \rk Y^G = \rk G - \rk N_G(B) + \rk Y = \rk G - \rk B + \rk Y\]
Both agree by Step~\ref{p:maximality:st:conjugacy} \ref{p:maximality:st:conjugacy:rk}, so $\bigcup_{\cF_1}$ is generic in $\bigcup_{\cF_2}$. However $\bigcup_{\cF_1} \subseteq \left(\bigcup_{\cF_2} \cap B\right)$, which contradicts \cite[Lemma 3.33]{DJGroups}.

This concludes the proof of Proposition~\ref{p:maximality}.\end{proof}

\section{The Proof --- After the Maximality Proposition}

\subsection{The Dihedral Case}\label{s:dihedral}

The following is a combination of two different lines of thought: the study of a pathological ``$W = 2$'' configuration in \cite[Chapitre 4]{DGroupes} (published as \cite[\S 3]{DGroupes2}) and the final argument in \cite{BCDAutomorphisms}. Since we can quickly focus on the $2^\perp$ case only a few details need be adapted in order to move from minimal connected simple to $N_\circ^\circ$-groups, so we feel that the resulting proposition owes much to Burdges and Cherlin. The final contradiction is by constructing two disjoint generic subsets of some definable subset of $G$.

\begin{proposition}[Dihedral Case]\label{p:dihedral}
Let $\hat{G}$ be a connected, $U_2^\perp$ group of finite Morley rank and $G \trianglelefteq \hat{G}$ be a definable, connected, non-soluble, $N_\circ^\circ$-subgroup. Suppose that for all $\iota \in I(\hat{G})$, $C_G^\circ(\iota)$ is soluble.

Suppose that the Sylow $2$-subgroup of $\hat{G}$ is isomorphic to that of $\PSL_2(\C)$. Suppose in addition that for $\iota \in I(\hat{G})$, the group $C_G^\circ(\iota)$ is contained in a unique Borel subgroup of $G$.

Then $\hat{G}/G$ is $2^\perp$ and $B_\iota = C_G^\circ(\iota)$ is a Borel subgroup of $G$ inverted by any involution $\omega \in C_G(\iota)\setminus\{\iota\}$.
\end{proposition}
\begin{proof}
First observe that by torality principles, all involutions in $\hat{G}$ are conjugate. If one is in $\hat{G}\setminus G$ then all are, and $G$ is $2^\perp$. If one is in $G$ then $\Pr_2(G) = 1$ and $\Pr_2(\hat{G}/G) = 0$; $\hat{G}/G$ is $2^\perp$ by the degenerate type analysis \cite{BBCInvolutions} and connectedness.

\begin{notationinproof}\
\begin{itemize}
\item
Let $V = \{1, \iota, \omega, \iota\omega\} \leq \hat{G}$ be a four-group.
\item
Let $\hat{T}_\iota$ be a $2$-torus containing $\iota$ and inverted by $\omega$, and $\hat{T}_\omega$ likewise.
\item
Let $B_\iota$ be the only Borel subgroup of $G$ containing $C_G^\circ(\iota)$, and $B_\omega$ likewise (observe that by uniqueness of $B_\iota$ over $C_G^\circ(\iota)$, $V$ normalises $B_\iota$ and $B_\omega$).
\item
Let $\rho = \rho_{B_\iota}$.
\end{itemize}
\end{notationinproof}

Here is a small unipotence principle we shall use with no reference: if $L \leq G$ is a definable, connected, soluble, $V$-invariant subgroup, then $\rho_L \preccurlyeq \rho$. This is obvious as otherwise all involutions in $V$ invert $U_{\rho_L}(L)$. Bigeneration, Fact \ref{f:bigeneration}, will also play a growing role in the subsequent pages.

\begin{step}\label{p:dihedral:st:distinct}
$B_\iota \neq B_\omega$.
\end{step}
\begin{proofclaim}
Suppose not. If $G$ is $2^\perp$, then it is $W_2^\perp$: by the Maximality Proposition~\ref{p:maximality}, $B_\iota$ is a Borel subgroup of $G$. Hence $C_G^\circ(\iota) = B_\iota = B_\omega = C_G^\circ(\omega)$, and therefore $B_\iota = C_G^\circ(\iota\omega)$ as well. Yet bigeneration, Fact \ref{f:bigeneration}, applies to the action of $V$ on the $2^\perp$ group $G$: a contradiction.

If $G$ is not $2^\perp$ then bigeneration might fail. But now all involutions are in $G$; by torality principles $\iota \in C_G^\circ(\iota) \leq B_\iota = B_\omega$ so $B_\omega$ contains $\hat{T}_\omega \rtimes\langle  \iota\rangle$, which contradicts the structure of torsion in connected, soluble groups.
\end{proofclaim}

\begin{notationinproof}
Let $H = (B_\iota\cap B_\omega)^\circ$.
\end{notationinproof}

Since $\omega$ normalises $B_\iota$ and vice-versa, $H$ is $V$-invariant.

\begin{step}\label{p:dihedral:st:H}
$H$ is abelian and $2^\perp$. Moreover $\iota$ centralises $U_\rho(B_\iota)$ and $\omega$ inverts it; $V$ centralises $H$ and $N_G^\circ(H) = C_G^\circ(H)$.
\end{step}
\begin{proofclaim}
If $H = 1$ then $C_{B_\iota}^\circ(\omega) = 1$ and $\omega$ inverts $B_\iota$; since $\omega$ inverts $\hat{T}_\iota$ which normalises $B_\iota$, commutation principles yield $[\hat{T}_\iota, B_\iota] = 1$ and $B_\iota \leq C_G^\circ(\iota)$. So $B_\iota = C_G^\circ(\iota)$ is an abelian Borel subgroup inverted by $\omega$ and by $\iota\omega$. Hence all our claims hold if $H = 1$. We now suppose $H \neq 1$.

Suppose that $H$ is not abelian and let $L = N_G^\circ(H')$, a definable, connected, soluble, $V$-invariant group. Then $\rho_L \preccurlyeq \rho$ but since $L$ contains $U_\rho(Z(F^\circ(B_\iota)))$ and $U_\rho(Z(F^\circ(B_\omega)))$, equality holds. Hence $U_\rho(Z(F^\circ(B_\iota))) \leq U_\rho(L)$; by uniqueness principles $U_\rho(B_\iota)$ is the only Sylow $\rho$-subgroup of $G$ containing $U_\rho(L)$. The same holds of $U_\rho(B_\omega)$, proving equality and $B_\iota = B_\omega$, against Step~\ref{p:dihedral:st:distinct}. So $H$ is abelian.

Now suppose that $U_\rho(H) \neq 1$ and let $L = N_G^\circ(U_\rho(H))$. Same causes having the same effects, we reach a contradiction again. Hence $U_\rho(H) = 1$, and it follows that $\omega$ inverts $U_\rho(B_\iota)$. The same argument works for $\iota\omega$, so $\iota$ centralises $U_\rho(B_\iota)$.

We now claim that $V$ centralises $H$. For let $K = [H, \iota]$; since $H$ is abelian, using Zilber's indecomposibility theorem we see that $K$ is a definable, connected, abelian group inverted by $\iota$; in particular it is $2$-divisible. Since $\iota$ centralises $U_\rho(B_\iota)$ and inverts $U_\rho(B_\omega)$, commutation principles yield $\langle  U_\rho(B_\iota), U_\rho(B_\omega)\rangle \leq C_G^\circ(K)$ and the latter is $V$-invariant. Uniqueness principles and Step~\ref{p:dihedral:st:distinct} forbid solubility of $C_G^\circ(K)$: this means $K = 1$, and $\iota$ centralises $H$. The same holds of $\omega$.

Suppose that $H$ has involutions: since it is $V$-invariant, so is its Sylow $2$-subgroup $T$ (no need for normalisation principles here). If $\iota \in T$, then $\iota \in H \leq B_\iota$ and $\omega \in B_\omega$ by conjugacy; hence $B_\omega$ contains $\hat{T}_\omega \rtimes \langle  \iota\rangle$, against the structure of torsion in connected, soluble groups. So $\iota \notin T$, and by assumption on the structure of the Sylow $2$-subgroup of $\hat{G}$, $\iota$ inverts $T$; the same holds of $\omega$ and $\iota\omega$, a contradiction.

It remains to show that $N_G^\circ(H) = C_G^\circ(H)$. Let $N = N_G^\circ(H)$. First assume that $G$ is $2^\perp$. Then using Lemma \ref{l:involutiveaction:central2torus} under the action of $\iota$ we write $N = (N^{+_\iota})^\circ \cdot \{N, \iota\}$ where $\{N, \iota\}$ is $2$-divisible. Since $\iota$ centralises $H$, commutation principles applied pointwise force $\{N, \iota\} \subseteq C_G(H)$. We turn to the action of $\omega$ on $N_1 = (N^{+_\iota})^\circ$; with Lemma \ref{l:involutiveaction:central2torus} again $N_1 = (N_1^{+_\omega})^\circ \cdot \{N_1, \omega\}$, and here again $\{N_1, \omega\} \subseteq C_G(H)$. Finally $(N_1^{+_\omega})^\circ \leq C_G^\circ(\iota, \omega) \leq H \leq C_G(H)$ by abelianity, so $N \leq C_G(H)$ and we conclude by connectedness of $N$.

Now suppose that $\hat{G}/G$ is $2^\perp$: as a consequence $V \leq G$. It is not quite clear whether $N$ has involutions and whether $\{N, \iota\}$ is $2$-divisible, so we argue as follows. By normalisation principles, there is a $V$-invariant Carter subgroup $Q$ of $N$. The previous argument applies to $Q$, so $Q \leq C_G^\circ(H)$; it also applies to $F^\circ(N)$, so $F^\circ(N) \leq C_G^\circ(H)$, and $N = F^\circ(N) \cdot Q \leq C_G^\circ(N)$.
\end{proofclaim}

\begin{step}\label{p:dihedral:st:G2perp}
We may suppose that $G$ is $2^\perp$.
\end{step}
\begin{proofclaim}
Suppose that $G$ contains involutions, i.e. $V \leq G$. We shall prove that $H = 1$. So suppose in addition that $H \neq 1$. For the consistency of notations, let $i = \iota \in G$, $w = \omega \in G$, and $T_i = \hat{T}_i$, $T_w = \hat{T}_w$.

We claim that $w$ does \emph{not} invert $F^\circ(B_i)$. For if it does, then $w$ inverts $T_i \leq B_i$ and $F^\circ(B_i)$, so by commutation principles $[T_i, F^\circ(B_i)] = 1$. Let $Q \leq B_i$ be a Carter subgroup of $B_i$ containing $T_i$; then $B_i = F^\circ(B_i) \cdot Q$ centralises $T_i$, and $T_w \leq Z(B_w)$ by conjugacy. Hence $T_i \rtimes \langle  w\rangle \leq \langle  T_i, T_w\rangle \leq C_G^\circ(H)$, against the structure of torsion in connected, soluble groups and $G$ being $N_\circ^\circ$.

Hence $Y_i = C_{F^\circ(B_i)}^\circ(w) \neq 1$. Since $U_\rho(B_i)$ is abelian by Step~\ref{p:dihedral:st:H}, $U_\rho(B_i) \leq C_G^\circ(Y_i)$; since $Y_i$ is $V$-invariant, our small unipotence principle and general uniqueness principles force $C_G^\circ(Y_i) \leq B_i$. Hence by Step~\ref{p:dihedral:st:H}:
\[N_{B_w}^\circ(H) = C_{B_w}^\circ(H) \leq C_{B_w}^\circ(Y_i) \leq H\]
which proves that $H$ is a Carter subgroup of $B_w$. It therefore contains involutions, against Step~\ref{p:dihedral:st:H}.

This contradiction shows that if $G$ has involutions then $H = 1$. Hence, as in the beginning of Step~\ref{p:dihedral:st:H}, $w$ inverts $B_i = C_G^\circ(i)$ and so does any other involution in $C_G(i)\setminus\{i\}$: if $G$ has involutions, Proposition~\ref{p:dihedral} is proved.
\end{proofclaim}

From now on, we suppose that $G$ is $2^\perp$; we are after a contradiction.
Since $G$ is $W_2^\perp$, Maximality Proposition~\ref{p:maximality} applies and $C_G^\circ(\iota) = B_\iota$ is a Borel subgroup of $G$.
Moreover since $G$ is $2^\perp$, it admits a decomposition $G = G^{+_\iota} \cdot G^{-_\iota}$ by Lemma \ref{l:involutiveaction:central2torus}, and the fibers are trivial. From the connectedness of $G$ we deduce that $C_G(\iota) = G^+$ is connected. Finally, since the $2$-torus $\hat{T}_\iota$ normalises $B_\iota$, it centralises the finite quotient $N_G(B_\iota)/B_\iota$, and so does $\iota$. Now $N = N_G(B_\iota)$ admits a decomposition $N = N^+\cdot \{N, \iota\}$ as well; we just proved $N^+ \leq B$ and $\{N, \iota\}\subseteq B$.
Hence $B_\iota = C_G(\iota)$ is a self-normalising Borel subgroup of $G$, which will be used with no reference.

\begin{step}\label{p:dihedral:st:Bi-k}
For any involution $\lambda \in C_{\hat{G}}(\iota)\setminus\{\iota\}$, 
$B_\iota^{-_\lambda} = F^\circ(B_\iota)$.
\end{step}
\begin{proofclaim}
The claim is actually obvious if $H = 1$, an extreme case in which the below argument remains however valid.
Let $X_\iota = C_{F^\circ(B_\iota)}^\circ(\omega)$ and $X_\omega = C_{F^\circ(B_\omega)}^\circ(\iota)$.

Suppose that $X_\iota \neq 1$ \emph{and} $X_\omega \neq 1$. By abelianity of $U_\rho(B_\iota)$ from Step~\ref{p:dihedral:st:H}, $U_\rho(B_\iota) \leq C_G^\circ(X_\iota)$. As the latter is $V$-invariant, it has parameter exactly $\rho$, so $C_G^\circ(X_\iota) \leq N_G^\circ(U_\rho(B_\iota)) = B_\iota$; by uniqueness principles $B_\iota$ is the only Borel subgroup of $G$ with parameter $\rho$ containing $C_G^\circ(X_\iota)$, and likewise for $B_\omega$ over $C_G^\circ(X_\omega)$. It follows that $C_{B_\omega}^\circ(H) \leq (B_\iota\cap B_\omega)^\circ = H$ and $H$ is a Carter subgroup of $B_\omega$. The latter is $\hat{T}_\omega \rtimes \langle  \iota\rangle$-invariant, so by normalisation principles $N_{\hat{G}}(H)$ contains a Sylow $2$-subgroup $\hat{S}$ of $\hat{G}$. Since $V \leq C_{\hat{G}}(H)$ by Step~\ref{p:dihedral:st:H}, we may assume $V \leq \hat{S}$.

Still assuming that $X_\iota \neq 1$ and $X_\omega \neq 1$, we denote by $\mu$ the involution of $V$ which lies in $\hat{S}^\circ = \hat{T}_\mu$ and fix $\nu \in V\setminus\langle  \mu\rangle$. Then by assumption on the structure of the Sylow $2$-subgroup of $\hat{G}$, $\nu$ inverts $\hat{T}_\mu$; it also centralises $H$, so by commutation principles $\hat{T}_\mu \rtimes \langle  \nu\rangle = \hat{S}$ centralises $H \geq \langle  X_\iota, X_\omega\rangle$. Since $B_\iota$ is the only Borel subgroup of $G$ with parameter $\rho$ containing $C_G^\circ(X_\iota)$ (and likewise for $\omega$), $\hat{S}$ normalises both $B_\iota$ and $B_\omega$. Remember that $V = \langle  \iota, \omega\rangle = \langle  \mu, \nu\rangle$; so up to taking $\nu\mu$ instead of $\nu$, we may suppose that $\hat{S}$ normalises $B_\nu$. Now $\nu$ inverts $\hat{T}_\mu$ and centralises $B_\nu$, so by commutation principles $[\hat{T}_\mu, B_\nu] = 1$ and $B_\nu \leq C_G^\circ(\mu) = B_\mu$: a contradiction to Step~\ref{p:dihedral:st:distinct}.

All this shows that $X_\iota = 1$ \emph{or} $X_\omega = 1$; we suppose the first. Then $\omega$ inverts $F^\circ(B_\iota)$. 
Using Lemma \ref{l:involutiveaction:central2torus} we write $B_\iota = B_\iota^{+_\omega} \cdot \{B_\iota, \omega\}$. Notice that since $B_\iota$ is $2^\perp$, $B_\iota^- = \{B_\iota, \omega\}$ (the sign $-$ refers to the action of $\omega$ throughout the present paragraph).
Since $\omega$ inverts the $2$-divisible subgroup $F^\circ(B_\iota)$, one has $F^\circ(B_\iota) \subseteq B_\iota^-$.
Since the set $B_\iota^-$ is $2$-divisible, commutation principles applied pointwise show $F^\circ(B_\iota) \subseteq B_\iota^- \subseteq C_{B_\iota}(F^\circ(B_\iota))$. Hence $B_\iota^-$ turns out to be a union of translates of $F^\circ(B_\iota)$. Now $C_{B_\iota}(F^\circ(B_\iota))$ is normal in $B_\iota$ and nilpotent, so by definition of the Fitting subgroup 
$C_{B_\iota}(F^\circ(B_\iota)) \leq F(B_\iota)$. As a consequence $B_\iota^- \subseteq F(B_\iota)$ is a union of \emph{finitely many} translates of $F^\circ(B_\iota)$. But $\deg B_\iota^- = \deg \{B_\iota, \omega\} = \deg \omega^{B_\iota} = 1$, so $F^\circ(B_\iota) = B_\iota^-$.

The previous paragraph shows that if $X_\iota = 1$, then our desired conclusion holds of $\lambda = \omega$; it then also holds of $\lambda = \iota\omega$. Now any involution $\lambda \in C_{\hat{G}}(\iota)\setminus\{\iota\}$ is a $C_{\hat{G}}(\iota)$-conjugate of $\omega$ or $\iota\omega$, say $\lambda = \omega^n$ with $n \in C_{\hat{G}}(\iota) \leq N_{\hat{G}}(B_\iota) \leq N_{\hat{G}}(F^\circ(B_\iota))$, so:
\[B_\iota^{-_\lambda} = B_\iota^{-_{\omega^n}} = \left(B_\iota^{-_\omega}\right)^n = (F^\circ(B_\iota))^n = F^\circ(B_\iota)\]

Similarly, if $X_\omega = 1$, then for any $\lambda \in C_{\hat{G}}(\omega)\setminus\{\omega\}$, $B_\omega^{-_\lambda} = F^\circ(B_\omega)$. We conjugate $\omega$ to $\iota$ and see that in this case we are done as well.
\end{proofclaim}

\begin{step}\label{p:dihedral:st:rkG-}
$\rk G^{-_\iota} \leq 2\rk F^\circ(B_\iota)$.
\end{step}
\begin{proofclaim}
Let $\kappa = \iota \omega$ and $\check{G} = G \rtimes V$.
Observe that in $\check{G}$ the involutions $\iota$, $\omega$, $\kappa$ are \emph{not} conjugate; one has exactly three conjugacy classes, which also are $G$-classes. So for $(\omega_1, \kappa_1)\in \omega^G \times \kappa^G$, the definable closure $d(\omega_1 \kappa_1)$ contains a unique involution which must be a conjugate $\iota_1$ of $\iota$.

Now consider the definable function from $\omega^G \times \kappa^G$ to $\iota^G$ which maps $(\omega_1, \kappa_1)$ to $\iota_1$; we shall compute its fibers.
If $(\omega_2, \kappa_2)$ also maps to $\iota_1$ then $\omega_1 \omega_2 \in C_G(\iota_1) = B_{\iota_1}$. Hence $\omega_1\omega_2 \in B_{\iota_1}^{-_{\omega_1}} = F^\circ(B_{\iota_1})$ by Step \ref{p:dihedral:st:Bi-k}, and fibers have rank at most $2 \rk F^\circ(B_\iota)$.
As the map is obviously onto, one has $2 \rk F^\circ(B_\iota) \geq \rk \check{G} - \rk B = \rk G^{-_\iota}$.
\end{proofclaim}

\begin{step}\label{p:dihedral:st:F^F}
$(F^\circ(B_\omega))^{F^\circ(B_\iota)}$ and $(F^\circ(B_{\iota\omega}))^{F^\circ(B_\iota)}$ are generic subsets of $G^{-_\iota}$.
\end{step}
\begin{proofclaim}
Recall from Step~\ref{p:dihedral:st:Bi-k} that $\iota$ inverts $F^\circ(B_\omega)$ and centralises $B_\iota$. In particular since $G$ is $2^\perp$, one has $F^\circ(B_\omega) \cap B_\iota = 1$; moreover $(F^\circ(B_\omega))^{F^\circ(B_\iota)} \subseteq G^{-_\iota}$. We now compute the rank.
Consider the definable function from $F^\circ(B_\iota) \times F^\circ(B_\omega)$ to $G$ which maps $(a, x)$ to $x^a$. Let us prove that it has finite fibers.

Suppose $x^a = y^b$ with $b \in F^\circ(B_\iota)$ and $y \in F^\circ(B_\omega)$; then $x^{ab^{-1}} = y$, and applying $\omega$ one finds:
\[y = y^\omega = x^{ab^{-1}\omega} = x^{\omega a^{-1}b} = x^{a^{-1}b} = y^{ba^{-2}b}\]
Since $F^\circ(B_\iota)$ is abelian and $G$ is $2^\perp$, this results in $a^{-1}b \in C_G(y)$ and $x = y$.
We now estimate the size $C_{F^\circ(B_\iota)}(x)$. Suppose $Y = C_{F^\circ(B_\iota)}^\circ(x)$ is infinite. Since $Y$ is $V$-invariant, so is $C_G^\circ(Y)$, a definable, connected, soluble group containing $F^\circ(B_\iota)$. As we know, $C_G^\circ(Y)$ has unipotence parameter at most $\rho$, so $C_G^\circ(Y)$ normalises $U_\rho(B_\iota)$ and $C_G^\circ(Y) \leq B_\iota$; as a matter of fact, by uniqueness principles $B_\iota$ is the only Borel subgroup of $G$ with parameter $\rho$ containing $C_G^\circ(Y)$. It follows $x \in N_G(B_\iota)$.
Hence $x \in N_G(B_\iota) \cap F^\circ(B_\omega) = C_G(\iota) \cap F^\circ(B_\omega) = 1$.

As a result, fibers are finite; it follows $\rk (F^\circ(B_\omega))^{F^\circ(B_\iota)} = 2 \rk F^\circ(B_\iota) \geq \rk G^{-_\iota}$ by Step~\ref{p:dihedral:st:rkG-}; inclusion forces equality. The same holds of $(F^\circ(B_{\iota\omega}))^{F^\circ(B_\iota)}$.
\end{proofclaim}

We now finish the proof of Proposition~\ref{p:dihedral}. By Step~\ref{p:dihedral:st:F^F}, both the sets $(F^\circ(B_\omega))^{F^\circ(B_\iota)}$ and $(F^\circ(B_{\iota\omega}))^{F^\circ(B_\iota)}$ are generic in $G^{-_\iota}$. So there is $t \in F^\circ(B_\omega) \cap F^\circ(B_{\iota\omega})^f \setminus\{1\}$ for some $f \in F^\circ(B_\iota)$.
Then the involution $(\iota\omega)^f = f^{-1}\iota\omega f = f^\omega \iota \omega f = \iota \omega f^2$ centralises $t$, whereas $\iota\omega$ inverts it. So $f^2 \in G$ inverts $t$. This creates an involution in $G$: against Step~\ref{p:dihedral:st:G2perp}.
\end{proof}

\subsection{Strong Embedding}\label{s:strongembedding}

Strong embedding is a classical topic in finite group theory \cite{BTransitive}. Recall that a proper subgroup $M$ of a group $G$ is said to be strongly embedded if $M$ contains an involution but $M\cap M^g$ does not for any $g \notin M$.
The reader should also keep in mind a few basic facts about strongly embedded configurations \cite[Theorem 10.19]{BNGroups} (checking the apparently missing assumptions would be almost immediate here):
\begin{itemize}
\item
involutions in $M$ are $M$-conjugate;
\item
a Sylow $2$-subgroup of $M$ is a Sylow $2$-subgroup of $G$;
\item
$M$ contains the centraliser of any of its involutions.
\end{itemize}
We need no more.
The study of a minimal connected simple group with a strongly embedded subgroup was carried in \cite[Theorem 1]{BCJMinimal}.

\begin{proposition}[Strong Embedding]\label{p:strongembedding}
Let $G$ be a connected, $U_2^\perp$, non-soluble, $N_\circ^\circ$-group of finite Morley rank. If $G$ has a definable, soluble, strongly embedded subgroup, then $\Pr_2(G) \leq 1$.
\end{proposition}

Our proof will be considerably shorter than \cite{BCJMinimal}: thanks to the Maximality Proposition~\ref{p:maximality} we need only handle the case of central involutions \cite[\S 4]{BCJMinimal}. Apart from this, our argument is a subset of the one in \cite[\S 4]{BCJMinimal}: we construct two disjoint generic sets. We only hope to have helped clarify matters in Step~\ref{p:strongembedding:st:BCB} below.

(Incidently, an alternative proof of the non-central case of \cite[Theorem 1]{BCJMinimal} was suggested using state-of-the-art genericity arguments in minimal connected simple groups \cite[Theorem 6.1]{ABFWeyl}. Yet this new proof reproduces the central case \cite[\S 4]{BCJMinimal} and affects only the configuration we need not consider by Maximality.)

\begin{proof}
We let $G$ be a minimal counterexample, i.e. $G$ satisfies the assumptions but $\Pr_2(G) \geq 2$. By the $2$-structure Proposition~\ref{p:2structure}, the Sylow $2$-subgroups of $G$ are connected.

\begin{notationinproof}
Let $M < G$ be a definable, soluble, strongly embedded subgroup. Let $S \leq M$ be a Sylow $2$-subgroup of $G$ and $A = \Omega_2(S^\circ)$ be the group generated by the involutions of $S^\circ$.
\end{notationinproof}

\begin{step}\label{p:strongembedding:st:centralisers}
For all $i \in I(G)$, $C_G^\circ(i)$ is soluble.
\end{step}
\begin{proofclaim}
First observe that $Z(G)$ has no involutions by strong embedding, as they would lie in $S\leq M$ and in any conjugate.

Suppose that there is $i \in A\setminus\{1\}$ with non-soluble $C_G^\circ(i)$. Fix some $2$-torus $\tau_i \leq S$ of Prüfer rank~$1$ containing $i$; since $C^\circ_G(\tau_i)$ is soluble because $G$ is an $N_\circ^\circ$-group, there exists by the descending chain condition some $\alpha \in \tau_i$ with $C_G^\circ(\alpha)$ soluble. We take $\alpha$ with minimal order; then $C_G^\circ(\alpha^2)$ is not soluble, and $\alpha^2 \neq 1$ since $\alpha \neq i$.

Let $H = C_G^\circ(\alpha^2)$ and $N = M \cap H$. Since $\alpha^2 \neq 1$ and $Z(G)$ has no $2$-elements, $H < G$. Observe how $\alpha \in \tau_i \leq S \leq N$. Let $\overline{H} = H/\langle  \alpha^2\rangle$ and $\overline{N} = N/\langle  \alpha^2\rangle$. Then $\overline{N}$ is definable, soluble, and strongly embedded in $\overline{H}$ which still has Prüfer rank $\geq 2$: against minimality of $G$ as a counter-example.
\end{proofclaim}

\begin{notationinproof}
Let $B = M^\circ$.
\end{notationinproof}

\begin{step}\label{p:strongembedding:st:B}
$B$ is a Borel subgroup of $G$ and $A \leq Z(B)$; the group $M/B$ is non-trivial and has odd order. Moreover the following hold.
\begin{enumerate}[(i)]
\item\label{p:strongembedding:st:B:iii}
Strongly real elements of $G$ which lie in $B$ actually lie in $A$.
\item\label{p:strongembedding:st:B:iv}
If $i \in I(B)$ inverts $n \in N_G(B)$ then $n \in B$.
\item\label{p:strongembedding:st:B:v}
For any $g \in G$, $B g I(G)$ is generic in $G$.
\item\label{p:strongembedding:st:B:vi}
$(B\cap B^g)^\circ = 1$ for $g \notin N_G(B)$.
\end{enumerate}
\end{step}
\begin{proofclaim}
By Step~\ref{p:strongembedding:st:centralisers}, connectedness of the Sylow $2$-subgroup, and the maximality Proposition~\ref{p:maximality}, $C_G^\circ(i)$ is a Borel subgroup of $G$ for any $i \in I(G)$. But for $i \in A\setminus\{1\}$, $C_G(i) \leq M$ by strong embedding of the latter, so $C_G^\circ(i) \leq B$ and equality follows. In particular, $A \leq Z(B)$.

By structure of the Sylow $2$-subgroup, $N_G(B)/B$ has odd order, and so has its subgroup $M/B$. But $M$ being strongly embedded conjugates its (more than one) involutions, which are central in $B$: this shows $B < M$.

If $b \in B$ is inverted by some $k \in I(G)$ then $k$ normalises $C_G(b) \geq A$; by normalisation principles and structure of the Sylow $2$-subgroup, one has $k \in C_G(b)$, so $b$ has order at most $2$; this is claim~\ref{p:strongembedding:st:B:iii}. If $i \in I(B)$ inverts $n \in N_G(B)$ then computing modulo $B$: $n^{-1}B = n^i B = nB$, and $n^2 \in B$. Since $N_G(B)/B$ has odd order, $n\in B$, proving \ref{p:strongembedding:st:B:iv}.

We move to \ref{p:strongembedding:st:B:v}. Consider the definable function $B \times I(G)$ which maps $(b, k)$ to $bk$. If $b_1 k_1 = b_2 k_2$ with obvious notations, then $b_2^{-1} b_1$ is a strongly real element of $G$ lying in $B$, hence has order at most $2$ by claim \ref{p:strongembedding:st:B:iii}: this happens only finitely many times, so fibers are finite and $\rk (B\cdot I(G)) = \rk B + \rk I(G) = \rk B + \rk G - \rk B = \rk G$. Then for any $g \in G$:
\[\rk \left(BgI(G)\right) = \rk \left(g B^g I(G)^g\right) = \rk \left(g \left(BI(G)\right)^g\right) = \rk \left(BI(G)\right) = \rk G\]

It remains to control intersections of conjugates of $B$, claim \ref{p:strongembedding:st:B:vi}. Suppose that $H = (B\cap B^g)^\circ$ is infinite. Let $Q \leq H$ be a Carter subgroup of $H$; since $A^g$ centralises $B^g \geq H \geq Q$, it normalises the definable, connected, soluble group $N_G^\circ(Q)$. By bigeneration, Fact \ref{f:bigeneration}, $N_G^\circ(Q) \leq \langle  C_G^\circ(a^g): a \in A\setminus\{1\}\rangle = B^g$, so $N_B^\circ(Q) \leq N_H^\circ(Q) = Q$ and $Q$ is actually a Carter subgroup of $B$. Hence $Q$ contains a Sylow $2$-subgroup of $B$: hence $A \leq Q \leq B^g$, and strong embedding guarantees $g \in N_G(B)$.
\end{proofclaim}

\begin{notationinproof}
Let $w \in M\setminus B$ (denoted $\sigma$ in \cite[Notation 4.1(2)]{BCJMinimal}).
\end{notationinproof}

\begin{step}\label{p:strongembedding:st:w}
We may assume that $w$ is strongly real, in which case the following hold.
\begin{enumerate}[(i)]
\item\label{p:strongembedding:st:w:i}
$C_G(w)$ has no involutions.
\item\label{p:strongembedding:st:w:ii}
If some involution $k \in I(G)$ inverts $w$, then $k$ inverts $C_G^\circ(w)$.
\item\label{p:strongembedding:st:w:iii}
$C_B^\circ(w) = 1$.
\end{enumerate}
\end{step}
\begin{proofclaim}
By Step~\ref{p:strongembedding:st:B} \ref{p:strongembedding:st:B:v}, both $BI(G)$ and $BwI(G)$ are generic in $G$, so they intersect. Hence up to translating by an element of $B$, we may suppose that $w$ is a strongly real element.

Suppose that there is an involution $\ell \in C_G(w)$. Then $w \in C_G(\ell) = C_G^\circ(\ell)$ by Steinberg's torsion theorem and connectedness of the Sylow $2$-subgroup; $C_G^\circ(\ell)$ is a conjugate of $B$ (Sylow theory suffices here; no need to invoke strong embedding). But $w$ is strongly real, so by Step~\ref{p:strongembedding:st:B} \ref{p:strongembedding:st:B:iii} it is an involution, against the fact that $M/B$ has odd order.

We prove \ref{p:strongembedding:st:w:ii}: let $k$ be an involution inverting $w$. Then $C_G^\circ(k)$ is a conjugate $B_k$ of $B$, and $k \in B_k$. Observe how $w \notin N_G(B_k)$ by Step~\ref{p:strongembedding:st:B} \ref{p:strongembedding:st:B:iv}. So $C_G^\circ(k, w) \leq (B_k \cap B_k^w)^\circ$ is trivial by Step~\ref{p:strongembedding:st:B}~\ref{p:strongembedding:st:B:vi}, and $k$ inverts $C_G^\circ(w)$.

Finally let $H = C_B^\circ(w)$ and suppose $H \neq 1$. Bear in mind that $A$ centralises $H$, so it normalises the definable, soluble group $N_G^\circ(H)$. By bigeneration, Fact \ref{f:bigeneration}, $N_G^\circ(H) \leq B$. But $k$ inverts $H$, so it normalises $N_G^\circ(H)$ as well. Hence $N_G^\circ(H) \leq B\cap B^k$, and Step~\ref{p:strongembedding:st:B} \ref{p:strongembedding:st:B:vi} forces $k \in N_G(B)$. Now $k \in B$ inverts $w \in N_G(B) \setminus B$, a contradiction to Step~\ref{p:strongembedding:st:B} \ref{p:strongembedding:st:B:iv}. This shows that $C_B^\circ(w) = 1$.
\end{proofclaim}

\begin{notationinproof}
Let $\check{C} = C_G^\circ(w) \setminus N_G(B)$.
\end{notationinproof}

$\check{C}$ is obviously generic in $C_G^\circ(w)$, as $C_{N_G(B)}^\circ(w) \leq C_B^\circ(w) = 1$ by Step~\ref{p:strongembedding:st:w} \ref{p:strongembedding:st:w:iii}.

\begin{step}\label{p:strongembedding:st:BCB}
$B \check{C} B$ is generic in $G$.
\end{step}
\begin{proofclaim}
This is the only part where we slightly rewrite the argument given in \cite{BCJMinimal}. Let $\cF = \{(m, \ell) \in Bw \times I(G): m^\ell = m^{-1}\}$.

Let $m \in Bw$. If $m$ is inverted by some involution in $G$, then by Step~\ref{p:strongembedding:st:w} \ref{p:strongembedding:st:w:iii} $C_B^\circ(m) = 1$ and $m^B \subseteq Bm$ is generic in $Bm$. So is $w^B$, and $m$ is therefore $B$-conjugate with $w$. So let us count involutions inverting $w$. First, there is such an involution $k$ by Step~\ref{p:strongembedding:st:w}. If $\ell$ is another such, then $k\ell \in C_G(w)$ and $\ell \in k C_G(w)$. Conversely, since $k$ inverts $C_G^\circ(w)$ by Step~\ref{p:strongembedding:st:w} \ref{p:strongembedding:st:w:ii}, any element in $k C_G^\circ(w)$ is an involution inverting $w$.
This together shows:
\[\rk \cF = \rk w^B + \rk C_G^\circ(w) = \rk B + \rk C_G^\circ(w)\]

On the other hand, since $BwI(G)$ and $BI(G)$ are generic in $G$ by Step~\ref{p:strongembedding:st:B} \ref{p:strongembedding:st:B:v}, a generic $\ell \in I(G)$ inverts some element in $Bw$. Hence $\rk \cF \geq \rk I(G) = \rk G - \rk B$.

Finally consider the definable function which maps $(b_1, c, b_2) \in B\times \check{C} \times B$ to $b_1 c b_2$. We claim that all fibers are finite. Since the fiber over $b_1 c_0 b_2$ has the same rank as the fiber over $c_0$, we compute the latter. Suppose $b_1 c b_2 = c_0$ with obvious notations. Then applying $w$:
\[c_0 = c_0^w = b_1^w c b_2^w = [w, b_1^{-1}] b_1 c b_2 [b_2, w] = [w, b_1^{-1}] c_0 [b_2, w]\]
In particular, $[w, b_1^{-1}]^{c_0} = [b_2, w]^{-1} \in B \cap B^{c_0}$ which is finite by Step~\ref{p:strongembedding:st:B} \ref{p:strongembedding:st:B:vi}. Since $C_B^\circ(w) = 1$ by Step~\ref{p:strongembedding:st:w} \ref{p:strongembedding:st:w:iii}, there are finitely many possibilities for $b_1$ and $b_2$, and $c$ is then determined. So the function has finite fibers, and therefore:
\[\rk \left(B \check{C} B\right) = 2 \rk B + \rk C_G^\circ(w) = \rk \cF + \rk B \geq \rk G\qedhere\]
\end{proofclaim}

We now finish the proof of Proposition~\ref{p:strongembedding}. By Steps \ref{p:strongembedding:st:B} \ref{p:strongembedding:st:B:v} and \ref{p:strongembedding:st:BCB}, both $BI(G)$ and $B\check{C} B$ are generic in $G$. So they intersect; there is an involution $k = b_1 c b_2 \in B \check{C} B$. Conjugating by $b_1$, there is an involution $\ell = c b \in \check{C} B$. Now applying $w$ one finds:
\[\ell^w = c b^w = cb [b, w] = \ell [b, w]\]
which means that $[b, w] \in B$ is a strongly real element. There are two possibilities. If $[b, w] \neq 1$ then by Step~\ref{p:strongembedding:st:B} \ref{p:strongembedding:st:B:iii}, $[b, w] \in A\setminus\{1\}$ and $\ell \in C_G([b, w])$, so $\ell$ and $c$ lie in $B$: a contradiction. If $[b, w] = 1$ then $w$ centralises $b$ and $cb = \ell$: against Step~\ref{p:strongembedding:st:w} \ref{p:strongembedding:st:w:i}.
\end{proof}

\subsection{November}\label{s:theorem}

\begin{theorem*}
Let $\hat{G}$ be a connected, $U_2^\perp$ group of finite Morley rank and $G \trianglelefteq \hat{G}$ be a definable, connected, non-soluble, $N_\circ^\circ$-subgroup.

Then the Sylow $2$-subgroup of $G$ has one of the following structures: isomorphic to that of $\PSL_2(\C)$, isomorphic to that of $\SL_2(\C)$, a $2$-torus of Prüfer $2$-rank at most $2$.

Suppose in addition that for all involutions $\iota \in I(\hat{G})$, the group $C_G^\circ(\iota)$ is soluble.

Then $m_2(\hat{G}) \leq 2$, one of $G$ or $\hat{G}/G$ is $2^\perp$, and involutions are conjugate in $\hat{G}$. Moreover one of the following cases occurs:
\begin{description}
\item[$\bullet$ PSL$_2$:]
$G \simeq \PSL_2(\K)$ in characteristic not $2$; $\hat{G}/G$ is $2^\perp$;
\item[$\bullet$ CiBo$_\emptyset$:]
$G$ is $2^\perp$; $m_2(\hat{G}) \leq 1$; for $\iota \in I(\hat{G})$, $C_G(\iota) = C_G^\circ(\iota)$ is a self-normalising Borel subgroup of $G$;
\item[$\bullet$ CiBo$_1$:]
$m_2(G) = m_2(\hat{G}) = 1$; $\hat{G}/G$ is $2^\perp$; for $i \in I(\hat{G}) = I(G)$, $C_G(i) = C_G^\circ(i)$ is a self-normalising Borel subgroup of $G$;
\item[$\bullet$ CiBo$_2$:]
$\Pr_2(G) = 1$ and $m_2(G) = m_2(\hat{G}) = 2$; $\hat{G}/G$ is $2^\perp$; for $i \in I(\hat{G}) = I(G)$, $C_G^\circ(i)$ is an abelian Borel subgroup of $G$ inverted by any involution in $C_G(i)\setminus\{i\}$ and satisfies $\rk G = 3 \rk C_G^\circ(i)$;
\item[$\bullet$ CiBo$_3$:]
$\Pr_2(G) = m_2(G) = m_2(\hat{G}) = 2$; $\hat{G}/G$ is $2^\perp$; for $i \in I(\hat{G}) = I(G)$, $C_G(i) = C_G^\circ(i)$ is a self-normalising Borel subgroup of $G$; if $i \neq j$ are two involutions of $G$ then $C_G(i) \neq C_G(j)$.
\end{description}
\end{theorem*}

\begin{proof}
\setcounter{step}{0}

\begin{step}\label{t:st:Cisoluble}
We may suppose that for all $\iota \in I(\hat{G})$, $C_G^\circ(\iota)$ is soluble.
\end{step}
\begin{proofclaim}
By the $2$-structure Proposition~\ref{p:2structure}, the Sylow $2$-subgroup of $G$ is isomorphic to that of $\PSL_2(\C)$, to that of $\SL_2(\C)$, or is connected. Our dividing line is based on the Prüfer $2$-rank.

If $\Pr_2(G) \leq 2$ then we are done with the first part of the theorem; since the second and longer part is precisely under the assumption that for all $\iota \in I(\hat{G})$, $C_G^\circ(\iota)$ is soluble, we may proceed if $\Pr_2(G) \leq 2$.

So suppose not: we shall prove a contradiction in Step~\ref{t:st:Prüfer} below. We may assume that $G$ is minimal with $\Pr_2(G) \geq 3$, and that $\hat{G} = G$.
First note that $G/Z(G)$ has Prüfer rank at least $3$ but is centreless. So we may suppose $Z(G) = 1$.
In this setting we actually \emph{prove} that for all $\iota \in I(\hat{G})$, $C_G^\circ(\iota)$ is soluble.

Suppose that there is some involution $i \in G = \hat{G}$ with $C_G^\circ(i)$ non-soluble. Then as in Step~\ref{p:strongembedding:st:centralisers} of Proposition~\ref{p:strongembedding} we take a $2$-torus of rank $1$ $\tau_i$ containing $i$ and $\alpha \in \tau_i$ of minimal order with $C_G^\circ(\alpha)$ soluble; $\alpha^2 \neq 1$. Let $H = C_G^\circ(\alpha^2)$: by torality principles, it has the same Prüfer $2$-rank as $G$, hence by minimality of $G$ as a counterexample $H = G$ and $\alpha^2 \in Z(G)$, a contradiction.

So if $G$ is minimal with $\Pr_2(G) \geq 3$, then for all $\iota \in I(\hat{G}) = I(G)$, $C_G^\circ(\iota)$ is soluble. We proceed under the assumption.
\end{proofclaim}

\begin{step}\label{t:st:W2perp}
We may suppose that $G$ is $W_2^\perp$.
\end{step}
\begin{proofclaim}
Suppose $G$ is not. By the $2$-Structure Proposition~\ref{p:2structure} and since centralisers$^\circ$ in $G$ of involutions are soluble, the Sylow $2$-subgroup of $G$ is isomorphic to that of $\PSL_2(\C)$, that is $\Pr_2(G) = 1$ and $m_2(G) = 2$. Fix $i \in I(G)$.

If $C_G^\circ(i)$ is contained in at least two Borel subgroups of $G$, then by the Algebraicity Proposition~\ref{p:algebraicity}, $G \simeq \PSL_2(\K)$ for some algebraically closed field of characteristic not $2$. The latter has no outer automorphisms \cite[Theorem 8.4]{BNGroups}; by the assumption on centralisers of involutions, $\hat{G}/G$ is $2^\perp$ and we are in case \textbf{PSL$_2$}.

So we may assume that $C_G^\circ(i)$ is contained in a unique Borel subgroup of $G$. We then apply the Dihedral Proposition~\ref{p:dihedral} inside $\check{G} = G$ to find that $C_G^\circ(i)$ is an abelian Borel subgroup of $G$ inverted by any involution in $C_G(i)\setminus\{i\}$. By torality principles in $G$ there exist Sylow $2$-subgroups of $G$, say $S_i = S_i^\circ \rtimes\langle  w\rangle$ with $i \in S_i^\circ$, and $S_w = S_w^\circ \rtimes \langle  i\rangle$ likewise. In order to reach case \textbf{CiBo$_2$} one first shows that $\hat{G}/G$ is $2^\perp$; only the rank estimate will remain to prove.

If $\hat{G}/G$ is not $2^\perp$ then $S_i$ is no Sylow $2$-subgroup of $\hat{G}$. Let $\hat{S} \leq \hat{G}$ be a Sylow $2$-subgroup containing $S_i$ properly; it is folklore that $\Pr_2(\hat{S}) \geq 2$. Since $\hat{S}^\circ$ is $2$-divisible and invariant under $\omega \in \hat{S}$, we may apply Maschke's Theorem (see for instance \cite[Fact 2]{DpRank}) to find a quasi-complement, i.e. a $w$-invariant $2$-torus $\hat{T}$ with $\hat{S}^\circ = S_i^\circ \qoplus \hat{T}$. Then using Zilber's indecomposibility theorem, $[\hat{T}, w] \leq (\hat{T}\cap G)^\circ =1$, that is, $w$ centralises $\hat{T}$. It follows that $\hat{T}$ normalises both $C_G^\circ(i)$ and $C_G^\circ(w)$; by the rigidity of tori, it centralises therefore both $S_i^\circ$ and $S_w^\circ$. Hence $S_i^\circ\rtimes\langle  w\rangle \leq \langle  S_i^\circ, S_w^\circ\rangle \leq C_G^\circ(\hat{T})$, so by the structure of torsion in connected, soluble groups, $C_G^\circ(\hat{T})$ may not be soluble. As $\hat{T} \not \leq G$ this does not contradict $G$ being $N_\circ^\circ$, but this is against the fact that $\hat{T} \neq 1$ contains an involution of $\hat{G}$, which has soluble centraliser$^\circ$ by assumption.

Hence $\hat{G}/G$ is $2^\perp$; we finally show $\rk G = 3 \rk C_G^\circ(i)$. This exactly follows \cite[Proposition 4.1.30 and Corollaire 4.1.31]{DGroupes} or \cite[Proposition 3.26 and Corollaire 3.27]{DGroupes2}: since $C_G(i)$ is not connected for $i \in I(G)$, using the map from \cite[\S5]{BBCInvolutions} (some day we shall return to this) one sees that generic, independent $j, k \in I(G)$ are such that $d(jk)$ is not $2$-divisible, and we let $\ell$ be the only involution in $d(jk)$. Then $(j, k) \mapsto\ell$ is a (generically) well-defined, definable function; obvious rank computations yield $\rk G = 3 \rk C_G^\circ(i)$.
\end{proofclaim}

\begin{notationinproof}
For $\iota \in I(\hat{G})$ let $B_\iota = C_G^\circ(\iota)$.
\end{notationinproof}

By Steps \ref{t:st:Cisoluble} and \ref{t:st:W2perp} and the Maximality Proposition~\ref{p:maximality}, $B_\iota$ is a Borel subgroup of $G$.

\begin{step}\label{t:st:Prüfer}
$\Pr_2(\hat{G}) \leq 2$.
\end{step}
\begin{proofclaim}
Suppose $\Pr_2(\hat{G}) \geq 3$. We may assume that $\hat{G} = G\cdot d(\hat{S})$ for some maximal $2$-torus $\hat{S}$ of $\hat{G}$. In particular $\hat{G}/G$ is $W_2^\perp$. But so is $G$ by Step~\ref{t:st:W2perp}; by Lemma \ref{l:W2perp:factor}, so is $\hat{G}$, i.e. $\hat{S}$ is actually a Sylow $2$-subgroup of $\hat{G}$. Let $A = \Omega_2(\hat{S})$ be the group generated by the involutions of $\hat{S}$; then $A \leq \hat{G}$ is an elementary abelian $2$-group with $2$-rank $\Pr_2(\hat{G}) \geq 3$. Let $\rho = \max\{\rho_{B_\iota}: \iota \in A\setminus\{1\}\}$ and $\iota\in A\setminus\{1\}$ be such that $\rho_{B_\iota} = \rho$.

We show that for any involution $\lambda \in A\setminus\{1\}$, $B_\lambda = B_\iota$.
Let $\kappa \in A\setminus\langle  \iota\rangle$ be such that $X = C_{U_\rho(Z(F^\circ(B_\iota)))}^\circ(\kappa) \neq 1$; this exists as $A$ has rank at least $3$. Then $X \leq C_G^\circ(\kappa) = B_\kappa$, so $\rho_\kappa = \rho$ and $X \leq U_\rho(B_\kappa)$. Let as always $\hat{B}_\iota = B_\iota \cdot d(\hat{S})$; one has $\{B_\iota, \kappa\}\subseteq (\hat{B}_\iota' \cap B)^\circ \leq F^\circ(B_\iota)$ so we may apply Lemma \ref{l:involutiveaction:soluble} and write $B_\iota = B_\iota^{+_\kappa} \cdot \{B_\iota, \kappa\} \subseteq B_\iota^+ \cdot F^\circ(B_\iota)$.
Now both $B_\iota^+$ and $F^\circ(B_\iota)$ normalise $X$, hence $X$ is normal in $B_\iota$. Uniqueness principles imply that $U_\rho(B_\iota)$ is the only Sylow $\rho$-subgroup of $G$ containing $X$. In particular $U_\rho(B_\iota) = U_\rho(B_\kappa)$.
Hence $C_G^\circ(\iota) = B_\iota = B_\kappa = C_G^\circ(\kappa) = C_G^\circ(\iota\kappa)$. Turning to an arbitrary $\lambda \in A\setminus\{1\}$ we apply bigeneration, Fact \ref{f:bigeneration}, to the action of $V = \langle  \iota, \kappa\rangle$ on the soluble group $B_\lambda$, and find $B_\lambda \leq \langle  C_{B_\lambda}^\circ(\mu): \mu \in V\setminus\{1\}\rangle \leq B_\iota$.
So $B_\lambda = B_\iota$ for any $\lambda \in A\setminus\{1\}$.

We claim that $\Pr_2(G) = 1$. First, if $G$ is $2^\perp$ then bigeneration applies and we find $G = \langle  C_G^\circ(\mu): \mu \in V\setminus\{1\}\rangle = B_\iota$, a contradiction. Therefore $G$ has involutions. In order to bound its Prüfer $2$-rank we shall use the Strong Embedding Proposition~\ref{p:strongembedding}.
We argue that $M = N_G(B_\iota)$ is strongly embedded in $G$. For let $j$ be an involution in $S = \hat{S}\cap G$, which is a Sylow $2$-subgroup of $G$; then $j \in N_G(B_\iota)$. But $G$ is $W_2^\perp$ and $B_\iota$ contains a maximal $2$-torus of $G$, so $j \in B_\iota$. Let $V = \langle  \iota, \kappa\rangle$; recall that $V$ centralises $B_\iota$. In particular $V$ centralises $j$, and normalises $B_j$. As the latter is soluble we apply bigeneration, Fact \ref{f:bigeneration}, and find $B_j = \langle  C_{B_j}^\circ(\lambda): \lambda \in V\setminus\{1\}\rangle \leq B_\iota$. Now if $j \in M^x$ with $x \in G$, then we argue likewise: $j \in B_\iota^x$, so $V^x$ centralises $j$, hence $V^x$ normalises $B_j$, and $B_j = B_\iota^x$. Therefore $x \in N_G(B_\iota)$, and $M = N_G(B_\iota)$ is strongly embedded in $G$.
By the Strong Embedding Proposition~\ref{p:strongembedding}, $\Pr_2(G) = 1$, as desired.

Observe that any two commuting involutions of $\hat{G}$ centralise the same Borel subgroup of $G$: for if $\langle  \mu, \nu\rangle$ is a four-subgroup of $\hat{G}$ then up to conjugacy $\langle  \mu, \nu\rangle \leq A$, so $B_\mu = B_\nu$. Now any two non-conjugate involutions of $\hat{G}$ commute to a third involution, so they centralise the same Borel subgroup of $G$. But there are at least two conjugacy classes of involutions in $\hat{G}$, since $\Pr_2(G) = 1$ and $\Pr_2(\hat{G}) \geq 3$. So actually any two involutions of $\hat{G}$ centralise the same Borel subgroup of $G$. This is to mean: for any $g \in G$, $B_\iota^g = B_\iota$; $B_\iota$ is normal in $G$, which contradicts $G$ being $N_\circ^\circ$.
\end{proofclaim}

\begin{step}\label{t:st:NBi=Bi}
Let $\iota \in I(\hat{G})$. If $\iota \in G$ or $G$ is $2^\perp$, then $B_\iota$ is self-normalising in $G$.
\end{step}
\begin{proofclaim}
First suppose $i = \iota \in I(G)$. We claim that $i$ is the only involution in $Z(B_i)$. If $\Pr_2(G) = 1$ this is clear by the structure of torsion in connected, soluble groups. If $\Pr_2(G) \geq 2$ (and one has equality by Step~\ref{t:st:Prüfer}), then let $k \in I(B_i)\setminus\{i\}$: if $k \in Z(B_i)$ then $B_k = B_i = B_{ik}$ is clearly strongly embedded, against Proposition~\ref{p:strongembedding}.
In particular, $N_G(B_i) \leq B_i \cdot C_G(i) \leq C_G(i) = C_G^\circ(i) = B_i$ by Steinberg's torsion theorem and connectedness of the Sylow $2$-subgroup of $G$ (Step~\ref{t:st:W2perp}).

Now suppose that $G$ is $2^\perp$ (this case was already covered in Proposition~\ref{p:dihedral}, between Steps \ref{p:dihedral:st:G2perp} and \ref{p:dihedral:st:Bi-k}).
Since $N = N_G(B_\iota) \leq G$ is $2^\perp$ it admits a decomposition $N = N^{+_\iota}\cdot N^{-_\iota}$ under the action of $\iota$. But on the one hand so does $G$: hence $G = C_G(\iota) \cdot G^{-_\iota}$ with trivial fibers and by a degree argument $C_G(\iota)$ is connected, so $N^+ \leq B_\iota$. And on the other hand, by torality principles there exists a $2$-torus $\hat{S}^\circ$ of $\hat{G}$ containing $\iota$; $\hat{S}^\circ$ normalises $B_\iota$ and $N_G(B_\iota)$. By connectedness, $\hat{S}^\circ$ centralises the finite group $N_G(B_\iota)/B_\iota$, and so does $\iota$. So $N^- \subseteq B_\iota$ and therefore $N = B_\iota$.
\end{proofclaim}

\begin{notationinproof}
For $\kappa, \lambda \in I(\hat{G})$ let $T_\kappa(\lambda) = T_{B_\kappa}(\lambda)$.
\end{notationinproof}

Before reading the following be very careful that Inductive Torsion Control, Proposition~\ref{p:Yanartas}, requires $\hat{G}$ to be $W_2^\perp$; for the moment only $G$ need be by Step~\ref{t:st:W2perp}.

\begin{step}[Antalya]\label{t:st:conjugacy}
If $\hat{G}$ is $W_2^\perp$ and $\lambda \notin N_{\hat{G}}(B_\kappa)$ then $T_\kappa(\lambda)$ is finite.

If in addition $\hat{G} = G \cdot d(\hat{S}^\circ)$ for some maximal $2$-torus $\hat{S}^\circ \leq \hat{G}$, then $\rk C_{\hat{G}}^\circ(\kappa) = \rk C_{\hat{G}}^\circ(\lambda)$ and the generic left translate $\hat{g} C_{\hat{G}}^\circ(\lambda)$ contains a conjugate of $\kappa$.
\end{step}
\begin{proofclaim}
Suppose that $\hat{G}$ is $W_2^\perp$ and $T_\kappa(\lambda)$ is infinite. Then by Inductive Torsion Control, Proposition~\ref{p:Yanartas}, $\T_\kappa(\lambda)$ is infinite and contains no torsion elements. Then $\lambda$ inverts $\T_\kappa(\lambda)$ pointwise, and normalises $C_{\hat{G}}(\T_\kappa(\lambda))$; the latter contains $\kappa$. By the structure of the Sylow $2$-subgroup of $\hat{G}$ and normalisation principles, $\lambda$ has a $C_{\hat{G}}(\T_\kappa(\lambda))$-conjugate $\mu$ commuting with $\kappa$. Now $\mu$ inverts $\T_\kappa(\lambda)$ and normalises $B_\kappa$. Since $N_{\hat{G}}(B_\kappa)$ already contains a Sylow $2$-subgroup of $\hat{G}$ which is a $2$-torus, $\mu$ is toral in $N_{\hat{G}}(B_\kappa)$ by torality principles. Hence $\T_\kappa(\lambda) \subseteq \{B, \mu\} \subseteq F^\circ(B)$. We now take any $t \in \T_\kappa (\lambda)\setminus\{1\}$ and $X = d(t)$, and we climb the Devil's Ladder, Proposition~\ref{p:DevilsLadder}: $B_\kappa$ is the only Borel subgroup of $G$ containing $C_G^\circ(X)$. In particular, $\lambda$ normalises $B_\kappa$, a contradiction.

For the rest of the argument we assume in addition that $\hat{G} = G \cdot d(\hat{S}^\circ)$ for some maximal $2$-torus $\hat{S}^\circ \leq \hat{G}$; in particular $\hat{G}$ is $W_2^\perp$ by Step~\ref{t:st:W2perp} and Lemma \ref{l:W2perp:factor}, but also $\hat{G}/G$ is abelian.

Let us introduce the following definable maps:
\[\begin{array}{cccc}
\pi_{\kappa, \lambda}: & \kappa^{\hat{G}}\setminus N_{\hat{G}}(B_\lambda) & \to & \hat{G}/C_{\hat{G}}^\circ(\lambda)\\
& \kappa_1 & \mapsto & \kappa_1 C_{\hat{G}}^\circ(\lambda)
\end{array}\]
We shall compute fibers.

Suppose that $\pi_{\kappa, \lambda}(\kappa_1) = \pi_{\kappa, \lambda}(\kappa_2)$. Then by the assumption that $\hat{G} = G \cdot d(\hat{S}^\circ)$, $G$ controls $\hat{G}$-conjugacy of involutions. Hence $\kappa_1\kappa_2 \in C_{\hat{G}}^\circ(\lambda) \cap G \leq C_G(\lambda)$. Be very careful that we do not a priori have connectedness of the latter, insofar as there is no outer version of Steinberg's torsion theorem; as a matter of fact connectedness is immediate only when $G$ is $2^\perp$ or $\lambda \in G$, not in general.

But if $c \in C_G(\lambda)$ is inverted by $\kappa$, then $\kappa$ normalises $C_{\hat{G}}(c)$ which contains $\lambda$; since $\hat{G}$ is $W_2^\perp$ and by normalisation principles, $\kappa$ has a $C_{\hat{G}}(c)$-conjugate $\mu$ commuting with $\lambda$. Now $\mu \in N_{\hat{G}}(C_G(\lambda))$ which contains a maximal $2$-torus by torality principles; torality principles again provide some maximal $2$-torus $T_\mu \leq N_{\hat{G}}(C_G(\lambda))$ containing $\mu$. Then by Zilber's indecomposibility theorem, $[c, \mu] \in [c, T_\mu] \leq C_G^\circ(\lambda)$, that is, $c^2 \in C_G^\circ(\lambda)$.
If $G$ is $2^\perp$ the conclusion comes easily; if $G$ contains involutions, then by torality principles $C_G^\circ(\lambda)$ contains a maximal $2$-torus of $G$ which is a Sylow $2$-subgroup of $G$ by Step~\ref{t:st:W2perp}, so $c \in C_G^\circ(\lambda)$.

Turning back to our fiber computation, we do have $\kappa_1 \kappa_2 \in C_G^\circ(\lambda)$, and $\kappa_1 \kappa_2 \in T_\lambda(\kappa)$. The latter is finite as first proved. Hence $\pi_{\kappa, \lambda}$ has finite fibers; it follows, keeping the Genericity Proposition~\ref{p:genericity} in mind:
\[\rk \kappa^{\hat{G}} \leq \rk \hat{G} - \rk C_{\hat{G}}^\circ(\lambda)\]
that is, $\rk C_{\hat{G}}^\circ(\lambda) \leq \rk C_{\hat{G}}^\circ(\kappa)$, and vice-versa. So equality holds.
By a degree argument, $\pi_{\kappa, \lambda}$ is now generically onto.
\end{proofclaim}

\begin{step}\label{t:st:PrhatG=1}
We may suppose that $\Pr_2(\hat{G}) = 1$.
\end{step}
\begin{proofclaim}
Suppose that $\Pr_2(\hat{G}) \geq 2$; equality follows from Step~\ref{t:st:Prüfer} and we aim at finding case \textbf{CiBo$_3$}. There seem to be three cases depending on the values of $\Pr_2(G)$ and $\Pr_2(\hat{G}/G) = 2 - \Pr_2(G)$. We give a common argument. Notice that we however rely on Step~\ref{t:st:Prüfer}, to the author's great aesthetic discontentment.

Let $\hat{S}^\circ \leq \hat{G}$ be a maximal $2$-torus of $\hat{G}$ and $\check{G} = G \cdot d(\hat{S}^\circ)$. Bear in mind that $\check{G}$ is $W_2^\perp$ by Step~\ref{t:st:W2perp} and Lemma \ref{l:W2perp:factor}. In particular, $\hat{S}^\circ$ is a Sylow $2$-subgroup of $\hat{G}$.
Let $\kappa, \lambda, \mu$ be the three involutions in $\hat{S}^\circ$.

If $\kappa$, $\lambda$ and $\mu$ are not pairwise $G$-conjugate, then they are not $\check{G}$-conjugate either. So $\check{G}$ has at least (hence exactly) three conjugacy classes of involutions by Lemma \ref{l:Goreme}: $\kappa$, $\lambda$ and $\mu$ are pairwise not $G$-conjugate.
We apply Step~\ref{t:st:conjugacy} in $\check{G}$. The generic left-translate $\check{g} C_{\check{G}}^\circ(\lambda)$ contains both a conjugate $\kappa_1$ of $\kappa$ and a conjugate $\mu_1$ of $\mu$. Now $\kappa_1$ and $\mu_1$ are not $\check{G}$-conjugate so $d(\kappa_1\mu_1)$ contains an involution $\nu$.
By the structure of the Sylow $2$-subgroup of $\check{G}$, $\nu$ must be a conjugate $\lambda_1$ of $\lambda$. Of course $\lambda_1 \in d(\kappa_1\mu_1) \leq C_{\check{G}}^\circ(\lambda)$. By the structure of the Sylow $2$-subgroup of $\check{G}$ again, $\lambda$ is the only conjugate of $\lambda$ in its centraliser. Hence $\lambda_1 = \lambda$. It follows that $\kappa_1, \mu_1 \in C_{\check{G}}(\lambda)$, and $\check{g} \in C_{\check{G}}(\lambda)$: a contradiction to genericity of $\check{g}C_{\check{G}}^\circ(\lambda)$ in $\check{G}/C_{\check{G}}^\circ(\lambda)$.

So involutions in $\check{G}$ are $G$-conjugate. This certainly rules out the case where $\Pr_2(G) = 1 = \Pr_2(\check{G}/G)$. Actually this also eliminates the case where $\Pr_2(G) = 0$ and $\Pr_2(\check{G}/G) = 2$. For in that case, $\kappa, \lambda, \mu$ remain distinct in the quotient $\check{G}/G$: so $G$ cannot conjugate them in $\check{G}$.

Hence $\Pr_2(G) = 2$ and by Step~\ref{t:st:Prüfer}, $\hat{G}/G$ is $2^\perp$. We have proved that $G$ conjugates its involutions; by Step~\ref{t:st:NBi=Bi} their centralisers$^\circ$ in $G$ are self-normalising Borel subgroups. Notice that if $i\neq j$ are two involutions of $G$ with $B_i = B_j$ then $i \in C_G^\circ(j)$ so $i$ and $j$ commute; now $B_i = B_j = B_{ij}$ is strongly embedded in $G$, against Proposition~\ref{p:strongembedding}. We recognize case \textbf{CiBo$_3$}.
\end{proofclaim}

This is the end.
If $G$ has involutions then by Steps \ref{t:st:W2perp} and \ref{t:st:PrhatG=1}, $m_2(G) = \Pr_2(G) = 1$ and $\Pr_2(\hat{G}/G) = 0 = m_2(\hat{G}/G)$: with a look at Step~\ref{t:st:NBi=Bi} this is case \textbf{CiBo$_1$}.
So we may suppose that $G$ is $2^\perp$. Since $\Pr_2 (\hat{G}) = 1$, Proposition~\ref{p:dihedral} yields $m_2(\hat{G}) = 1$. With a look at Step~\ref{t:st:NBi=Bi} this is case \textbf{CiBo$_\emptyset$}.
\renewcommand{\qedsymbol}{\textsc{In Memoriam}}
\end{proof}


\section*{References}
References are divided into three categories:
\begin{itemize}
\item
Finite Groups, with keys of the form \cite{BUniqueness} (Author, Year);
\item
Groups of finite Morley rank, with keys of the form \cite{ABAnalogies} (Author-Year, short);
\item
Stages of Development, with keys of the form \cite{JFT} (chronological order).
\end{itemize}

\renewcommand{\refname}{Finite Groups}

\renewcommand{\refname}{Groups of Finite Morley Rank}

\renewcommand{\refname}{Stages of Development (chronological in spirit)}


\begin{thebibliography}{}
\bibitem[{Bender, }1970]{BUniqueness}
Helmut Bender.
\newblock On the uniqueness theorem.
\newblock {\em Illinois J. Math.}, 14:376--384, 1970.

\bibitem[{Bender, }1971]{BTransitive}
Helmut Bender.
\newblock Transitive {G}ruppen gerader {O}rdnung, in denen jede {I}nvolution
  genau einen {P}unkt festl\"a\ss t.
\newblock {\em J. Algebra}, 17:527--554, 1971.

\bibitem[{Bender, }1974a]{BBrauer}
Helmut Bender.
\newblock The {B}rauer-{S}uzuki-{W}all theorem.
\newblock {\em Illinois J. Math.}, 18:229--235, 1974.

\bibitem[{Bender, }1974b]{BFiniteLarge}
Helmut Bender.
\newblock Finite groups with large subgroups.
\newblock {\em Illinois J. Math.}, 18:223--228, 1974.

\bibitem[{Bender, }1981]{BFiniteDihedral}
Helmut Bender.
\newblock Finite groups with dihedral {S}ylow {$2$}-subgroups.
\newblock {\em J. Algebra}, 70(1):216--228, 1981.

\bibitem[{Borovik, }1984]{BClassification}
A.~V. Borovik.
\newblock Classification of periodic linear groups over fields of odd
  characteristic.
\newblock {\em Sibirsk. Mat. Zh.}, 25(2):67--83, 1984.

\bibitem[{Brauer \bgroup \em et al.\egroup ,
  }1958]{BSWCharacterization}
Richard Brauer, Michio Suzuki, and Gordon Wall.
\newblock A characterization of the one-dimensional unimodular projective
  groups over finite fields.
\newblock {\em Illinois J. Math.}, 2:718--745, 1958.

\bibitem[{Goldschmidt, }1974]{GElements}
David Goldschmidt.
\newblock Elements of order two in finite groups.
\newblock {\em Delta}, 4:45--58, 1974.

\bibitem[{Gorenstein and Lyons, }1976]{GLNonsolvable}
Daniel Gorenstein and Richard Lyons.
\newblock Nonsolvable finite groups with solvable {$2$}-local subgroups.
\newblock {\em J. Algebra}, 38(2):453--522, 1976.

\bibitem[{Thomas, }1983]{TClassification}
Simon Thomas.
\newblock The classification of the simple periodic linear groups.
\newblock {\em Arch. Math. (Basel)}, 41(2):103--116, 1983.

\bibitem[{Thompson, }1968]{TNonsolvable}
John Thompson.
\newblock Nonsolvable finite groups all of whose local subgroups are solvable.
\newblock {\em Bull. Amer. Math. Soc.}, 74:383--437, 1968.

\bibitem[{Zassenhaus, }1935]{ZKennzeichnung}
Hans Zassenhaus.
\newblock Kennzeichnung endlicher linearer {G}ruppen als {P}ermutationsgruppen.
\newblock {\em Abh. Math. Sem. Univ. Hamburg}, 11(1):17--40, 1935.
\end{thebibliography}

\begin{thebibliography}{}

\bibitem[AB08]{ABAnalogies}
Tuna Alt{\i}nel and Jeffrey Burdges.
\newblock On analogies between algebraic groups and groups of finite {M}orley
  rank.
\newblock {\em J. Lond. Math. Soc. (2)}, 78(1):213--232, 2008.

\bibitem[ABC08]{ABCSimple}
Tuna Alt{\i}nel, Alexandre Borovik, and Gregory Cherlin.
\newblock {\em Simple groups of finite {M}orley rank}, volume 145 of {\em
  Mathematical Surveys and Monographs}.
\newblock American Mathematical Society, Providence, RI, 2008.

\bibitem[ABF13]{ABFWeyl}
Tuna Alt{\i}nel, Jeffrey Burdges, and Olivier Fr{\'e}con.
\newblock On {W}eyl groups in minimal simple groups of finite {M}orley rank.
\newblock {\em Israel J. Math.}, 197(1):377--407, 2013.

\bibitem[AC99]{ACCentral}
Tuna Alt{\i}nel and Gregory Cherlin.
\newblock On central extensions of algebraic groups.
\newblock {\em J. Symbolic Logic}, 64(1):68--74, 1999.

\bibitem[Bau96]{BNew}
Andreas Baudisch.
\newblock A new uncountably categorical group.
\newblock {\em Trans. Amer. Math. Soc.}, 348(10):3889--3940, 1996.

\bibitem[BBC07]{BBCInvolutions}
Alexandre Borovik, Jeffrey Burdges, and Gregory Cherlin.
\newblock Involutions in groups of finite {M}orley rank of degenerate type.
\newblock {\em Selecta Math. (N.S.)}, 13(1):1--22, 2007.

\bibitem[BC08]{BCGeneration}
Jeffrey Burdges and Gregory Cherlin.
\newblock A generation theorem for groups of finite {M}orley rank.
\newblock {\em J. Math. Log.}, 8(2):163--195, 2008.

\bibitem[BC09]{BCSemisimple}
Jeffrey Burdges and Gregory Cherlin.
\newblock Semisimple torsion in groups of finite {M}orley rank.
\newblock {\em J. Math Logic}, 9(2):183--200, 2009.

\bibitem[BD10]{BDWeyl}
Jeffrey Burdges and Adrien Deloro.
\newblock Weyl groups of small groups of finite {M}orley rank.
\newblock {\em Israel J. Math.}, 179:403--423, 2010.

\bibitem[BD15]{BDActions}
Alexandre Borovik and Adrien Deloro.
\newblock Rank 3 Bingo.
\newblock {\em J. Symbolic Logic}, to appear, 2016.

\bibitem[BDN94]{BDNCIT}
Alexandre Borovik, Mark DeBonis, and Ali Nesin.
\newblock C{IT} groups of finite {M}orley rank. {I}.
\newblock {\em J. Algebra}, 165(2):258--272, 1994.

\bibitem[BN94a]{BNCIT2}
Alexandre Borovik and Ali Nesin.
\newblock C{IT} groups of finite {M}orley rank. {II}.
\newblock {\em J. Algebra}, 165(2):273--294, 1994.

\bibitem[BN94b]{BNGroups}
Alexandre Borovik and Ali Nesin.
\newblock {\em Groups of finite {M}orley rank}, volume~26 of {\em Oxford Logic
  Guides}.
\newblock The Clarendon Press - Oxford University Press, New York, 1994.
\newblock Oxford Science Publications.

\bibitem[Bor82]{BInvolutions}
Alexandre Borovik.
\newblock {\em Involutions in groups with dimension}.
\newblock Preprint. Academy of Sciences of the {U}{S}{S}{R}, {S}iberian branch,
  Novosibirsk, 1982.

\bibitem[Bor95]{BSimplelocally}
Alexandre Borovik.
\newblock Simple locally finite groups of finite {M}orley rank and odd type.
\newblock In {\em Finite and locally finite groups (Istanbul, 1994)}, volume
  471 of {\em NATO Adv. Sci. Inst. Ser. C Math. Phys. Sci.}, pages 247--284.
  Kluwer Acad. Publ., Dordrecht, 1995.

\bibitem[BP90]{BPTores}
Aleksandr~Vasilievich Borovik and Bruno~Petrovich Poizat.
\newblock Tores et {$p$}-groupes.
\newblock {\em J. Symbolic Logic}, 55(2):478--491, 1990.

\bibitem[Bur04a]{BSignalizer}
Jeffrey Burdges.
\newblock A signalizer functor theorem for groups of finite {M}orley rank.
\newblock {\em J. Algebra}, 274:215--229, 2004.

\bibitem[Bur04b]{BSimple}
Jeffrey Burdges.
\newblock {\em Simple Groups of Finite {M}orley Rank of Odd and Degenerate
  Type}.
\newblock PhD thesis, Rutgers University, New Brunswick, New Jersey, 2004.

\bibitem[Bur06]{BSylow}
Jeffrey Burdges.
\newblock Sylow 0-unipotent subgroups in groups of finite {M}orley rank.
\newblock {\em J. Group Theory}, 9(4):467, 2006.

\bibitem[Bur07]{BBender}
Jeffrey Burdges.
\newblock The {B}ender method in groups of finite {M}orley rank.
\newblock {\em J. Algebra}, 307(2):704--726, 2007.

\bibitem[Bur09]{BFrattini}
Jeffrey Burdges.
\newblock On {F}rattini arguments in {L}$^\ast$ groups of finite {M}orley rank.
\newblock Preprint. arXiv:0904-3027, 2009.

\bibitem[CD12]{CDSmall}
Gregory Cherlin and Adrien Deloro.
\newblock Small representations of {$\rm SL_2$} in the finite {M}orley rank
  category.
\newblock {\em J. Symbolic Logic}, 77(3):919--933, 2012.

\bibitem[Che79]{CGroups}
Gregory Cherlin.
\newblock Groups of small {M}orley rank.
\newblock {\em Ann. Math. Logic}, 17(1-2):1--28, 1979.

\bibitem[Che05]{CGood}
Gregory Cherlin.
\newblock Good tori in groups of finite {M}orley rank.
\newblock {\em J. Group Theory}, 8(5):613--621, 2005.

\bibitem[Del09a]{DActions}
Adrien Deloro.
\newblock Actions of groups of finite {M}orley rank on small abelian groups.
\newblock {\em Bull. Symb. Log.}, 15(1):70--90, 2009.

\bibitem[Del09b]{DSteinberg}
Adrien Deloro.
\newblock Steinberg's torsion theorem in the context of groups of finite
  {M}orley rank.
\newblock {\em J. Group Theory}, 12(5):709--710, 2009.

\bibitem[Del12]{DpRank}
Adrien Deloro.
\newblock {$p$}-rank and {$p$}-groups in algebraic groups.
\newblock {\em Turkish J. Math.}, 36(4):578--582, 2012.

\bibitem[DMT08]{DMTSpecial}
Tom De~Medts and Katrin Tent.
\newblock Special abelian {M}oufang sets of finite {M}orley rank.
\newblock {\em J. Group Theory}, 11(5):645--655, 2008.

\bibitem[DN93]{DNSharply}
Franz Delahan and Ali Nesin.
\newblock Sharply {$2$}-transitive groups revisited.
\newblock {\em Do\u ga Mat.}, 17(1):70--83, 1993.

\bibitem[DN95]{DNZassenhaus}
Franz Delahan and Ali Nesin.
\newblock On {Z}assenhaus groups of finite {M}orley rank.
\newblock {\em Comm. Algebra}, 23(2):455--466, 1995.

\bibitem[FJ05]{FJExistence}
Olivier Fr{\'e}con and {\'E}ric Jaligot.
\newblock The existence of {C}arter subgroups in groups of finite {M}orley
  rank.
\newblock {\em J. Group Theory}, 8(5):623--633, 2005.

\bibitem[FJ08]{FJConjugacy}
Olivier Fr{\'e}con and {\'E}ric Jaligot.
\newblock Conjugacy in groups of finite {M}orley rank.
\newblock In {\em Model theory with applications to algebra and analysis.
  {V}ol. 2}, volume 350 of {\em London Math. Soc. Lecture Note Ser.}, pages
  1--58. Cambridge Univ. Press, Cambridge, 2008.

\bibitem[Fr{\'e}00a]{FSous}
Olivier Fr{\'e}con.
\newblock Sous-groupes anormaux dans les groupes de rang de {M}orley fini
  r\'esolubles.
\newblock {\em J. Algebra}, 229(1):118--152, 2000.

\bibitem[Fr{\'e}00b]{FHall}
Olivier Fr{\'e}con.
\newblock Sous-groupes de {H}all g\'en\'eralis\'es dans les groupes
  r\'esolubles de rang de {M}orley fini.
\newblock {\em J. Algebra}, 233(1):253--286, 2000.

\bibitem[Fr{\'e}06]{FAround}
Olivier Fr{\'e}con.
\newblock Around unipotence in groups of finite {M}orley rank.
\newblock {\em J. Group Theory}, 9(3):341--359, 2006.

\bibitem[Fr{\'e}10]{FAutomorphisms}
Olivier Fr{\'e}con.
\newblock Automorphism groups of small simple groups of finite {M}orley rank.
\newblock {\em Proc. Amer. Math. Soc.}, 138(7):2591--2599, 2010.

\bibitem[GH93]{GHStable}
Claus Gr{\"u}nenwald and Frieder Haug.
\newblock On stable torsion-free nilpotent groups.
\newblock {\em Arch. Math. Logic}, 32(6):451--462, 1993.

\bibitem[Hru89]{HAlmost}
Ehud Hrushovski.
\newblock Almost orthogonal regular types.
\newblock {\em Ann. Pure Appl. Logic}, 45(2):139--155, 1989.
\newblock Stability in model theory, II (Trento, 1987).

\bibitem[Nes90a]{NSharply}
Ali Nesin.
\newblock On sharply {$n$}-transitive superstable groups.
\newblock {\em J. Pure Appl. Algebra}, 69(1):73--88, 1990.

\bibitem[Nes90b]{NSplit}
Ali Nesin.
\newblock On split {$B$}-{$N$} pairs of rank {$1$}.
\newblock {\em Compositio Math.}, 76(3):407--421, 1990.

\bibitem[Poi87]{PGroupes}
Bruno Poizat.
\newblock {\em Groupes stables}.
\newblock Nur al-Mantiq wal-Ma'rifah, Villeurbanne, 1987.

\bibitem[Poi01]{PStable}
Bruno Poizat.
\newblock {\em Stable groups}, volume~87 of {\em Mathematical Surveys and
  Monographs}.
\newblock American Mathematical Society, Providence, RI, 2001.
\newblock Translated from the 1987 French original by Moses Gabriel Klein.

\bibitem[PW93]{PWSous-groupes}
Bruno Poizat and Frank Wagner.
\newblock Sous-groupes p{\'e}riodiques d'un groupe stable.
\newblock {\em J. Symbolic Logic}, 58(2):385--400, 1993.

\bibitem[PW00]{PWLiftez}
Bruno Poizat and Frank Wagner.
\newblock Liftez les {S}ylows ! {U}ne suite {\`a} ``{S}ous-groupes
  p{\'e}riodiques d'un groupe stable''.
\newblock {\em J. Symbolic Logic}, 65(2):703--704, 2000.

\bibitem[Wag94]{WNilpotent}
Frank Wagner.
\newblock Nilpotent complements and {C}arter subgroups in stable
  {$\mathfrak{R}$}-groups.
\newblock {\em Arch. Math. Logic}, 33(1):23--34, 1994.

\bibitem[Wag01]{WFields}
Frank Wagner.
\newblock Fields of finite {M}orley rank.
\newblock {\em J. Symbolic Logic}, 66(2):703--706, 2001.

\bibitem[Wis11]{WGroups}
Josh Wiscons.
\newblock On groups of finite {M}orley rank with a split {$BN$}-pair of rank 1.
\newblock {\em J. Algebra}, 330:431--447, 2011.

\end{thebibliography}

\begin{thebibliography}{}

\bibitem{JFT}
{\'E}ric Jaligot.
\newblock F{T}-{G}roupes.
\newblock {\em Prépublications de l'institut Girard Desargues}, 33, 2000.

\bibitem{JTame}
{\'E}ric Jaligot.
\newblock Tame {FT}-{G}roups of {P}r{\"u}fer $2$-rank $1$.
\newblock Preprint, 2002.

\bibitem{CJTame}
Gregory Cherlin and {\'E}ric Jaligot.
\newblock Tame minimal simple groups of finite {M}orley rank.
\newblock {\em J. Algebra}, 276(1):13--79, 2004.

\bibitem{BCJMinimal}
Jeffrey Burdges, Gregory Cherlin, and {\'E}ric Jaligot.
\newblock Minimal connected simple groups of finite {M}orley rank with strongly
  embedded subgroups.
\newblock {\em J. Algebra}, 314(2):581--612, 2007.

\bibitem{DGroupes}
Adrien Deloro.
\newblock {\em Groupes simples connexes minimaux de type impair}.
\newblock PhD thesis, Universit\'{e} Paris 7, Paris, 2007.

\bibitem{DGroupes1}
Adrien Deloro.
\newblock Groupes simples connexes minimaux alg\'ebriques de type impair.
\newblock {\em J. Algebra}, 317(2):877--923, 2007.

\bibitem{DGroupes2}
Adrien Deloro.
\newblock Groupes simples connexes minimaux non-alg\'ebriques de type impair.
\newblock {\em J. Algebra}, 319(4):1636--1684, 2008.

\bibitem{DJGroups}
Adrien Deloro and {\'E}ric Jaligot.
\newblock Groups of finite {M}orley rank with solvable local subgroups.
\newblock {\em Comm. Algebra}, 40(3):1019--1068, 2012.

\bibitem{DJSmall}
Adrien Deloro and {\'E}ric Jaligot.
\newblock Small groups of finite {M}orley rank with involutions.
\newblock {\em J. Reine Angew. Math.}, 644:23--45, 2010.

\bibitem{DJLie}
Adrien Deloro and {\'E}ric Jaligot.
\newblock Lie rank in groups of finite {M}orley rank with solvable local
  subgroups.
\newblock {\em J. Algebra}, 395:82--95, 2013.

\bibitem{DJ4alpha}
Adrien Deloro and {\'E}ric Jaligot.
\newblock Groups of finite {M}orley rank with solvable local subgroups and of
  odd type.
\newblock Preprint, 2008.

\bibitem{BCDAutomorphisms}
Jeffrey Burdges, Gregory Cherlin, and Adrian~M. DeLoro.
\newblock Automorphisms of minimal simple groups of degenerate type.
\newblock Preprint, 2009.
\end{thebibliography}
\end{document}